\documentclass[reqno,12pt,twoside]{amsart} 

\usepackage{amsmath,amsthm,amssymb,amsfonts,mathrsfs,color}
\usepackage{latexsym,esint}
\usepackage{mathrsfs}  
\usepackage[colorinlistoftodos,prependcaption,textsize=tiny]{todonotes}
\usepackage{hyperref}
\usepackage{enumitem}
\usepackage{esint}
\usepackage{mathtools}
\usepackage{pgf,tikz}
\usepackage{thmtools}
\usepackage{tkz-euclide}
\usetikzlibrary{arrows,patterns,calc,babel}

\newtheorem{theorem}{Theorem}[section]
\newtheorem{proposition}{Proposition}[section]

\newtheorem{lemma}{Lemma}[section]
\newtheorem{corollary}{Corollary}[section]
\newtheorem{definition}{Definition}[section]

\numberwithin{equation}{section} \numberwithin{theorem}{section}
\numberwithin{proposition}{section} \numberwithin{lemma}{section}
\numberwithin{corollary}{section}
\numberwithin{definition}{section} \numberwithin{remark}{section}

\newcommand{\C}{{\mathcal C}}

\def\Xint#1{\mathchoice
   {\XXint\displaystyle\textstyle{#1}}%
   {\XXint\textstyle\scriptstyle{#1}}%
   {\XXint\scriptstyle\scriptscriptstyle{#1}}%
   {\XXint\scriptscriptstyle\scriptscriptstyle{#1}}%
   \!\int}
\def\XXint#1#2#3{{\setbox0=\hbox{$#1{#2#3}{\int}$}
     \vcenter{\hbox{$#2#3$}}\kern-.5\wd0}}

\def\dashint{\Xint-}

\author[M. Rakovsky]{ Martin Rakovsky}
\address[M. Rakovsky]{Universit\'e Paris-Saclay, CNRS,  Laboratoire de math\'ematiques d'Orsay, 91405, Orsay, France.}
\email{martin.rakovsky@ens-paris-saclay.fr}

\title[]{Existence, uniqueness and regularity of solutions to the parabolic Ambrosio-Tortorelli system}

\begin{document}

\begin{abstract}
%We analyze the initial-boundary value problem corresponding to the gradient flow of the Ambrosio–Tortorelli functional, establishing existence, uniqueness, and regularity results in arbitrary dimension.
We investigate the existence, uniqueness, and regularity of the gradient flow of the Ambrosio–Tortorelli functional, viewed as an initial-boundary value problem, in arbitrary dimension.
%We investigate the existence, uniqueness, and regularity of the initial-boundary gradient flow $(u,v)$ of the Ambrosio-Tortorelli functional in arbitrary dimension. 
For any initial data, using a time-discrete Euler scheme, we establish the existence of a weak gradient flow satisfying a maximum principle. We also identify a functional space in which uniqueness holds. We further show that such gradient flow is smooth in the interior of the space-time domain. Under additional assumptions on the initial data, the regularity of the boundary of the domain, we prove optimal regularity for the solution, up to the space-time boundary.
\end{abstract}

\maketitle

%\tableofcontents

\section{Introduction}

The Mumford--Shah functional was introduced by D.~Mumford and J.~Shah in 1989 in their celebrated paper \cite{MS89}, in the context of image segmentation. It was later given a meaningful interpretation in fracture mechanics to describe brittle fractures in elastic materials (see \cite{FM98}). If $\Omega$ is a bounded open subset of $\mathbb{R}^d$, it can be defined for functions in $SBV^2(\Omega)$, the space of Special Bounded Variation functions, defined by 

\begin{multline*}
SBV^2(\Omega) = \{u \in L^1(\Omega), Du = \nabla u \mathcal{L}^d + (u^+ -u^-)\nu_u \mathcal{H}^{d-1}_{\rvert J_u},\\
\nabla u \in L^2(\Omega; \mathbb{R}^d), \mathcal{H}^{d-1} (J_u) < +\infty \},
\end{multline*}
where $\mathcal{L}^d$ is the Lebesgue measure, $\mathcal{H}^{d-1}$ the $d-1$-Hausdorff measure, $\nabla u$ is the approximate gradient of $u$, the set $J_u$ is the jump set of $u$, the vector $\nu_u$ is the approximate normal to $J_u$ and $u^+$ and $u^-$ are the approximate one-sided limits across $J_u$. For any $u \in SBV^2(\Omega)$, the Mumford Shah functional is defined by 
\[
MS(u) = \int_\Omega |\nabla u|^2 \, dx +  \mathcal{H}^{d-1}(J_u).
\]
For an elastic material, the function $u$ represents the displacement and $J_u$ denotes a part of the crack set.
%Describing the time propagation of a brittle fracture which respects amounts to study local minimizers or critical points of the Mumford-Shah functional at each time, a notion that is not canonical considering the set of admissible cracks. 
Computing local minimizers of the Mumford--Shah functional numerically is challenging, as the geometric properties of the discontinuity set $J_u$ are part of the unknown.

A possible approach to the (local) minimization of $MS$ is to use a steepest descent method, which amounts to determining the gradient flow of $MS$. While defining a suitable notion of gradient for $MS$ is delicate, a generalized gradient flow can be constructed via time discretization using an implicit Euler scheme. Letting the time step tend to zero, any limit curve is called a minimizing movement for the Mumford--Shah functional (see \cite{DG93}, \cite{A95}, \cite{CD97}). 

\smallskip

The Mumford-Shah functional admits various variational regularization in $L^2(\Omega)$ using De Giorgi’s $\Gamma$-convergence (see \cite{AT90}, \cite{AT92}, \cite{G98}).
%, where this approximation is understood in the sense of the Gamma-convergence developed by De Giorgi. 
The most common regularization of the Mumford-Shah functional is probably the Ambrosio and Tortorelli functional, proposed in \cite{AT92} and defined for $(u,v) \in (H^1(\Omega) \times [H^1(\Omega)\cap L^\infty(\Omega)])^2$ by 
\begin{equation}\label{AT}
AT_\varepsilon (u,v) = \int_\Omega (\eta_\varepsilon + v^2) |\nabla u|^2 dx + \int_\Omega \left(\varepsilon |\nabla v|^2 + \frac{(1-v)^2}{\varepsilon}\right)dx
\end{equation}
where $0 < \eta_\varepsilon \ll \varepsilon$ is a small parameter ensuring ellipticity. This approximation follows the spirit of the Allen--Cahn model for phase transitions (\cite{M87}, \cite{S88}) and forms the basis for many numerical simulations (see \cite{BFM08}).
The phase-field variable $v$ can be interpreted mechanically as a damage variable taking values in $[0,1]$, where $\{v=1\}$ corresponds to a sane material and $\{v=0\}$ to a fully damaged state. The existence of minimizers for $AT_\varepsilon$ follows from the direct method in the calculus of variations. Using the $\Gamma$-convergence result of \cite{AT92}, one can approximate a global minimizer of $MS$ by considering a convergent sequence $(u_\varepsilon, v_\varepsilon)$ of minimizers of $AT_\varepsilon$.  However, this approach provides no information about local minimizers. Extensions of the Gamma-convergence result to general critical points of $AT_\varepsilon$ can be found in \cite{FLS09}, \cite{BMRa23}, \cite{BRR26}. 

\medskip

Similarly to \cite{G98}, where a gradient flow for the one-dimensional Mumford--Shah functional was obtained via a non local approximation, the approach proposed in \cite{FP04} consists in deriving a gradient flow for $MS$ as the limit, as $\varepsilon \to 0$, of the gradient flows of $AT_\varepsilon$. The gradient flow of $AT_\varepsilon$ with respect to $L^2$ is defined as a pair $(u,v)$ in $\big[L^\infty((0,+\infty); L^2(\Omega)) \cap L^2((0,+\infty); H^1(\Omega))\big]^2$
satisfying
\begin{equation}\label{systeme départ intro}
\left\{
\begin{array}{ll}
\partial_t u - \mathrm{div}\big((\eta_\varepsilon + v^2) \nabla u\big) = 0 & \text{in } (0,+\infty) \times \Omega,\\
\partial_t v - \varepsilon \Delta v + \dfrac{v-1}{4\varepsilon} + v |\nabla u|^2 = 0 & \text{in } (0,+\infty) \times \Omega,\\
(u,v) = (g,1) & \text{on } (0,+\infty) \times \partial\Omega,\\
(u(0,\cdot), v(0,\cdot)) = (u_0, v_0) & \text{in } \Omega.
\end{array}
\right.
\end{equation}
The same approach was used in \cite{BM14} to obtain a unilateral gradient flow for $MS$, incorporating an irreversibility condition on the crack, which can only grow in time. 
%This irreversibility constraint is also addressed in \cite{G98}. 
Consequently, studying the $L^2$ gradient flow of $AT_\varepsilon$ becomes a center of interest in order to understand the gradient flow of $MS$.

\smallskip

The existence of weak solutions to the gradient flow system \eqref{systeme départ intro} was investigated in \cite{FP04}, where a uniqueness class for strong solutions was also identified, in dimensions $2$ and $3$. The existence of a strong gradient flow in dimension $2$ was studied in \cite{Cortopassi_2024} using a Galerkin spatial discretization. 
The main difficulty in proving the existence of weak solutions for the gradient flow of $AT_\varepsilon$ lies in the presence of a quadratic gradient term in the nonlinear part, which makes standard compactness arguments inefficient (see \cite{E90}). In the Galerkin method, the space of test functions allowed in the Galerkin method is dense in $D(\Omega)$ only for the $H^1$ topology, which is not enough to pass to the limit the term $v|\nabla u|^2$ without stronger estimates on $|\nabla u|^2$. The issue raised by the presence of a quadratic gradient is also discussed in \cite{X94} in the context of flows of incompressible fluids with temperature-dependent viscosity, and is overcome by using a time-discretization approach. We adopt the same approach in this work.  

\smallskip

The time-discretization method was also used in \cite{BM14} to construct a gradient flow for the variable $u$, coupled with an obstacle condition on the damage variable $v$ that enforces irreversibility of the evolution (see also \cite{G05}). 
%This approach is particularly well suited for numerical implementations, as it bypasses the lack of convexity of the biquadratic nonconvex term $v^2 |\nabla u|^2$, which otherwise prevents uniqueness in the standard minimization procedure.

\smallskip

In this paper, we prove the existence of a weak gradient flow for $AT_\varepsilon$ with initial--boundary conditions in any dimension. The existence is established in \autoref{Section existence faible}. In view of applications to the study of a gradient flow for the Mumford--Shah functional, we also derive a uniform energy bound and a maximum principle satisfied by the constructed weak solution $(u,v)$. We further provide, in \autoref{Section unicité}, a uniqueness condition for a ``strong '' gradient flow, which depends on the dimension. Under suitable boundary conditions, the solutions of the weak gradient flow appear to be regular up to the boundary, as shown in \autoref{section régularité}.

\smallskip

The optimal regularity of the elliptic counterpart of \eqref{systeme départ intro} was established in \cite{BMRa23}, relying on fine elliptic estimates involving Morrey--Campanato spaces, with the same method as \cite{R93} for harmonic maps with value in a revolution torus. To adapt this approach to the parabolic setting, we use arguments from \cite{L96}, completed by the method of \cite{Yin_1997}. The main argument is to prove Hölder regularity for the pair $(u,v)$. De Giorgi-Nash-Moser theorem provides the Hölder regularity for the function $u$. Using the first equation of \eqref{systeme départ intro}, we deduce a Campanato--Morrey estimate on $\nabla u$ (see \autoref{prop:grad_spat_Morrey}). The second equation of \eqref{systeme départ intro} becomes a heat equation with a source term in a Morrey space. In the elliptic case, this part is a direct consequence of De Giorgi-Nash Moser theorem. To prove that $v$ is Hölder, according to a parabolic Campanato lemma (see \cite[Lemma 4.3]{L96}), it is sufficient to prove a decay estimate on the $L^2$ norm of $v$ minus its average on parabolic cylinders. To obtain this estimate, we consider $w$ a ``heat ''-extension of $v$. While $w$ satisfies the desired estimate, the distance of the error $v-w$ to its mean value can be controled using a Poincare-type inequality for solutions of a heat equation with a source term in a Morrey space (see \autoref{prop : Poincare}). 
%, we are able to connect the Hölder regularity of $u$ with a  In particular, adapting the elliptic strategy requires a Caccioppoli-type inequality controlling the decay of the spatial gradient $\nabla u$, as well as a Poincaré-type inequality controlling the distance of $u$ to its parabolic mean by the $L^2$-norm of its gradient. 
The latter does not hold in general for functions in $L^2((0,+\infty); H^1(\Omega))$ and relies on the shape of the equation satisfied by the pair $(u,v)$.

\smallskip

While the method of \cite{X94} relies strongly on the specific structure of the system considered, the method presented in \autoref{section régularité} can be adapted to other parabolic system with a source term in $L^1$. In particular, a similar regularity result can be expected for the gradient flow of the Ambrosio-Tortorelli functional with a fidelity term, which is the functional studied in \cite{Cortopassi_2024} in the special case of a time-independent fidelity term. Another type of system in which this approach could provide regularity is the time evolution for a harmonic map flow with value in a revolution torus, for which the regularity in the elliptic case is proved in \cite{R93}.

\smallskip   

We close this introduction with some notations and definitions used in the rest of the paper. 

%The $\Gamma$-convergence approach developed by De Giorgi enables to get around this difficulty by considering a variational approximation of the energy. 

%The numerical implementation of this minimization method still requires some extra care, as the lack of convexity of the term $v^2|\nabla u|^2$ makes the classical gradient method fail to converge to a global or even a local minimizer. Another approach to the minimization of $MS$ would be to 

\medskip

\textbf{Standard notations.}
In the rest of the article, $d\geqslant 1$ is a positive integer. Let $B_R(x_0)$ be the ball of radius $R$ centered at $x$. We denote by $\omega_d$ the Lebesgue measure of $B_1(0)$. For any open subset $\Omega\in \mathbb{R}^d$ and $p\geqslant 1$, for a measurable function $f: \Omega \rightarrow \mathbb{R}$, we denote by 

\[\|f\|_{\infty, \Omega} = \sup_{x\in \Omega} |f(x)|, \quad \|f\|_p = \left(\int_{\Omega} |f|^p\right)^{1/p}, \quad \|f\|_{H^1} = \left(\int_{\Omega} |f(x)|^2 + \int_{\Omega} |\nabla f|^2\right)^{1/2}.\]

We also use the notation $\lfloor x\rfloor$ for the largest integer smaller than $x$ and $\displaystyle\binom{n}{a,b,c} = \dfrac{n!}{a!b!c!}$ for the multinomial coefficient. 

\medskip

\textbf{Parabolic framework.}
We introduce some tools from the theory of parabolic equations. To take into account the specificities of the parabolic framework, we introduce parabolic cylinders, which are analogous to open balls in the elliptic framework: if $X_0=  (t_0, x_0) \in \mathbb{R} \times \mathbb{R}^d$ and $R >0$, the parabolic cylinder $Q_R(X_0)$ or $Q_R(t_0, x_0)$ is defined by

\[Q_R(X_0) = (t_0 -R^2, t_0] \times B_R(x_0) \subset \mathbb{R}\times \mathbb{R}^d.\]
While $Q_R(X_0)$ is not an open set, it is the convenient set that fits the translation invariance properties of the parabolic equations. For any interval $I \subset \mathbb{R}$ and $\Omega \subset \mathbb{R}^d$, we call the parabolic boundary the set 

\[\partial_{\sqcup} (I \times \Omega) := (I \times \partial \Omega) \cup (\inf I \times \Omega).\]
In particular, $\partial_\sqcup Q_R(t_0,x_0) =((t_0-R^2, t_0) \times \partial B_r(x_0))\cap (\{t_0 -R^2\} \times B_R(x_0))$.   

\smallskip

The parabolic distance $d_p$ on $\mathbb{R}\times \mathbb{R}^d$ is defined by

\[d_p((t,x) ; (s, y)) = |t-s|^{1/2} + |x-y| \quad  \text{for all } (t,x),(s,y) \in \mathbb{R}\times \mathbb{R}^d.\]
The regularity of solutions of parabolic equations is studied in terms of the oscillation of $u$ in parabolic cylinders and is naturally described in terms of parabolic Hölder spaces. For any real number $\alpha\in (0,1)$, the Hölder space $\mathcal{C}_{par}^{\alpha/2, \alpha}(I\times \Omega)$ is defined as 

\[\mathcal{C}_{par}^{\alpha/2, \alpha}(I\times \Omega) = \left\{ f \in \mathcal{C}^0 (I\times \Omega), \sup_{X\neq Y \in I\times \Omega} \frac{|f(X) -f(Y)|}{d_p(X,Y)^\alpha} < \infty\right\}.\]
The above quantity defines a semi-norm for $f$, also denoted by

\[[f]_{\alpha , I\times \Omega} = \sup_{X\neq Y \in I\times \Omega} \frac{|f(X) -f(Y)|}{d_p(X,Y)^\alpha},\]
so that the natural norm associated to $C_{par}^{\alpha/2, \alpha}(I\times \Omega)$ is denoted by 

\[\|.\|_{\alpha} = \|.\|_{\infty , I\times \Omega} + [.]_{\alpha , I\times \Omega}.\]
The notion of parabolic space can be extended to parameters larger than $1$. For any nonintegral positive number $\ell$, we define the space $\mathcal{C}_{par}^{\ell/2, \ell}(I \times \Omega)$ as the space of continuous functions which admit derivatives $\partial_t^s \partial_x^r$ for any integers $s$ and $r$ such that $2r+s<\ell$. More precisely,  

\begin{equation}\label{def : Ckl}
\mathcal{C}_{par}^{\ell/2, \ell}(I \times \Omega) = \left\{f \in \mathcal{C}(I \times \Omega),  \sum_{2r+s \leqslant \lfloor \ell\rfloor} \|\partial_t^r \partial_x^s f\|_{\infty, I \times \Omega} + \sum_{2r+s = \lfloor \ell\rfloor} [\partial_t^r \partial_x^s f]_{\ell - \lfloor \ell \rfloor, I\times \Omega}  < +\infty \right\}.
\end{equation}  
Observe that a function $f\in \mathcal{C}_{par}^{\alpha/2, \alpha}((0,T)\times \Omega)$ is $\alpha/2-$Hölder with respect to the time variable and $\alpha-$Hölder with respect to the space variable, following the idea that a solution to a parabolic solution is twice as regular in the space variable than in the time variable.  

\smallskip

Finally, for a parabolic cylinder $Q_R(t_0,x_0)$ and a function $f\in L^1(Q_R(t_0,x_0))$, the average of $f$ on $Q_R(t_0,x_0)$ is denoted by 

\[\{f\}_{R, t_0,x_0} := \frac{1}{|Q_R(t_0,x_0)|} \iint_{Q_R(t_0,x_0)} f = \frac{1}{\omega_d r^{d+2}} \iint_{Q_R(t_0,x_0)} f.\]

\medskip

%\textbf{Weak solutions of a parabolic equation.}
%We specify the definition of a weak solution of a parabolic equation. We use the definition of \cite[Definition 8]{I26}. 
%
%\begin{definition}\label{def : sol faible parabolique}
%Let $I$ be an interval and $\Omega\subset \mathbb{R}^d$ be an open set. Let $X$ be a metric space (either $L^p(\Omega)$ or $H^{-1}(\Omega)$) and let $X_0^\star = \overline{\mathcal{C}_c^{\infty}}^{\|.\|_{X^\star}}$. Let $g \in L^2(I ; X)$ and 
%let $A : \mathbb{R}^d \rightarrow \mathbb{M}_d (\mathbb{R})$ be a bounded measurable function.  
%
%\smallskip
%
%A function $f$ is a \textit{weak} solution solution of the equation $\partial_t f - \text{div}(A \nabla f) = g$ if 
%
%\[f \in \mathcal{C}(I ; L^2(\Omega)), \quad \nabla f \in L^2(I \times \Omega), \quad \partial_t f \in L^2(I ; H^{-1}(\Omega)),\]
%
%and for every function $\varphi \in H^1(I \times \Omega) \cap L^2(I ; X_0^\star)$, 
%
%\[\int_I \int_\Omega f(t,x) \partial_t \varphi(t,x)dxdt + \int_I \int_\Omega A(x) \nabla f (t,x) \cdot \nabla \varphi(t,x) dxdt =  \int_I \langle g(t) ,\varphi(t) \rangle_{X_0^\star, X} dt.\] 
%\end{definition}
%
%For a detailed review of parabolic equations, we refer to \cite{L96}. 
%
%\medskip

\textbf{Parabolic Morrey spaces.}
Given $I \subset \mathbb{R}$ an interval and $\Omega \subset \mathbb{R}^d$, $p\geqslant 1$ and $\alpha \in (0,1)$, we define the space of Morrey functions with parameters $p, \alpha$ by 

\[\mathcal{L}^{p,d +\alpha} (I \times \Omega) = \left\{ f\in L^p(I ; L^p(\Omega)) ; \sup_{Q_\rho(t_0,x_0) \subset I \times \Omega} \frac{1}{\rho^{d+\alpha}}\iint_{Q_\rho(t_0,x_0)} |f(t,x)|^p dx\, dt < +\infty\right\}.\]
The space $\mathcal{L}^{p,d +\alpha} (I \times \Omega)$ can be endowed with the norm 

\[\|f\|_{\mathcal{L}^{p,d+\alpha}} = \sup_{Q_\rho(t_0,x_0) \subset I \times \Omega} \left(\frac{1}{\rho^{d+\alpha}}\iint_{Q_\rho(t_0,x_0)} |f(t,x)|^p dx\, dt\right)^{1/p}.\]

%\cite{FP04}, \cite{Cortopassi_2024}, \cite{DB93}, \cite{G18}, \cite{E10}, gobino
%
%Pourquoi Scheven ? Unicité ?  
%Preuve de l'unicité, truc bizarre. 
 
\section{Statement of the results}

In the sequel, we consider an integer $d\geqslant 1$ and $\Omega \subset \mathbb{R}^d$ a bounded open domain with a Lipschitz boundary $\partial\Omega$. The sequence $(\eta_\varepsilon)_{\varepsilon >0}$ is a sequence of positive numbers such that $\eta_\varepsilon \ll\varepsilon$.  
The Ambrosio-Tortorrelli functional is defined as \eqref{AT}.
Let $u_0, v_0 ,g$ be three functions satisfying the following conditions: 

\medskip

\begin{enumerate}[label=(H.\arabic*)]
\item $(u_0,v_0) \in [H^1(\Omega) \cap L^\infty(\Omega)]^2$ and $g\in H_{loc}^1 ([0,\infty) ; H^1(\Omega))\cap L^\infty([0,\infty)\times \Omega)$,
\item $g(0,\cdot) = u_0$ and $0 \leqslant v_0 \leqslant 1$ almost everywhere on $\Omega$,

\item $(u_0 -g(0,\cdot), v_0 -1) \in H_0^1 (\Omega)$,  
\item $g(t)\in L^\infty(\partial\Omega) \cap H^{1/2}(\partial\Omega)$ for almost every $t\in (0,+\infty)$.
\end{enumerate}

\medskip
% $ be a pair such that $0 \leqslant v_0 \leqslant 1$, and let $g\in H^1((0,\infty) ; H^1(\Omega))$ be such that $g(0,\cdot) = u_0$ on $\overline{\Omega}$. We also suppose that $g(t)\in L^\infty(\partial\Omega) \cap H^{1/2}(\partial\Omega)$ for any $t\in (0,+\infty)$. 
The system of equations corresponding to the gradient flow of the functional $AT_\varepsilon$ writes as : 

\begin{equation}\label{système de départ}
\left\{\begin{array}{ll}
\partial_t u - \text{div} ((\eta_\varepsilon + v^2) \nabla u) =0 & \text{on} \, (0,+\infty) \times \Omega\\
\partial_t v -  \varepsilon \Delta v + \frac{v-1}{4\varepsilon} + v |\nabla u|^2 =0 &\text{on} \, (0,+\infty) \times \Omega\\
(u,v) =(g,1) &\text{on} \, (0,+\infty) \times \partial\Omega\\
(u(0,\cdot ), v(0, \cdot)) = (v_0, u_0)& \text{on} \, \Omega\\
\end{array}
\right.
\end{equation}

We define the following notion of weak solution for the system \eqref{système de départ} : 

\begin{definition}\label{solution faible}
Let $u_0,v_0,g$ be functions satisfying conditions $(H.1), (H.2), (H.3)$ and $(H.4)$. Denote by 
\begin{equation}\label{notation : F}
\mathcal{F} = [L_{loc}^2([0,+\infty);H^1(\Omega))\cap H_{loc}^1([0,+\infty) ; L^2(\Omega))\cap L^\infty([0,+\infty)\times \Omega)]^2
\end{equation}
A pair $(u,v)\in \mathcal{F}$ is a weak solution of \eqref{système de départ} if for every every $\varphi\in \mathcal{C}_c^\infty ((0,+\infty)\times \Omega)$, 

\begin{equation}\label{eq : eq in u}
\int_0^{+\infty} \int_\Omega \left[\partial_t u(t,x) \varphi(t,x) + (\eta_\varepsilon + v^2) \nabla u (t,x) \nabla \varphi(t,x)\right] dxdt = 0,
\end{equation}
and 
\begin{multline}\label{eq : eq in v}
\int_0^{+\infty} \int_\Omega \left[\partial_t v(t,x) \varphi(t,x) +  \varepsilon \nabla v(t,x)\cdot  \nabla \varphi(t,x) +\frac{v(t,x) - 1}{4\varepsilon} \varphi(t,x)\right.\\
+ v(t,x) |\nabla u(t,x)|^2 \varphi(t,x)\bigg] dxdt =0. 
\end{multline}
\end{definition}

Our first main statement, which is proved in \autoref{Section existence faible}, is the following : 

\begin{theorem}\label{Theorem 1}
%Let $(u_0, v_0) \in [H^1(\Omega) \cap L^\infty(\Omega)]^2$ be a pair of functions such that $0\leqslant v_0 \leqslant 1$ on $\Omega$. Let $g\in H_{loc}^2((0,\infty) ; H^1(\Omega))$ be such that $g(0,\cdot) = u_0$ on $\overline{\Omega}$.
Let $u_0,v_0,g$ be functions satisfying conditions $(H.1), (H.2), (H.3)$ and $(H.4)$. 
%We also suppose that $g(t)\in L^\infty(\partial\Omega) \cap H^{1/2}(\partial\Omega)$ for any $t\in (0,+\infty)$.

\smallskip

The system \eqref{système de départ} admits a weak solution $(u,v)$ in the sense of \autoref{solution faible}. Moreover, the pair $(u,v)$ satisfies the following uniform energy bound for every $t >0$, 

\begin{multline}\label{energy bound thm 1}
AT_\varepsilon (u(t), v(t)) + \int_0^t \left(\|\partial_t u(s)\|_2^2 +\|\partial_t v(s)\|_2^2\right)ds\\
\leqslant AT_\varepsilon (u_0, v_0) +\int_0^t \|\partial_t g (s)\|_2^2 ds + \int_0^t \int_\Omega (\eta_\varepsilon + v(s)^2) |\partial_t \nabla g(s)|^2 dx ds\\
 + 2 \int_0^t \int_\Omega (\eta_\varepsilon + v(s)^2) \nabla u(s) \cdot \nabla (\partial_t g(s))dx ds,
\end{multline}
and it satisfies the maximum principle 

\begin{equation}\label{eq : principe maximum}
0\leqslant v\leqslant 1, \quad |u| \leqslant \max(\|g\|_{\infty, [0,+\infty) \times \mathbb{R}^d}, \|u_0\|_{\infty , \Omega}) \, \text{a.e on } (0,+\infty)\times \Omega .
\end{equation}
\end{theorem}

To our knowledge, this notion of weak solution is not sufficient to grant the uniqueness of the solution of \eqref{système de départ} (see \cite[Remark 2.3]{FP04}). However, inspired by the proof of \cite[Theorem 2.3, Step 3]{FP04}, we obtain the following notion of uniqueness for the system \eqref{système de départ}:

\begin{theorem}\label{Theorem 3}
Let $u_0,v_0,g$ be functions satisfying conditions $(H.1), (H.2), (H.3)$ and $(H.4)$. 
Suppose that the boundary $\partial \Omega$ is of class $\mathcal{C}^2$.

\smallskip

Let $(u_1, v_1)$ and $(u_2, v_2)$ be weak solutions of \eqref{système de départ} in the sense of \autoref{solution faible}. Suppose that $\nabla u_2, \nabla u_1 \in L^{d+2}((0,T); L^{d+2}(\Omega))$ for every $T>0$.   
Then $(u_1 , v_1) = (u_2, v_2)$.
%Then $(u,v) \in \mathcal{C}^\infty ((0,+\infty) \times \overline{\Omega})$. 
\end{theorem}

%This statement could be re-interpreted as this : if \eqref{système de départ} admits a "strong" solution, than it admits a unique weak solution.  

\smallskip

Our final statement tackles the regularity of the weak solution $(u,v)$. We prove that the pair is smooth in $(0,+\infty) \times\Omega$. The regularity up to the boundary requires additional properties on $\partial\Omega$ and on the initial and boundary conditions $(u_0,v_0)$ and $g$, especially compatibility conditions on the derivatives of $g$ and $(u_0,v_0)$ at $\{0\}\times \partial\Omega$. To understand the necessary conditions, we can look briefly at the expected compatibility condition satisfied by a pair $(u,v)$ in $\mathcal{C}_{par}^{1+\alpha/2,2+\alpha}([0,+\infty)\times \overline{\Omega})$ with $\alpha \in (0,1)$, focusing on the equation in $u$. If $(u,v)\in \mathcal{C}_{par}^{1+\alpha/2,2+\alpha}([0,+\infty)\times \overline{\Omega})$, then we must have, by the time continuity of $v, \nabla v, \nabla u$ and $\Delta u$, 
\[\partial_t u(0,x) = \lim\limits_{t\rightarrow 0} \text{div}((\eta_\varepsilon +v^2(t,x)) \nabla u(t,x))= \text{div}((\eta_\varepsilon +v_0(x)^2)\nabla u_0(x))\quad \text{for any } x\in \Omega,\]
and this equality must then be satisfied for any $x\in \partial\Omega$, yielding the condition
%
% $\left.\partial_t u(t)\right|_{x\in \partial \Omega} =\left. \partial_t g(t)\right|_{x\in \partial \Omega}$ for any $t >0$ and, passing this equality to the limit:
%\[\left.\partial_t g(0,x)\right|_{x\in \partial \Omega} = \lim\limits_{t\rightarrow 0} \partial_t \left(\left. u\right|_{x\in \partial \Omega}\right) =\left. \left[\lim\limits_{t\rightarrow 0} \text{div}((\eta_\varepsilon +v^2(t,x)) \nabla u(t,x))\right]\right|_{x\in \partial \Omega},\]
%which results to the condition 
\[\partial_t g(0,x) = \text{div}((\eta_\varepsilon +v_0(x)^2)\nabla u_0(x))\quad \text{for any } x\in \partial\Omega.\]
To obtain the compatibility conditions of superior order, we differentiate with respect to $t$ in the equations of \eqref{systeme départ intro} and use again that $\partial_t^k u(0,x) = \lim\limits_{t\rightarrow 0} \partial_t^k u(t,x)$, and a similar reasoning for the function $v$. The following definition describes the quantity that is expected to correspond to $(\partial_t^k u(0,\cdot), \partial_t^k v(0,\cdot))$ in the case where $(u,v) \in \mathcal{C}^{k+\alpha/2,2k+\alpha}([0,+\infty) \times \overline{\Omega})$ with $\alpha \in (0,1)$. 

%\[\lim\limits_{t\rightarrow 0} \text{div}((\eta_\varepsilon +v^2(t,x)) \nabla u(t,x)) = \text{div}((\eta_\varepsilon +v_0(x)^2)\nabla u_0(x)) \quad \text{for any } x\in \Omega.\]
%In particular, for any $x\in \partial\Omega$, we must have 
%
%\[\text{div}((\eta_\varepsilon +v_0(x)^2)\nabla u_0(x)) = \lim\limits_{t\rightarrow 0} \partial_t u 
%
%For the solution $u$ to be smooth up to points of $\{0\}\times \partial \Omega$, the equation \eqref{systeme départ intro} must remain true at the boundary if we extend $(\partial_t u, \partial_t v)$ by $(\partial_t g, 0)$ and the spatial derivatives of $(u,v)$ by the spatial derivatives of $(u_0,v_0)$. The following definition allows to state those conditions. 

\smallskip

\begin{definition}\label{def : compatibilité}
Let $k\geqslant 0$ be an integer and let $(u_0,v_0)\in [\mathcal{C}^{k+1,1}(\overline{\Omega})]^2$. We define recursively the sequence $(\partial_t^{(\ell)} u_0, \partial_t^{(\ell)} v_0)_{\ell \leqslant k} : \Omega \rightarrow \mathbb{R}$ by 

\[(\partial_t^{(0)} u_0, \partial_t^{(0)} v_0) = (u_0, v_0)\]

\[(\partial_t^{(1)} u_0, \partial_t^{(1)} v_0) = \left(\text{div}((\eta_\varepsilon+ v_0^2)\nabla u_0), \varepsilon \Delta v_0 + \dfrac{1-v_0}{4\varepsilon} - v_0 |\nabla u_0|^2\right)\]
and for every $1 \leqslant \ell \leqslant k-1$, 

\begin{multline*}
\partial_t^{(\ell+1)} u_0 = \sum_{a+b+c=\ell} \binom{k}{a,b,c} \partial_t^{(a)} v_0 \partial_t^{(b)} v_0 \Delta \partial_t^{(c)} u_0 \\
+ 2\sum_{a+b+c=\ell} \binom{k}{a,b,c} \partial_t^{(a)} v_0 \nabla \partial_t^{(b)} v_0 \cdot \nabla \partial_t^{(c)} u_0,
\end{multline*}

\[\partial_t^{(\ell +1)} v_0 = \varepsilon \Delta \partial_t^{(\ell)} v_0 - \frac{1}{4\varepsilon} \partial_t^{(\ell)} v_0
 - \sum_{a+b+c=\ell} \binom{k}{a,b,c} \partial_t^{(a)} v_0 \nabla \partial_t^{(b)} u_0 \cdot \nabla \partial_t^{(c)} u_0.\]
\end{definition}

Notice that if $(u, v)\in \mathcal{C}^{k+\alpha/2, 2k+\alpha}([0,+\infty)\times \overline{\Omega})$, for $\alpha \in (0,1)$, is a solution of \eqref{système de départ}, then one has the equality $(\partial_t^{(\ell)} u_0, \partial_t^{(\ell)} v_0) = (\partial_t^{(\ell)} u(0, \cdot) , \partial_t^{(\ell)} v(0, \cdot))$ on $\overline{\Omega}$. 

\smallskip
%Suppose that $(u_0,v_0) \in W^{k+2, \infty}$ for an integer $k\geqslant 0$. We  $(\partial_t^{(\ell)} u(0,x), \partial_t^{(\ell)} v(0,x))$ 

\begin{theorem}\label{Theorem 2}
Suppose that $\Omega$ has a Lipschitz boundary $\partial \Omega$ and let $u_0,v_0,g$ be functions satisfying conditions $(H.1), (H.2), (H.3)$ and $(H.4)$. 
Let $(u,v)$ be a weak solution of \eqref{système de départ} and suppose that it also satisfies \eqref{energy bound thm 1} and \eqref{eq : principe maximum}. Then $(u,v) \in [\mathcal{C}^{\infty}((0,\infty)\times \Omega)]^2$.

\smallskip

Moreover, suppose that there exists an integer $k\geqslant 1$ and a real number $\beta \in (0,1)$ satisfying the following conditions: 

\begin{itemize}
\item[(i)] $\partial\Omega \in \mathcal{C}^{k, \beta}$,
\item[(ii)] $(u_0, v_0) \in \mathcal{C}^{k+1,1}(\overline{\Omega})$,
\item[(iii)] $g\in \mathcal{C}_{par}^{\lfloor k/2\rfloor + \beta/2,2\lfloor k/2\rfloor +\beta}([0,+\infty) \times \Omega)$, 
\item[(iv)] for any integer $\ell \leqslant \lfloor k/2\rfloor +1$,

\[(\partial_t^{(\ell)} u_0 , \partial_t^{(\ell)} v_0) =  (\partial_t^{(\ell)} g(0,x), 0).\]
\end{itemize}

Then $(u,v) \in \mathcal{C}_{par}^{\lfloor k/2\rfloor + \beta/2, 2\lfloor k/2\rfloor + \beta }([0,+\infty) \times \overline{\Omega})$. 
%In particular, if $k\geqslant 2$, the system \eqref{système de départ} admits a unique solution. 
\end{theorem}

%Let $\Omega \subset \mathbb{R}^d$ be a bounded open set with smooth boundary. 
%Let $(u_0, v_0) \in H^1(U)^2$ such that $0 \leqslant v_0 \leqslant 1$. Let $g(t)$ be a smooth function.

Notice that the regularity stated in \autoref{Theorem 2} is non-trivial. Indeed, the second equation of \eqref{système de départ} is of the form
\[\partial_t v - \varepsilon \Delta v = f, \quad \text{with } \sup_{(0,+\infty)} \|f(t)\|_{1} < + \infty,\]
for which the standard theory does not apply. To prove the regularity for the solution of \eqref{système de départ}, we obtain parabolic Hölder regularity for the pair $(u,v)$ using arguments from the theory of parabolic differential equations (see \cite{L96}). The key point is to prove that $\nabla u$ belongs to a suitable Morrey space, which is achieved using a parabolic version of Caccioppoli' inequality, and to deduce that $v$ is Hölder, which is achieved using a Poincaré's inequality and the parabolic Campanato lemma (see \cite[Lemma 4.3]{L96}).  

\smallskip

In the sequel, we adopt the following notation given a function $g \in L^\infty(\partial \Omega)\cap H^{1/2}(\partial\Omega)$: 

\[\mathcal{A}_g = \{ (u,v) \in [H^1(\Omega) \cap L^\infty (\Omega)]^2, u_{|\partial\Omega} = g, v_{|\partial\Omega} =1\}.\]

A consequence of the direct method in calculus of variation is the following result:  

\begin{proposition}\label{prop : méthode directe}
Let $\delta >0$.
Let $(\tilde{u}, \tilde{v}) \in [H^1(\Omega)\cap L^\infty(\Omega)]^2$ with $0\leqslant \tilde{v} \leqslant 1$ a.e in $\Omega$ and $g \in L^\infty(\partial \Omega)\cap H^{1/2}(\partial\Omega)$. The problem 

\[\min\left\{AT_\varepsilon (u,v) + \frac{1}{\delta} \int_\Omega \left(|u-\tilde{u}|^2 +|v-\tilde{v}|^2\right), (u,v) \in \mathcal{A}_g\right\}\]
admits a solution $(u,v)$ which satisfies the following maximum principle : 

\[ |u| \leqslant \max(\|g\|_{\infty, \partial\Omega}, \|\tilde{u}\|_{\infty,\Omega}), \quad \quad 0\leqslant v\leqslant 1 \quad \text{a.e in }\Omega.\]
\end{proposition}

The proof of the maximum principle follows from a classical truncation argument.

\section{Existence of solutions}
\label{Section existence faible}

The object of \autoref{Section existence faible} is to prove \autoref{Theorem 1}, namely the existence of a weak solution to \eqref{système de départ} by using the finite difference approximation in time, in a similar fashion as in \cite{X94} and \cite{BM14}. The Euler scheme is defined as followed :

\begin{definition}\label{Eueler Scheme}
Let $\delta >0$ and $(u_0, v_0)\in (H^1(\Omega) \cap L^\infty(\Omega))^2$ be a pair such that $(u_0, v_0)= (g(0,\cdot ), 1)$ on $\partial\Omega$ and such that $0\leqslant v_0 \leqslant 1$. The sequence $((u_i^\delta, v_i^\delta))_{i\geqslant 0}$ is defined recursively by $(u_0^\delta, v_0^\delta) = (u_0, v_0)$ and for $i\geqslant 1$,

\[(u_{i+1}^\delta, v_{i+1}^\delta) \in \text{Argmin} \left\{AT_\varepsilon (u,v) + \frac{1}{\delta} \int_\Omega \left(|u-u_i^\delta|^2 +|v-v_i^\delta|^2\right), (u,v) \in \mathcal{A}_{g((i+1)\delta)}\right\}.\] 
\end{definition}

The pair $(\overline{u}^\delta(t), \overline{v}^\delta(t)) = \sum_{i=0}^\infty (u_i^\delta(t), v_i^\delta(t)) 1_{t\in [i\delta, (i+1)\delta)}(t)$ aims to be a time approximation of a weak solution to \eqref{solution faible}. 

\smallskip

Let us highlight that, as the functional is not strictly convex, the minimizing pair $(u_{i+1}^\delta, v_{i+1}^\delta)$ might not be unique, which explains why the functional space $\mathcal{F}$ is not a uniqueness class for \autoref{solution faible}. The lack of uniqueness is also an issue for the numerical approximation, as the classic gradient scheme might not converge to the minimizer $(u_i^\delta, v_i^\delta)$. See \cite{BM14} for an alternative algorithm used for numerical experiments.

\medskip

Observe that this definition implies that the pair $(u_i^\delta, v_i^\delta)$ is a weak distributional solution of the following Euler-Lagrange equation for any index $i\geqslant 1$ : 

\begin{equation}\label{EL approché discret}
\left\{\begin{array}{ll}
\displaystyle\frac{u_i^\delta - u_{i-1}^\delta}{\delta} - \text{div}((\eta_\varepsilon + |v_i^\delta|^2) \nabla u_i^\delta) = 0 & \text{in } \mathcal{D}'(\Omega),\\
\displaystyle\frac{v_i^\delta - v_{i-1}^\delta}{\delta} - \varepsilon\Delta v_i^\delta + \frac{v_i^\delta -1}{4\varepsilon} + v_i^\delta |\nabla u_i^\delta|^2 =0 & \text{in } \mathcal{D}'(\Omega),\\
(u_i^\delta - g(i\delta , \cdot) , v_i^\delta - 1) \in [H_0^1(\Omega)]^2\\
\end{array}\right.
\end{equation}

\subsection{Study of the sequence $(u_i , v_i)$}

In this subsection, we fix a positive number $\delta >0$ and consider the sequence $((u_i^\delta, v_i^\delta))_{i\geqslant 0}$ defined in \autoref{Eueler Scheme}. In order to pass to the limit in the sequence $((u_i^\delta, v_i^\delta))_{i\geqslant 0}$ as $\delta \rightarrow 0$, we need compactness properties, which are the issue of this subsection. More precisely, we provide a uniform energy bound on $(AT_\varepsilon (u_i^\delta, v_i^\delta))_{i\geqslant 0}$. We drop the subscript $\delta$ and write simply the sequence as $(u_i, v_i)$. 

We also define 

\begin{equation}\label{gi}
g_i(x) = g(x, i\delta) \quad \forall x\in \mathbb{R}^d.
\end{equation} 

%In the sequel, we fix an integer $N$ and we set $\delta = T/N$. We set $t_i = i \delta$. We construct iteratively the sequence $(u_i, v_i)$, initiated by $(u_0, v_0)$. 

A consequence of \autoref{prop : méthode directe} is the discrete maximum principle :

\begin{proposition}\label{Maximum principle}
For any index $i\geqslant 1$, the pair $(u_i, v_i)$ satisfies

\begin{equation}\label{maximum principle}
0\leqslant v_i \leqslant 1, \quad |u_i| \leqslant \max(\|g\|_{\infty, (0,\infty)\times \Omega}, \|u_0\|_{\infty,\Omega}). 
\end{equation}
\end{proposition} 

We now establish the discrete energy bound on $(AT_\varepsilon (u_i, v_i))_{i\geqslant 0}$. 

\begin{proposition}\label{prop : energy bound discret}
For any $N\geqslant 1$, 

\begin{multline}\label{energy bound discret}
AT_\varepsilon (u_N, v_N) + \sum_{i=1}^N \frac{1}{\delta} \int_\Omega \left(|u_i -u_{i-1}|^2 + |v_i -v_{i-1}|^2\right) \leqslant AT_\varepsilon (u_0, v_0)+ \sum_{i=1}^N \frac{1}{\delta} \int_\Omega |g_i -g_{i-1}|^2\\
+ \sum_{i=1}^N \left[\int_\Omega (\eta_\varepsilon + v_{i-1} ^2) (|\nabla (g_i - g_{i-1})|^2 + 2\nabla u_{i-1} \cdot \nabla (g_i - g_{i-1}))\right].
\end{multline}
\end{proposition}

\begin{proof}
%We start by comparing $AT_\varepsilon (u_i, v_i)$ with $AT_\varepsilon (u_{i-1}, v_{i-1})$ for any index $i\geqslant 1$. 
Take the pair $(u,v) = (u_{i-1} + g_i - g_{i-1} , v_{i-1})$, which is an element of $\mathcal{A}_{g_i}$, as a competitor. The construction of the pair $(u_i, v_i)$ in \autoref{Eueler Scheme} provides 
%
%\[\mathcal{AT}_{\varepsilon , \delta , u_{i-1}, v_{i-1}} (u_i ,v_i)\leqslant \mathcal{AT}_{\varepsilon , \delta , u_{i-1}, v_{i-1}} (u,v) \]
%
%which rewrites as 

\begin{multline}\label{bound discret 1}
AT_\varepsilon (u_i,  v_i) + \frac{1}{\delta} \int_\Omega \left(|u_i -u_{i-1}|^2 + |v_i -v_{i-1}|^2\right)\\
\leqslant AT_\varepsilon (u,  v) + \frac{1}{\delta} \int_\Omega \left(|u -u_{i-1}|^2 + |v -v_{i-1}|^2\right) 
=AT_\varepsilon (u_{i-1}, v_{i-1}) + \frac{1}{\delta} \int_\Omega |g_i -g_{i-1}|^2\\
+ \int_\Omega (\eta_\varepsilon + v_{i-1} ^2)\left(|\nabla (g_i - g_{i-1})|^2 + 2\nabla u_{i-1} \cdot \nabla (g_i - g_{i-1})\right).
\end{multline}
Summing \eqref{bound discret 1} for $i=1, \ldots , N$ and simplifying yields the desired inequality. 
\end{proof}

\subsection{Convergence with delta}

In the subsection, we let $\delta \rightarrow 0$ to obtain a solution of \eqref{système de départ}. 
We define the respectively piecewise constant and affine functions on $\mathbb{R}_+ \times \Omega$ : 

\[(\overline{u}_\delta (t), \overline{v}_\delta (t),\overline{g}_\delta (t) )= (u_i , v_i, g_i) \quad \quad  \text{for } t\in [i\delta , (i+1)\delta ),\]

\begin{multline*}
(\hat{u}_\delta (t) , \hat{v}_\delta (t), \hat{g}_\delta (t)) = \\
\left((u_i , v_i, g_i) + \dfrac{t-i\delta}{\delta} (u_{i+1} - u_i, v_{i+1}- v_i, g_{i+1}-g_i)\right) \quad \text{for }t\in [i\delta , (i+1)\delta ).
\end{multline*}
Note that $(\overline{u}_\delta, \overline{v}_\delta, \overline{g}_\delta) \in [L^\infty_{loc}([0,\infty); H^1(\Omega)]^3$ and $(\hat{u}_\delta, \hat{v}_\delta, \hat{g}_\delta) \in [H^1_{loc}([0,\infty);H^1(\Omega))]^3$. 
For any time $t>0$ let $t_\delta = \delta \left\lfloor\dfrac{t}{\delta}\right\rfloor$. The maximum principle \autoref{Maximum principle} rewrites as 
\begin{equation}\label{max prin delta}
0\leqslant \overline{v}_\delta, \hat{v}_\delta \leqslant 1, \quad |\overline{u}_\delta|, |\hat{u}_\delta| \leqslant \max(\|g\|_{\infty, (0,+\infty)\times \Omega}, \|u_0\|_{\infty,\Omega}), 
\end{equation}
while \autoref{prop : energy bound discret} rewrites as 

\begin{multline}\label{energy bound continue}
AT_\varepsilon (\overline{u}_\delta (t) , \overline{v}_\delta (t)) + \int_0^{t_\delta} \left(\|\partial_t \hat{u}_\delta(s)\|_2^2 + \|\partial_t \hat{v}_\delta(s)\|_2^2 \right)ds 
\leqslant AT_\varepsilon (u_0 , v_0) + \int_0^{t_\delta} \|\partial_t \hat{g}_\delta (s)\|_2^2 ds\\
+ \int_0^{t_\delta}\int_\Omega (\eta_\varepsilon + \overline{v}_\delta(s)^2)\left(|\partial_t \nabla \hat{g}_\delta (s)|^2 + 2 \nabla \overline{u}_\delta (s)\cdot \nabla (\partial_t \hat{g}_\delta (s)\right) dx ds.
\end{multline}
Considering the Euler--Lagrange equations \eqref{EL approché discret},  the pair $(u_i, v_i)$ satisfies for a.e $t >0$

\begin{equation}\label{EL approché t fixe}
\left\{\begin{array}{ll}
\partial_t \hat{u}_\delta (t) - \text{div}((\eta_\varepsilon + \overline{v}_\delta^2(t)) \nabla \overline{u}_\delta(t)) = 0 & \text{in } \mathcal{D}'(\Omega),\\
\partial_t \hat{v}_\delta (t) - \varepsilon\Delta \overline{v}_\delta(t) + \frac{\overline{v}_\delta(t) -1}{4\varepsilon} + \overline{v}_\delta(t) |\nabla \overline{u}_\delta(t)|^2 =0 & \text{in } \mathcal{D}'(\Omega),\\
(\overline{u}_\delta(t) - g_i(t), v_i^\delta(t) -1) \in H_0^1(\Omega). &\\
\end{array}\right.
\end{equation}
Let $\varphi\in \mathcal{D}((0,\infty)\times \Omega)$. For any $t >0$, the function $\varphi(t,\cdot)$ is in $\mathcal{D}(\Omega)$, so 
\begin{equation}\label{EL approché t fixe1}
\int_\Omega \partial_t \hat{u}_\delta (t,x)\varphi(t,x)dx +\int_\Omega (\eta_\varepsilon + \overline{v}_\delta^2(t,x)) \nabla \overline{u}_\delta(t,x) \cdot \nabla\varphi(t,x)dx = 0,
\end{equation}

\begin{multline}\label{EL approché t fixe2}
\int_\Omega \partial_t \hat{v}_\delta (t,x) \varphi(t,x)dx +\int_\Omega \bigg[\varepsilon\nabla \overline{v}_\delta(t,x) \cdot \nabla \varphi (t,x)\\
\left. + \frac{\overline{v}_\delta(t,x) -1}{4\varepsilon}\varphi(t,x) + \overline{v}_\delta(t,x) |\nabla \overline{u}_\delta(t,x)|^2 \varphi(t,x)\right]dx =0  
\end{multline}
Integrating \eqref{EL approché t fixe1} and \eqref{EL approché t fixe2} over $(0,\infty)$, we deduce that

\begin{equation}\label{EL approché t fixe3}
\int_0^\infty\int_\Omega \partial_t \hat{u}_\delta (t,x)\varphi(t,x)dxdt +\int_0^\infty\int_\Omega (\eta_\varepsilon + \overline{v}_\delta^2(t,x)) \nabla \overline{u}_\delta(t,x) \cdot \nabla\varphi(t,x)dxdt = 0,
\end{equation}

\begin{multline}\label{EL approché t fixe4}
\int_0^\infty\int_\Omega \partial_t \hat{v}_\delta (t,x) \varphi(t,x)dxdt +\int_0^\infty\int_\Omega \bigg[\varepsilon\nabla \overline{v}_\delta(t,x) \cdot \nabla \varphi (t,x) + \frac{\overline{v}_\delta(t,x) -1}{4\varepsilon}\varphi(t,x)\\
\left. + \overline{v}_\delta(t,x) |\nabla \overline{u}_\delta(t,x)|^2 \varphi(t,x)\right]dxdt =0.
\end{multline}

As \eqref{EL approché t fixe3} and \eqref{EL approché t fixe4} are true for any $\varphi \in \mathcal{D}((0,\infty)\times \Omega)$,
% & \text{on } \partial\Omega,\\
%\overline{v}_\delta = 1 & \text{on } \partial\Omega, \\
%(\overline{u}_\delta(0,\cdot), \overline{v}_\delta(0,\cdot))= (u_0, v_0) & \text{on } \Omega.
the space of admissible test functions $\mathcal{C}_c^\infty((0,+\infty) \times \Omega)$ can be extended by density to $L_{loc}^\infty((0,+\infty) \times \Omega) \cap H_{loc}^1((0,+\infty); L^2(\Omega)) \cap L^2((0,+\infty); H_0^1(\Omega))$. 

\smallskip

Let us first recall the following standard interpolation result for the sequence $(\overline{g}_\delta , \hat{g}_\delta)$. 

\begin{proposition}\label{prop : cv g}
Suppose that $g \in H^2_{loc}((0,+\infty); H^1(\Omega))$. Then 

\begin{enumerate}
\item $\overline{g}_\delta (t) \longrightarrow g(t)$ strongly in $H^1(\Omega)$ for almost every $t >0$,

\item $\overline{g}_\delta \longrightarrow g$ strongly in $L^2_{loc}((0,+\infty) ; H^1(\Omega))$

\item $\partial_t \hat{g}_\delta (t) \longrightarrow \partial_t g(t)$ strongly in $H^1(\Omega)$ for almost every $t >0$. 

\item $\partial_t \hat{g}_\delta \longrightarrow \partial_t g$ strongly in $L^2_{loc}((0,+\infty) ; H^1(\Omega))$, 

\end{enumerate}
\end{proposition}

We now prove uniform estimates for $(\overline{u}_\delta, \overline{v}_\delta)$ and $(\hat{u}_\delta, \hat{v}_\delta)$ derived from \eqref{energy bound continue}. 

\begin{corollary}\label{cor h1 bound}
For any $T >0$, there exists a constant $C_{g,T} >0$, independent of $\delta$, such that 

\begin{equation}\label{bound overline+hat}
\sup_{t\in [0,T]} \|\overline{u}_\delta (t)\|_{H^1(\Omega)}+ \sup_{t\in [0,T]} \|\overline{v}_\delta (t)\|_{H^1(\Omega)}+ \int_0^T \int_\Omega \left(\|\partial_t \hat{u}_\delta(s)\|_2^2  + \|\partial_t \hat{v}_\delta(s)\|_2^2\right) ds < C_{g,T}.
\end{equation}
%\begin{equation}\label{bound overline}
%\sup_{t\in [0,T]} \|\overline{u}_\delta (t)\|_{H^1(\Omega)}, \sup_{t\in [0,T]} \|\overline{v}_\delta (t)\|_{H^1(\Omega)} \leqslant C_{g,T}
%\end{equation}
%
%\begin{equation}\label{bound hat}
%\int_0^T \int_\Omega \left(\|\partial_t \hat{u}_\delta(s)\|_2^2  + \|\partial_t \hat{v}_\delta(s)\|_2^2\right) ds < C_{g,T}.
%\end{equation} 
\end{corollary}

\begin{proof}
Fix $t >0$ and set $t_\delta = \left\lceil \dfrac{t}{\delta}\right\rceil$. Using that $2ab\leqslant a^2 +b^2$ in the right-hand side of \eqref{energy bound continue}, 

\begin{align*}
&AT_\varepsilon (\overline{u}_\delta (t) , \overline{v}_\delta (t)) + \int_0^{t_\delta} \left(\|\partial_t \hat{u}_\delta (s)\|_2^2 + \|\partial_t \hat{v}_\delta (s)\|_2^2\right)dx ds \\
&\leqslant AT_\varepsilon (u_0 , v_0)\\
&\quad + \int_0^{t_\delta} \|\partial_t \hat{g}_\delta (s)\|_2^2ds + \int_0^{t_\delta}\int_\Omega (\eta_\varepsilon + \overline{v}_\delta(s)^2)\left(|\partial_t \nabla \hat{g}_\delta (s)|^2 + 2 \nabla \overline{u}_\delta (s)\cdot \nabla \partial_t \hat{g}_\delta (s)\right) dxds\\
&\leqslant AT_\varepsilon (u_0 , v_0) + \int_0^{t_\delta} \|\partial_t \hat{g}_\delta(s)\|_2^2 ds + 2(1+\eta_\varepsilon) \int_0^{t_\delta}\int_\Omega |\partial_t \nabla \hat{g}_\delta(s)|^2 dxds\\
&\quad +\int_0^{t_\delta} \int_\Omega (\eta_\varepsilon + \overline{v}_\delta^2(s)) |\nabla \overline{u}_\delta (s)|^2 dxds .
\end{align*}
%&\leqslant AT_\varepsilon (u_0, v_0) + \widetilde{\mathcal{C}_g} + \int_0^t AT_\varepsilon (\overline{u}_\delta (s), \overline{v}_\delta (s)).\\ 
From \autoref{prop : cv g}, the sequence $(\|\partial_t \hat{g}_\delta\|_{L^2((0,T);H^1(\Omega))})_\delta$ is uniformly bounded, so there exists a constant $C_{g,T}$ such that  

\begin{multline}\label{gron}
AT_\varepsilon (\overline{u}_\delta (t) , \overline{v}_\delta (t)) + \int_0^{t_\delta} \left(\|\partial_t \hat{u}_\delta (s)\|_2^2 + \|\partial_t \hat{v}_\delta (s)\|_2^2 \right)ds\\
\leqslant  AT_\varepsilon (u_0, v_0) + C_{g,T} + \int_0^{t_\delta} AT_\varepsilon (\overline{u}_\delta (s), \overline{v}_\delta (s))ds.\\ 
\end{multline}
Applying Gronwall's lemma in \eqref{gron} yields for all $t\in [0,T]$
\begin{equation}\label{gron2}
AT_\varepsilon (\overline{u}_\delta (t) , \overline{v}_\delta (t)) \leqslant (AT_\varepsilon (u_0, v_0) + C_{g,T}) e^T,
\end{equation} 
and injecting it in \eqref{gron} implies in turn for all $t\in [0,T]$
\begin{equation}\label{gron 3}
\int_0^{t_\delta} \left(\|\partial_t \hat{u}_\delta(s)\|_2^2 + \|\partial_t \hat{v}_\delta(s)\|_2^2\right)ds \leqslant  AT_\varepsilon (u_0, v_0) + C_{g,T} + \int_0^t AT_\varepsilon (\overline{u}_\delta (s), \overline{v}_\delta (s))ds \leqslant C_{g,T},
\end{equation}
which implies \eqref{bound overline+hat}. 
\end{proof}
We are now ready to pass to the limit in \eqref{EL approché t fixe3}, \eqref{EL approché t fixe4} and \eqref{energy bound continue}. 

\begin{proposition}\label{extraction}
There exists a pair $(u,v) \in \left(L_{loc}^\infty(\mathbb{R}_+; H^1(\Omega)) \cap H_{loc}^1(\mathbb{R}_+; L^2(\Omega))\right)^2$ and a (not relabeled) subsequence $(\overline{u}_\delta, \overline{v}_\delta, \hat{u}_\delta, \hat{v}_\delta)$ satisfying :  

\begin{enumerate}
\item $(\hat{u}_\delta , \hat{v}_\delta) \rightarrow (u,v) \, \text{strongly} \text{ in } (C^0([0,T];L^2(\Omega))^2$ for every $T >0$,

\item $(\partial_t \hat{u}_\delta , \partial_t \hat{v}_\delta) \rightharpoonup (\partial_t u,\partial_t v) \, \text{weakly} \text{ in } (L^2([0,T]; L^2(\Omega)))^2$,

\item $(\overline{u}_\delta, \overline{v}_\delta) \longrightarrow (u,v)$ strongly in $[L^2_{loc}((0,+\infty); L^2(\Omega))]^2$ and almost everywhere on $(0,+\infty)\times \Omega$,

\item $(\overline{u}_\delta (t), \overline{v}_\delta(t)) \rightharpoonup (u(t),v(t)) \, \text{weakly } \text{ in } [H^1(\Omega)]^2$ for almost every  $t >0$, 

\item $(\overline{u}_\delta , \overline{v}_\delta ) \rightharpoonup (u,v)) \, \text{weakly } \text{ in } [L^2_{loc}((0, +\infty) ;H^1(\Omega))]^2$, 

\item $(\overline{u}_\delta, \overline{v}_\delta) \overset{\star}{\rightharpoonup} (u, v)$ weakly-$\star$ in $[L^\infty_{loc}((0,+\infty) \times \Omega)]^2$,

\item $(u(t) - g(t), v(t) - 1) \in H_0^1 (\Omega)$ for almost every $t >0$,  
\end{enumerate}
\end{proposition}

\begin{proof}
\textbf{Step 1 :} There exists a pair $(u,v) \in C^0([0,+\infty);L^2(\Omega))$ and a converging subsequence  $(\hat{u}_\delta, \hat{v}_\delta)$ such that $(\hat{u}_\delta, \hat{v}_\delta) \rightarrow (u(t) , v(t))$ strongly in $(C^0([0,T];L^2(\Omega))^2$ for every $T >0$.  

\smallskip

Let us show that, for any $T >0$, the sequence $(\hat{u}_\delta, \hat{v}_\delta)_\delta$, satisfies the hypothesis of the Ascoli theorem. 

\smallskip

\begin{itemize}
\item[$\bullet$] For any fixed time $t$, the sets $\{\|\hat{u}_\delta (t)\|_2^2, \delta >0\}$ and $\{\|\hat{v}_\delta (t)\|_2^2, \delta >0\}$ are relatively compact in $L^2(\Omega)$ for the strong topology. Indeed, according to \autoref{Maximum principle},

\[\sup_{[0,T]} \|\hat{u}_\delta(t)\|_{\infty , \Omega}\leqslant \max (\|u_0\|_{\infty, \Omega} , \|g\|_{\infty, (0,+\infty)\times \Omega}), \quad \sup_{[0,T]} \|\hat{v}_\delta(t)\|_{\infty , \Omega} \leqslant 1,\]

and according to \autoref{cor h1 bound}, 

\begin{equation}\label{compacitestar}
\sup_{[0,T]} \|\nabla \hat{u}_\delta(t)\|_{L^2(\Omega)} \leqslant \sup_{[0,T]} \|\nabla \overline{u}_\delta(t)\|_{L^2(\Omega)} \leqslant C_{T,g},
\end{equation}

\[\sup_{[0,T]} \|\nabla \hat{v}_\delta(t)\|_{L^2(\Omega)} \leqslant \sup_{[0,T]} \|\nabla \overline{v}_\delta(t)\|_{L^2(\Omega)} \leqslant C_{T,g}.\]
Hence, the sets $\{\|\hat{u}_\delta (t)\|_2^2, \delta >0\}$ and $\{\|\hat{v}_\delta (t)\|_2^2, \delta >0\}$ are bounded in $[H^1(\Omega)]^2$, so they are compact in $[L^2(\Omega)]^2$ according to the Rellich-Theorem. %It means that for any $t>0$, $\{(\hat{u}_\delta(t), \hat{v}_\delta), \delta >0\}$ is compact for the strong topology of $(L^2(\Omega))^2$.  

\smallskip

\item[$\bullet$] The sequence $(\hat{u}_\delta, \hat{v}_\delta)_\delta$ is equi-continuous for the $L^2(\Omega)$ norm. This is a consequence of \eqref{bound overline+hat} and Jensen inequality : for any $t<s$, 
\begin{equation}\label{eq : prop3.4 1}
\|\hat{u}_\delta (t) - \hat{u}_\delta(s)\|_2^2 = \left\|\int_t^s \partial_t \hat{u}_\delta\right\|_2^2 \leqslant (s-t) \int_t^s \|\partial_t \hat{u}_\delta \|_2^2\leqslant (s-t) C_{g,T}.
\end{equation}
%\begin{multline*}
%\|\hat{u}_\delta (t) - \hat{u}_\delta(s)\|_2^2 = \left\|\int_t^s \partial_t \hat{u}_\delta\right\|_2^2 \leqslant (s-t) \int_t^s \|\partial_t \hat{u}_\delta \|_2^2\leqslant (s-t) C_{g,T}\\
%\leqslant (s-t) e^{T} (AT_\varepsilon (u_0 ,v_0) + C_{g,T}).
%\end{multline*}
A similar inequality occurs for $v$. 
\end{itemize} 
Hence, according to Ascoli's theorem, the sets $\{(\hat{u}_\delta, \hat{v}_\delta), \delta >0\}$ is a compact subset of $C^0([0,T]; L^2(\Omega))$ for any $T >0$. Using a diagonal extraction argument, we deduce the existence of a pair $(u,v) \in [C^0([0, +\infty) ; L^2(\Omega))]^2$ and a (not relabelled) converging subsequence $(\hat{u}_\delta, \hat{v}_\delta)$ to $(u,v)$ strongly in $[C^0([0,+\infty) ; L^2(\Omega))]^2$. 
This proves item (1).  

\medskip

\textbf{Step 2 :} $(\partial_t \hat{u}, \partial_t \hat{v}) \rightharpoonup (\partial_t u, \partial_t v)$ weakly in $[L_{loc}^2((0,\infty); L^2(\Omega))]^2$.  

We prove that $\partial_t \hat{u}_\delta \rightharpoonup \partial_t u$ weakly in $[L^2([0,T];L^2(\Omega)]^2$. Indeed, we first have $(\partial_t \hat{u}_\delta, \partial_t \hat{v}_\delta) \rightharpoonup  (\partial_t u, \partial_t v)$ in $[\mathcal{D}'((0,T)\times \Omega)]^2$. 
%for any test function $\varphi\in C_c^1((0,T)\times \Omega)$, 
%
%\[
%\int_0^T\int_\Omega \partial_t \hat{u}_\delta \varphi = -\int_0^T \int_\Omega \hat{u}_\delta \partial_t \varphi \longrightarrow -\int_0^T \int_\Omega u \partial_t \varphi = \int_0^T \int_\Omega \partial_t u \varphi.\]
%A similar convergence holds for $v$. We deduce that $(\partial_t \hat{u}, \partial_t \hat{v}) \rightharpoonup (\partial_t u, \partial_t v)$ weakly in $\mathcal{D}'((0,+\infty)\times \Omega)$. 
Using \autoref{cor h1 bound}, the sequence $(\partial_t \hat{u}_\delta, \partial_t \hat{v}_\delta)$ is bounded in $[L_{loc}^2((0,+\infty);L^2(\Omega))]^2$, so it admits a converging subsequence to an element in $[L_{loc}^2((0,+\infty);L^2(\Omega))]^2$. This element has to agree with $(\partial_t u, \partial_t v)$. This proves item (2). In particular, $(u,v)\in [H^1((0,T);L^2(\Omega))]^2$. 

\medskip

\textbf{Step 3 :} We now prove item (3).
%\textbf{Step 3 :} For almost every $t >0$, $(\overline{u}_\delta (t), \overline{v}_\delta (t)) \rightharpoonup (u(t), v(t))$ in $H^1(\Omega)$. 

\smallskip

%Now we prove that $(\overline{u}_\delta , \overline{v}_\delta) \rightharpoonup (u, v)$ in $L_t^2 H_x^1$. 
Fix a time $t  >0$. We prove that $\overline{u}_\delta (t) \rightarrow u(t)$ strongly in $L^2(\Omega)$. Indeed, let $i= \lfloor t/ \delta \rfloor$. Observe that with \autoref{Maximum principle} and the strong convergence $\hat{u}_\delta(t) \rightarrow u(t)$ in $L^2(\Omega)$,

\begin{multline*}
\|\overline{u}_\delta (t) - u(t) \|_2 = \|\hat{u}_\delta (i\delta) - u(t)\|_2 \leqslant \|\hat{u}_\delta (i\delta) - \hat{u}_\delta (t)\|_2 + \|\hat{u}_\delta (t) -u(t)\|_2 \\
\leqslant C_{g,T}\sqrt{\delta} + \|\hat{u}_\delta (t) -u(t)\|_2 \rightarrow 0.
\end{multline*}

Using \autoref{Maximum principle} and item (1), we can apply Lebesgue convergence theorem to conclude that $\overline{u}_\delta \rightarrow u$ strongly in $L_{loc}^2((0,+\infty)\times \Omega)$. We also deduce that there exists a (not relabelled) subsequence $\overline{u}_\delta$ that converges to $u$ almost everywhere on $(0,+\infty)\times \Omega$. A similar argument holds for $\overline{v}_\delta$, which concludes (3). 

\medskip

\textbf{Step 4 :} We prove items (4) and (5).

\smallskip

Recall from \autoref{Maximum principle} and \eqref{compacitestar} that the sequence $(\hat{u}_\delta(t) )$ is bounded in $H^1(\Omega)$. Hence, it admits a weakly-converging subsequence in $H^1(\Omega)$ and then, again due to Rellich theorem, a strong converging subsequence in $L^2(\Omega)$. Thanks to item (1), we deduce that $\hat{u}_\delta(t) \rightharpoonup u(t)$ weakly in $H^1(\Omega)$. We deduce that $t\mapsto u(t) \in H^1(\Omega)$ is weakly measurable as the limit of weak measurable functions. Since $H^1(\Omega)$ is separable, according to Pettis theorem (see \cite[Theorem III.1.1]{S97}), this map is in fact strongly measurable. We can integrate this convergence in time using Lebesgue dominated convergence theorem and \eqref{compacitestar} to conclude item (5). A similar argument applies to $\overline{v}_\delta$. 

\smallskip

\textbf{Step 5 :} The pair $(u, v)$ belongs to $[L^\infty_{loc}((0,+\infty); H^1(\Omega))\cap H^1_{loc}((0,+\infty) ; L^2(\Omega))]^2$. 

\smallskip

Using again \autoref{Maximum principle} and by the weak lower semi-continuity of $\|.\|_{H^1(\Omega)}$, we deduce that for all $t\in [0,T]$,

\[\|u(t)\|_{H^1(\Omega)} \leqslant \liminf\limits_{\delta\rightarrow 0} \|\overline{u}_\delta (t)\|_{H^1(\Omega)} \leqslant C_{g,T},\]
which proves that $u\in L_{loc}^\infty((0,+\infty); H^1(\Omega))$. Morevoer, since $\partial_t u \in L_{loc}^2([0,\infty);L^2(\Omega))$, we deduce that $u\in H^1_{loc}((0,\infty); L^2(\Omega))$.  
The similar reasoning applies to $\overline{v}_\delta$.  

\smallskip

\textbf{Step 6 :} Item (6) is a consequence of \autoref{Maximum principle} and \eqref{max prin delta}.

%Let $\varphi \in L^1((0,+\infty)\times \Omega)$. According to \autoref{max prin delta}, the function $\overline{u}_\delta(t,x) \varphi(t,x)$ is bounded almost everywhere on $(0,+\infty)\times \Omega$ by $\max( \|g\|_{\infty}, \|u_0\|_{\infty, \Omega})$ and according to (3), it converges almost everywhere on $(0,+\infty) \times \Omega$ to $u(t,x) \varphi(t,x)$. According to Lebesgue convegence theorem, we conclude that    
%\[\int_0^{+\infty}\int_\Omega \overline{u}_\delta (t,x)\varphi (t,x) dx dt \rightarrow \int_0^{+\infty} \int_\Omega u(t,x) \varphi(t,x) dx dt,\]
%which implies that $\overline{u}_\delta \overset{\star}{\rightharpoonup} u \in L^\infty((0,+\infty) \times \Omega)$ weakly-$\star$. A similar argument holds for $\overline{v}_\delta$.

\textbf{Step 7 :} We prove that $(u(t) - g(t), v(t)-1) \in H_0^1(\Omega)$ for almost every $t >0$. 

\smallskip

Indeed, according to Step 4 and \autoref{prop : cv g}, for almost every $t >0$, $\overline{u}_\delta (t) - \overline{g}_\delta (t) \rightharpoonup u(t) -g(t)$ weakly in $H^1(\Omega)$. Since $\overline{u}_\delta (t) - \overline{g}_\delta (t) \in H_0^1(\Omega)$, we deduce that $u(t) -g(t) \in H_0^1(\Omega)$ as well. The similar reasoning allows to conclude that $v(t) - 1\in H_0^1(\Omega)$ for almost every $t >0$ as well. This proves item (7).

Finally, as $(\overline{u}_\delta, \overline{v}_\delta \rightarrow (u,v)$ strongly in $L^2([0,T];L^2(\Omega))$, we can extract a subsequence that converges a.e on $[0,T]\times \Omega$, which proves (5).
\end{proof}

We can easily deduce the maximum principle for the pair $(u,v)$, which we state in the following proposition. 

\begin{proposition}\label{Maximum principle vrai}
Recall that $0 \leqslant v_0(x) \leqslant 1$ for all $x$. Then $0 \leqslant v(t,x) \leqslant 1$.
The pair $(u,v)$ satisfies 

\[0\leqslant v\leqslant 1, \quad \|u\|_\infty \leqslant \max(\|g\|_\infty, \|u_0\|_{\infty,\Omega}).\]
\end{proposition}

\begin{proof}
This is a consequence of \eqref{max prin delta} and item (3) of \autoref{extraction}. 
%This is a consequence of \autoref{extraction}, item (5). 
\end{proof}

We now prove that the pair $(u,v)$ is a solution of \eqref{système de départ} in the sense of \autoref{solution faible}.  

\begin{proposition}
The pair $(u,v)$ is a solution of \eqref{solution faible}. 
\end{proposition}

\begin{proof}
In order to pass to the limit in \eqref{EL approché t fixe3} and especially in the second equation (via weak-strong convergence) \eqref{EL approché t fixe4}, we must prove the strong convergence of the term \(\nabla \overline{u}_\delta\) in $L^2(\Omega)$ in the equation \eqref{solution faible}. Our proof is inspired from the method of \cite{X94}. Namely we multiply the second equation of \eqref{EL approché t fixe3} by \( \overline{u}_\delta-u\) and obtain the strong convergence of \( \nabla (\overline{u}_\delta-u)\) in \(L^2_{\rm loc}((0,+\infty)\times \Omega )\). Here are the details. Let $\varphi \in \mathcal C^\infty_c((0,+\infty)\times \Omega)$ with $\varphi \geqslant 0$. According to \autoref{Maximum principle vrai} and 
\autoref{cor h1 bound}, the function $\varphi(\overline{u}_\delta-u)$ then belongs to $L_{loc}^\infty((0,+\infty)\times \Omega)\cap H_{loc}^1((0,+\infty) ; L^2(\Omega)) \cap L^2((0,+\infty); H_0^1 (\Omega))$. Applying it as a test function in \eqref{EL approché t fixe3} yields
\begin{multline*}
\int_0^{+\infty} \int_\Omega \partial_t \hat{u}_\delta(\overline{u}_\delta-u) \varphi+\int_0^{+\infty}\int_\Omega (\eta_\varepsilon+\overline{v}_\delta^2)\varphi \nabla \overline{u}_\delta\cdot \nabla (\overline{u}_\delta-u)\\
+\int_0^{+\infty}\int_\Omega (\eta_\varepsilon+\overline{v}_\delta^2)(\overline{u}_\delta-u) \nabla \overline{u}_\delta\cdot \nabla \varphi=0.
\end{multline*}
By weak-strong convergence, the first and third terms satisfy
$$\int_0^{+\infty} \int_\Omega \partial_t \hat{u}_\delta(\overline{u}_\delta-u) \varphi+\int_0^{+\infty}\int_\Omega (\eta_\varepsilon+\overline{v}_\delta^2)(\overline{u}_\delta-u) \nabla \overline{u}_\delta\cdot \nabla \varphi \to 0,$$
hence
\begin{equation}\label{existence1}
\int_0^{+\infty}\int_\Omega (\eta_\varepsilon+\overline{v}_\delta^2)\varphi \nabla \overline{u}_\delta\cdot \nabla (\overline{u}_\delta-u) \to 0.
\end{equation}
According to \autoref{extraction} (item (3) and (5)), $\eta_\varepsilon + \overline{v}_\delta ^2$ converges strongly to $\eta_\varepsilon + v^2$ in $L^2_{loc}((0,+\infty);L^2(\Omega))$  and $\nabla \overline{u}_\delta \rightharpoonup \nabla u$ weakly in $L_{loc}^2((0,+\infty);L^2(\Omega))$, which implies
\begin{equation}\label{cv b}
\int_0^{+\infty}\int_\Omega (\eta_\varepsilon + \overline{v}_\delta ^2) \nabla \overline{u}_\delta \cdot \nabla \varphi \longrightarrow \int_0^{+\infty}\int_\Omega (\eta_\varepsilon + v ^2) \nabla u \nabla \varphi.
\end{equation}
Combining \eqref{existence1} and \eqref{cv b}, we deduce
\begin{multline*}\lim_{\delta \to 0} \int_0^{+\infty}\int_\Omega (\eta_\varepsilon+\overline{v}_\delta^2)\varphi |\nabla \overline{u}_\delta|^2 =\lim_{\delta \to 0} \int_0^{+\infty}\int_\Omega (\eta_\varepsilon+\overline{v}_\delta^2)\varphi \nabla \overline{u}_\delta\cdot \nabla u\\
=\int_0^{+\infty}\int_\Omega (\eta_\varepsilon+v^2)\varphi |\nabla u|^2,
\end{multline*}
%where the last equality follows from the strong convergence $\overline{v}_\delta \rightarrow v$ in $L^2_{loc}((0,+\infty); L^2(\Omega))$ and the weak convergence $\nabla \overline{u}_\delta \rightharpoonup \nabla u$ in $L_{loc}^2((0,+\infty) ; L^2(\Omega))$. 
%Notice that since \autoref{extraction} (item (1)-(2))  , for any $t\in (0,T)$,  
%
%\[\displaystyle \int_\Omega (\eta_\varepsilon + \overline{v}_\delta^2(t)) \varphi \nabla \overline{u}_\delta (t) \cdot \nabla u (t) \rightarrow \int (\eta_\varepsilon + v^2) \nabla u(t) \cdot \nabla u(t).\]
%
%Using \eqref{energy bound continue} and the Lebesgues dominated convergence, we deduce the following from \eqref{existence1} : 
This implies that
\begin{align*}
\eta_\varepsilon \int_0^{+\infty} \int_\Omega  |\nabla (\overline{u}_\delta-u)|^2 \varphi &\leqslant \int_0^{+\infty} \int_\Omega  (\eta_\varepsilon+\overline{v}_\delta^2) |\nabla (\overline{u}_\delta-u)|^2 \varphi \\
& = \int_0^{+\infty} \int_\Omega (\eta_\varepsilon+\overline{v}_\delta^2)|\nabla \overline{u}_\delta|^2 \varphi- 2\int_0^{+\infty} \int_\Omega (\eta+\overline{v}_\delta^2)\varphi \nabla u\cdot \nabla \overline{u}_\delta\\
&\quad +\int_0^{+\infty} \int_\Omega (\eta_\varepsilon+\overline{v}_\delta^2)|\nabla u|^2 \varphi\longrightarrow 0,
\end{align*}
 hence
\begin{equation}\label{cv forte L1}
\lim_{\delta \to 0} \int_0^{+\infty} \int_\Omega |\nabla (\overline{u}_\delta-u)|^2\varphi \, dt\, dx=0.
\end{equation}
We deduce that $\nabla \overline{u}_\delta \to \nabla u$ strongly in $L^2_{\rm loc}((0,+\infty) \times \Omega)$.
 
 \medskip
 
We are now ready to pass to the limit in each equation. Let $\varphi \in \mathcal{C}_c^\infty ((0,+\infty) \times \Omega)$ be a test function. For any $\delta$, 

\[0 =\int_0^{+\infty} \int_\Omega \partial_t \hat{u}_\delta \varphi + \int_0^{+\infty}\int_\Omega (\eta_\varepsilon + \overline{v}_\delta ^2) \nabla \overline{u}_\delta \nabla \varphi.\]
Using \autoref{extraction} item (2), $\partial_t \hat{u}_\delta \rightharpoonup \partial_t u$ in $L_{loc}^2((0,+\infty);L^2(\Omega))$, hence,

\begin{equation}\label{cv a}
\int_0^{+\infty} \int_\Omega \partial_t \hat{u}_\delta \varphi \rightarrow \int_0^{+\infty} \int_\Omega \partial_t u \varphi
\end{equation}
%Using that $\eta_\varepsilon + \overline{v}_\delta ^2$ converges strongly to $\eta_\varepsilon + v^2$ in $L^2([0,T];L^2(\Omega))$  and that $\nabla \overline{u}_\delta \rightharpoonup \nabla u$ weakly in $L^2([0,T];L^2(\Omega))$, 
Combining \eqref{cv a} and \eqref{cv b}, we deduce that $u$ satisfies 

\begin{equation}\label{cv 1}
0 = \int_0^{+\infty} \int_\Omega \partial_t u \varphi + \int_0^{+\infty}\int_\Omega (\eta_\varepsilon + v ^2) \nabla u \nabla \varphi.
\end{equation}
Now consider \eqref{EL approché t fixe4}.

\[0 = \int_0^{+\infty} \int_\Omega \partial_t \hat{v}_\delta \varphi + \int_0^{+\infty} \int_\Omega \left[\varepsilon \nabla \overline{v}_\delta \cdot \nabla \varphi + \frac{\overline{v}_\delta -1}{4\varepsilon} \varphi + \overline{v}_\delta |\nabla \overline{u}_\delta |^2 \varphi\right].\]
Using that $ \partial_t \hat{v}_\delta \rightharpoonup \partial_t v$ in $L_{loc}^2((0,+\infty);L^2(\Omega))$, we have 

\begin{equation}\label{cv c}
\int_0^{+\infty} \int_\Omega \partial_t \hat{v}_\delta \varphi \rightarrow \int_0^{+\infty} \int_\Omega \partial_t v \varphi
\end{equation}
Using that  $\overline{v}_\delta \rightharpoonup  v$ in $L_{loc}^2((0,+\infty);H^1(\Omega))$ (\autoref{extraction} item (5)), 

\begin{equation}\label{cv d}
\int_0^{+\infty} \int_\Omega \varepsilon \nabla \overline{v}_\delta \cdot \nabla \varphi + \frac{\overline{v}_\delta -1}{4\varepsilon} \varphi \rightarrow \int_0^{+\infty} \int_\Omega \varepsilon \nabla v \cdot \nabla \varphi + \frac{v -1}{4\varepsilon} \varphi.
\end{equation}

Finally, since $\overline{v}_\delta \overset{\star}{\rightharpoonup} v$ in $L^\infty((0,+\infty); L^{\infty}(\Omega))$ (\autoref{extraction} item (6)) and  $|\nabla \overline{u}_\delta |^2$ converges to $|\nabla u|^2$ in $L_{loc}^1((0,+\infty)\times \Omega)$ (see \eqref{cv forte L1}), we deduce by Lebesgue dominated convergence that 

\begin{equation}\label{cv e}
\int_0^{+\infty} \int_\Omega  \overline{v}_\delta |\nabla \overline{u}_\delta |^2 \varphi \rightarrow \int_0^{+\infty} \int_\Omega v|\nabla u|^2 \varphi.
\end{equation}
Combining \eqref{cv c}, \eqref{cv d} and \eqref{cv e} yields \eqref{eq : eq in v}. 
\end{proof}

To conclude the proof of \autoref{Theorem 1}, we now prove the energy bound \eqref{energy bound thm 1}. 

\begin{proposition}
For almost every $t >0$, one has 
\begin{multline*}
AT_\varepsilon (u(t), v(t)) + \int_0^t \left(\|\partial_t u(s)\|_2^2 +\|\partial_t v(s)\|_2^2 \right) ds \leqslant AT(u_0, v_0)\\
 + \int_0^t \|\partial_t g(s)\|_2^2 ds + \int_0^t\int_\Omega (\eta_\varepsilon + v^2(s)) \left(|\partial_t \nabla g(s)|^2 + 2 \nabla u(s) \cdot \nabla \partial_t g(s)\right) dxds.
\end{multline*}
\end{proposition}

\begin{proof}
Recall that $(\overline{u}_\delta, \overline{v}_\delta, \overline{g}_\delta)$ satisfy \eqref{energy bound continue} 
\begin{multline*}
AT_\varepsilon (\overline{u}_\delta (t) , \overline{v}_\delta (t)) + \int_0^{t_\delta} \left(\|\partial_t \hat{u}_\delta(s)\|_2^2 + \|\partial_t \hat{v}_\delta(s)\|_2^2 \right)ds 
\leqslant AT_\varepsilon (u_0 , v_0) + \int_0^{t_\delta} \|\partial_t \hat{g}_\delta (s)\|_2^2 ds\\
+ \int_0^{t_\delta}\int_\Omega (\eta_\varepsilon + \overline{v}_\delta(s)^2)\left(|\partial_t \nabla \hat{g}_\delta (s)|^2ds + 2 \nabla \overline{u}_\delta (s)\cdot \nabla (\partial_t \hat{g}_\delta (s)\right) dx ds.
\end{multline*}
We pass to the limit in each term. 

First, since for almost every $t> 0$, $\overline{v}_\delta(t) \rightharpoonup v(t)$ weakly in $H^1(\Omega)$ and $\overline{v}_\delta(t) \rightarrow v(t)$ strongly in $L^2(\Omega)$ (according to \autoref{extraction} items (3) and (4)), we obtain that 

\begin{equation}\label{bound continue 1}
\int_\Omega \varepsilon |\nabla v(t)|^2 + \frac{(1-v(t))^2}{4\varepsilon} \leqslant \liminf\limits_{\delta \rightarrow 0}\int_\Omega \varepsilon |\nabla \overline{v}_\delta(t)|^2 + \frac{(1-\overline{v}_\delta)^2}{4\varepsilon}.
\end{equation}
Since, for almost every $t >0$, $\overline{u}_\delta(t) \rightharpoonup u(t)$ weakly in $H^1(\Omega)$ and $\overline{v}(t) \rightarrow v(t)$ strongly in $L^2(\Omega)$ (according again to \autoref{extraction} items (3) and (4)), 

\begin{equation}\label{bound continue 2}
\int_\Omega(\eta_\varepsilon + v(t)^2) |\nabla u(t)|^2 \leqslant \liminf\limits_{\delta \rightarrow 0}\int_\Omega(\eta_\varepsilon + \overline{v}_\delta(t)^2) |\nabla \overline{u}_\delta(t)|^2.
\end{equation}
Secondly, using that $(\partial_t \hat{u}_\delta , \partial_t \hat{v}_\delta) \rightharpoonup (\partial_t v, \partial_t u)$ weakly in $[L^2((0,t) \times \Omega)]^2$ (according to \autoref{extraction} item (2)) and that $1_{(0,t_\delta)} \rightarrow 1_{(0,t)}$ strongly in $L^2((0,T); L^2(\Omega))$, 

\begin{multline}\label{bound continue 3}
\int_0^{t} \left(\|\partial_t u(s)\|_{2}^2 + \|\partial_t v(s)\|_{2}^2\right) ds \\
\leqslant \liminf\limits_{\delta\rightarrow 0} \int_0^{t} \left(\|\partial_t \hat{u}_\delta(s)\|_{2}^2 + \|\partial_t \hat{v}_\delta(s)\|_{2}^2\right)1_{(0,t_\delta)}(s) \, ds. 
\end{multline}
Combining \eqref{bound continue 1}, \eqref{bound continue 2} and \eqref{bound continue 3}, we deduce that for almost every $t>0$,

\begin{multline}\label{bound continue 4}
AT_\varepsilon (u(t), v(t)) + \int_0^t \left(\|\partial_t u(s)\|_{2}^2 +\|\partial_t v(s)\|_{2}^2 \right)ds\\
\leqslant \liminf\limits_{\delta \rightarrow 0} AT_\varepsilon (\overline{u}_\delta(t), \overline{v}_\delta(t)) + \int_0^{t_\delta} \left(\|\partial_t \hat{u}_\delta(s) \|_{2}^2 + \|\partial_t \hat{v}_\delta (s) \|_{2}^2\right)ds.
\end{multline}
On the other hand, using that $\partial_t\hat{g}_\delta\rightarrow \partial_t g$ strongly in $L^2((0,t)\times \Omega)$ (according to \autoref{prop : cv g} item (3)) and that $1_{(0,t_\delta)} \rightarrow 1_{(0,t)}$ strongly in $L^2((0,t); L^2(\Omega))$, we deduce that 

\begin{equation}\label{bound continue 5}
\lim\limits_{\delta \rightarrow 0} \int_0^{t_\delta} \|\partial_t \hat{g}_\delta(s)\|_2^2 ds = \int_0^t \|\partial_t g(s)\|_2^2 ds. 
\end{equation}
Using that $\partial_t\nabla \hat{g}\rightarrow \partial_t \nabla g$ strongly in $L^2((0,t);L^2(\Omega))$, that $\overline{v}_\delta \rightarrow v$ strongly in $L^2((0,t)\times \Omega)$ and that $\nabla \overline{u} \rightharpoonup \nabla u$ weakly in $L^2((0,t)\times \Omega)$, %that $1_{(0,t_\delta)} \rightarrow 1_{(0,t)}$ strongly in $L_t^2 L_x^2$, we deduce that 

\begin{equation}\label{bound continue 6}
\lim\limits_{\delta \rightarrow 0} \int_0^{t_\delta} \int_\Omega(\eta_\varepsilon +\overline{v}_\delta^2(s)) |\partial_t \nabla \hat{g}_\delta(s)|^2 \, dx ds = \int_0^{t_\delta}\int_\Omega (\eta_\varepsilon + v^2(s))|\partial_t \nabla g(s)|^2 dx ds,
\end{equation}

\begin{equation}\label{bound continue 7}
\lim\limits_{\delta \rightarrow 0} \int_0^{t_\delta}\int_\Omega 2(\eta_\varepsilon +\overline{v}_\delta^2(s)) \nabla \overline{u}_\delta(s)\cdot \nabla \partial_t \hat{g}_\delta(s) \, dxds= \int_0^{t_\delta} \int_\omega 2(\eta_\varepsilon + v^2(s))\nabla u(s) \cdot \nabla \partial_t g(s) \, dxds. 
\end{equation}
Combining \eqref{bound continue 4}, \eqref{bound continue 5}, \eqref{bound continue 6}, \eqref{bound continue 7}, we deduce the desired property. 
\end{proof}

\textbf{Remark :} In the case where $(u,v)\in \mathcal{C}_{par}^{1+\alpha/2,2+\alpha}([0,\infty)\times \overline{\Omega})$ for $\alpha \in (0,1)$, the pair $(\partial_t u -\partial_t g, \partial_t v)$ is in $\mathcal{F}$. Using $\partial_t u-\partial_t g$ as a test function in \eqref{eq : eq in u} and $\partial_t v$ in \eqref{eq : eq in v}, we can obtain an equality in \eqref{energy bound thm 1}. 

%To conclude the proof of \autoref{Theorem 1}, we state the maximum principle for the pair $(u,v)$ :

\section{Uniqueness}
\label{Section unicité}

In this section, we prove \autoref{Theorem 3}. This is merely an adaptation of the proof of \cite[Theorem 2.3, Step 3]{FP04}. 

Recall the Gagliardo-Nirenberg inequality (see \cite{N59}) :

\begin{lemma}\label{lem : GN}
Suppose that the boundary $\partial\Omega$ is $\mathcal{C}^2$. 
Let $p, r,s$ and $\alpha\in (0,1)$ be positive numbers satisfying the relation 

\[\frac{1}{p} = \alpha \left(\frac{1}{r} -\frac{1}{d}\right) + \frac{1-\alpha}{s}.\]
There exists a constant $C_G$ (depending on $p, r, s$ and $\alpha$) such that for any $f \in L^s(\Omega)$ such that $\nabla f \in L^r(\Omega)$,
\[\|f\|_{p} \leqslant C_G \left(\|\nabla f\|_r^{\alpha} \|f\|_s^{1-\alpha} + \|f\|_s\right).\]
\end{lemma}

\begin{proof}[Proof of \autoref{Theorem 3}]
Let $(u_1, v_1)$ and $(u_2, v_2)$ be two solutions of \eqref{système de départ} in the sense of \autoref{solution faible} and set $u = u_1 -u_2$ and $v=v_1 - v_2$. 
Substracting the equations satisfied by $u_1$ and $u_2$ and taking $u1_[0,t]$ as a test function leads to 

\begin{equation}\label{uniqueness 1}
\int_0^t \int_\Omega \partial_t u u + \int_0^t \int_\Omega \nabla u\cdot [(\eta_\varepsilon + v_1^2) \nabla u_1 - (\eta_\varepsilon + v_2^2) \nabla u_2) = 0 \quad \text{for all t >0}.
\end{equation}
Differentiating \eqref{uniqueness 1} with respect to $t$ leads to

\begin{align}\label{uniqueness 2}
\frac{1}{2}\frac{d}{dt} \|u\|_2^2 + \eta_\varepsilon  \|\nabla u(t)\|_2^2 &= \int_{\Omega} (v_2^2(t)\nabla u_2(t) - v_1^2(t) \nabla u_1(t))\cdot \nabla u (t) \nonumber\\
&= \int_\Omega (v_2(t)^2 \nabla u_2(t) - v_1(t)^2 (\nabla u(t) + \nabla u_2(t)))\cdot \nabla u(t) \nonumber\\
&= -\int_\Omega v_1(t)^2 |\nabla u(t)|^2 - \int_\Omega v(t) (v_1(t) +v_2(t))\nabla u_2(t)\cdot \nabla u(t) \nonumber\\
&\leqslant - \int_\Omega v(t) (v_1(t) +v_2(t))\nabla u_2(t)\cdot \nabla u(t).
\end{align}
%which rewrites as 
%
%\begin{equation}\label{uniqueness 2}
%\frac{1}{2}\frac{d}{dt} \|u\|_2^2 + \int_\Omega (\eta_\varepsilon + v_1^2) |\nabla u|^2 = \int_\Omega  v(v_1+v_2) \nabla u_2 \nabla u.
%\end{equation}
We proceed similarly for $v$ : substracting the equations satisfied by $v_1$ and $v_2$ and taking $v 1_[0,t]$ as a test function leads to, after differentiating with respect to $t$,  

\begin{align}\label{uniqueness 3}
&\frac{1}{2}\frac{d}{dt} \|v(t)\|_2^2 + \varepsilon \|\nabla v(t)\|_2^2 + \frac{1}{4\varepsilon} \|v(t)\|_2^2\\
 &= \int_\Omega (v_2(t)|\nabla u_2(t)|^2 - v_1(t)|\nabla u_1(t)|^2) v(t) \nonumber\\
 &= \int_\Omega (v_2(t)|\nabla u_2(t)|^2 - (v(t)+v_2(t))|\nabla u_1(t)|^2) v(t)\nonumber \\
 &= -\int_\Omega v(t)^2 |\nabla u_1(t)|^2 - \int_\Omega v_2(t) v(t) \nabla u(t)\cdot (\nabla u_2(t) + \nabla u_1(t)) \nonumber\\
 & \leqslant - \int_\Omega v_2(t) v(t) \nabla u(t)\cdot (\nabla u_2(t) + \nabla u_1(t)).
\end{align}
%which rewrites as 
%
%\begin{equation}\label{uniqueness 3}
%\frac{1}{2}\frac{d}{dt} \|v(t)\|_2^2 + \varepsilon \|\nabla v(t)\|_2^2 + \frac{1}{4\varepsilon} \|v(t)\|_2^2 + \int_\Omega v^2 |\nabla u_1|^2 =\int_\Omega v_2 v \nabla u\cdot (\nabla u_2 + \nabla u_1). 
%\end{equation}
Summing \eqref{uniqueness 2} and \eqref{uniqueness 3}, we obtain 

\begin{multline}\label{uniqueness 4}
\frac{1}{2} \frac{d}{dt} \left(\|u(t)\|_2^2+ \|v(t)\|_2^2\right) +\varepsilon\|\nabla v(t)\|_2^2 +\eta_\varepsilon \|\nabla u(t)\|^2 \\
\leqslant - \int_\Omega  v(t)(v_1(t)+2v_2(t)) \nabla u_2(t) \nabla u(t) - \int_\Omega v(t) v_2(t) \nabla u(t) \cdot \nabla u_1(t). 
\end{multline}

When we apply \autoref{lem : GN} to $v(t)$, with $p = 2 \dfrac{d+2}{d}$, $r=s=2$ and $\alpha = \dfrac{d}{d+2}$, we find
\[\|v(t)\|_{2\frac{d+2}{d}} \leqslant C_G \|\nabla v(t)\|_2^{\frac{d}{d+2}} \|v(t)\|_2^{\frac{2}{d+2}} + C_G \|v(t)\|_2.\] 
Since $0\leqslant v_1, v_2\leqslant 1$, using twice Young's inequality for a given $\delta >0$ to be chosen later and Hölder inequality for the pair $(1+d/2,1+2/d)$, there exist $C_\delta, C_\delta ' >0$ such that  

\begin{align*}
&\left|\int_\Omega (v_1(t) +2v_2(t)) v(t) \nabla u(t) \cdot \nabla u_2(t)\right| \\
&\leqslant 3 \left(C_\delta \int_\Omega v(t)^2 |\nabla u_2(t)|^2 + \delta \int_\Omega |\nabla u(t)|^2\right) \\
& \leqslant 3C_\delta \|\nabla u_2(t)\|_{d+2}^2 \|v(t)\|_{\frac{2(d+2)}{d}}^2 + 3\delta\int_\Omega |\nabla u(t)|^2\\
&\leqslant 3C_\delta \|\nabla u_2(t)\|_{d+2}^2 C_G^2 \left( \|\nabla v(t)\|_2^{\frac{d}{d+2}} \|v(t)\|_2^{\frac{2}{d+2}} + \|v(t)\|_2\right)^2 + 3\delta\int_\Omega |\nabla u(t)|^2\\
&\leqslant 6C_\delta C_G^2 \left(\|\nabla u_2(t)\|_{d+2}^2 \|\nabla v(t)\|_2^{\frac{2d}{d+2}} \|v(t)\|_2^{\frac{4}{d+2}} + \|\nabla u_2(t)\|_{d+2}^2 \|v(t)\|_2^2 \right)+ 3\delta\int_\Omega |\nabla u(t)|^2\\
&\leqslant 6C_\delta C_G^2 \left(C'_\delta \|\nabla u_2(t)\|_{d+2}^{2\frac{d+2}{2}} \|v(t)\|_2^2  + \delta \|\nabla v(t)\|_2^2 \right)\\
&\quad + 6C_\delta C_G^2 \|\nabla u_2(t)\|_{d+2}^2 \|v(t)\|_2^2 + 3\delta\int_\Omega |\nabla u(t)|^2.
\end{align*}
Similarly, we get that 

\begin{multline*}
\left|\int_\Omega v(t) v_2(t) \nabla u(t) \cdot \nabla u_1(t) \right| \\
\leqslant C_{G,\delta} \left(\|\nabla u_1 (t)\|_{d+2}^{2\frac{d+2}{2}} \|v(t)\|_2^2 + \|\nabla u_1(t)\|_{d+2}^2\|v(t)\|_2^2 + \delta \|\nabla v(t)\|_2^2 \right) +\delta \int_{\Omega} |\nabla u(t)|^2.
\end{multline*}
For $\delta$ small enough, we can simplify \eqref{uniqueness 4} into 

\begin{multline}\label{G9}
\frac{d}{dt}\left(\|u(t)\|_2^2 +\|v(t)\|_2^2\right)\\
\leqslant C \left(\|\nabla u_2(t)\|_{d+2}^{d+2} +\|\nabla u_1(t)\|_{d+2}^{d+2} + \|\nabla u_2(t)\|_{d+2}^2 +\|\nabla u_1(t)\|_{d+2}^2\right)\|v(t)\|_2^2. 
\end{multline}
Since $\nabla u_2, \nabla u_1 \in L^{d+2}((0,T); L^{d+2}(\Omega))$ for any $T >0$, and since $x^2 \leqslant 1 + x^{d+2}$ for any positive real number $x$, we have that for any $T >0$, 

\begin{multline*}
\int_0^T \left(\|\nabla u_2(t)\|_{d+2}^{d+2} + \|\nabla u_2(t)\|_{d+2}^2\right)dt\\
 \leqslant \int_0^T \|\nabla u_2(t)\|_{d+2}^{d+2}dt + \int_0^T \left( 1+ \|\nabla u_2(t)\|_{d+2}^{d+2} \right)dt <+\infty.
\end{multline*}

\begin{multline*}
\int_0^T \left(\|\nabla u_1(t)\|_{d+2}^{d+2} + \|\nabla u_1(t)\|_{d+2}^2\right)dt\\
 \leqslant \int_0^T \|\nabla u_1(t)\|_{d+2}^{d+2}dt + \int_0^T \left( 1+ \|\nabla u_1(t)\|_{d+2}^{d+2} \right)dt <+\infty.
\end{multline*}
We are then in position to apply Gronwall's lemma to \eqref{G9} to conclude that $\|u(t)\|_2^2 +\|v(t)\|_2^2 =0$ for any $t$, leading to $u=v=0$. This concludes the proof. 
\end{proof}

\section{Regularity of solutions to the Ambrosio-Tortorelli equations}
\label{section régularité}

We prove \autoref{Theorem 2}, i.e., that the weak solutions constructed in \autoref{Section existence faible} are in fact smooth. 

We recall the outline of the proof of the inner regularity. The main step is to prove that $(u,v)\in \mathcal{C}^{\alpha/2, \alpha}(Q_R(t_0,x_0))$ for a real number $\alpha$ and a parabolic cylinder $Q_R(t_0,x_0) \subset (0,+\infty)\times \Omega$. For the function $u$, this is a consequence of DeGiorgi--Nash--Moser theorem.

For the function $v$, as in the elliptic framework, the parabolic Hölder spaces can be linked with the Morrey-Campanato spaces (see \cite[Section 5.1]{GM12}, and the Hölder condition can be rewritten as a control of the distance of the function to its average, via the following lemma, which is stated in \cite[Lemma 4.3]{L96}   
\begin{lemma}\label{lem : Lieberman}
Let $f\in L^1(Q_{2R}(t_0,x_0))$. Suppose that there exist $\alpha \in (0,1)$ and a constant $H >0$ such that, for any $Y\in Q_{R}(t_0,x_0)$ and any $r\in (0,R)$, 

\[\iint_{Q_r (Y)} |f(X) - \{f\}_{t,Y}|dX \leqslant H r^{d+2+\alpha}.\] 
Then there is a constant $C(d,\alpha)$ such that 

\[\sup_{Q_R(t_0,x_0)} \frac{|f(X) - f(X')|}{d_p (X, X')^{\alpha}} \leqslant C(d, \alpha) H.\]
\end{lemma}

The proof of this estimate of \eqref{lem : Lieberman} for the function $v$ follows the following steps : 

\begin{itemize}
\item \textbf{Step 1 :} We prove that $\nabla u \in \mathcal{L}^{2,d+2\alpha}(Q_R(t_0,x_0))$. In the elliptic case, this is a consequence of DeGiorgi--Nash--Moser theorem. In the parabolic case, this is the object of \autoref{prop:grad_spat_Morrey}.

\item \textbf{Step 2 :} We introduce the ``heat''-extension $w$ of $v$ on a smaller cylinder $Q_r(t_0,x_0)$, which corresponds to the solution of the heat equation $\partial_t w - \varepsilon \Delta w =0$ on $Q_r(t_0,x_0)$ with $w=v$ on $\partial_\sqcup Q_r(t_0,x_0)$. The function $w$ satisfies 

\[\iint_{Q_\rho(t_0,x_0)} |w -\{w\}_{\rho,t_0,x_0}|^2 \leqslant C \left(\frac{\rho}{r}\right)^{d+4} \iint_{Q_r(t_0,x_0)} |w -\{w\}_{r,t_1,x_1}|^2\]
for any $\rho\leqslant r$.

\item \textbf{Step 3 :} The error function $v_2 =v-w$ is the solution of a heat equation with a source term in a Morrey space. As $v_2\in L^2((0,+\infty);H_0^1(\Omega))$ we can use it as a test function to estimate $\iint_{Q_r(t_0,x_0)} |\nabla v_2|^2$. Finally we can control $\|v_2 -\{v_2\}_{\rho, t_0,x_0}\|_{2, Q_\rho(t_0,x_0)}$ by $\|\nabla v_2\|_{2,Q_\rho(t_0,x_0)}$ using a Poincare-type inequality. The proof of the Poincare inequality for solutions of heat equation with a source term in a Morrey space is the object of \autoref{prop : Poincare}.  
\end{itemize} 

\smallskip

To address the regularity at the boundary $\partial_\sqcup [(0,+\infty)\times \Omega]$, we need to consider the bottom boundary $\{0\}\times \Omega$ and the lateral boundary $(0,+\infty) \times \partial \Omega$. In each case, the strategy is to extend the pair $(u,v)$ beyond the boundary and then prove the inner regularity on the extended domain. The resulting equations for the extended pairs contain non constant coefficients and additional source terms. The strategy of the proof remains the same as the strategy below for the inner regularity, although the equations involve more terms (see \eqref{extension uv eq}, \eqref{boundary : equ} and \eqref{boundary : eqv}). The lemmas \autoref{prop:grad_spat_Morrey} and \autoref{prop : Poincare} are stated in a general way to encapsulate all the equations.  

\medskip

The section is divided into three parts. In the first subsection, we establish preliminary results concerning parabolic equations with a source term that belongs to a Morrey space, namely a Campanato-Morrey and a Poincaré inequality. The second section is dedicated to the proof of the inner Hölder continuity of the pair $(u, v)$. The third section is dedicated to the proof of the Hölder continuity of the pair $(u,v)$ at the boundary $((0,+\infty) \times \partial\Omega) \cup (\{0\} \times \Omega)$.  

\medskip  

\subsection{Preliminary results - Parabolic equations and Morrey space}

In this subsection, we give several results concerning estimates in a the Morrey space of solutions of the heat equation.
%with a source term belonging to a Morrey space $\mathcal{L}^{1, d+2\alpha}$. 

\smallskip

\subsubsection{The Morrey-Campanato inequality} 
 
The following lemma connects the parabolic Hölder continuity with a Morrey-Campanato. In other word, for a function $f$ satisfying a PDE with a source term in a Morrey space and which is $\alpha-$Hölder, we expect the gradient $\nabla f$ to satisfy a Morrey-Campanato estimate. In the elliptic case, this estimate is a consequence of the elliptic form of De-Giorgi-Nash-Moser theorem (see \cite[Theorem 3.1]{BMRa23}). In our case, the estimates follow from a Caccioppoli-type estimate. This lemma will be applied to the function $u$ to prove that $|\nabla u| \in \mathcal{L}^{2,d+2\alpha}(Q_R)$.      

\begin{lemma}\label{prop:grad_spat_Morrey}(Morrey-Campanato inequality)
Let $(t_0, x_0) \in (0, +\infty) \times \Omega$ and $R >0$ be such that $Q_R =Q_R(t_0,x_0) \subset \subset (0,T) \times \Omega$. Let $\alpha \in (0,1)$ and let $A : \mathbb{R}^d  \rightarrow \mathbb{M}_d(\mathbb{R})$, $B : \mathbb{R}^d \rightarrow \mathbb{R}^d$ be measurable functions and $h \in L^{\infty}(Q_R), h_1 \in \mathcal{C}^1(Q_R)$ and $h_2 \in L^\infty (Q_R ;\mathbb{R}^d)$ such that 
\[0< \lambda I_d \leqslant A(x) \leqslant \Lambda I_d, \quad |B(x)| \leqslant \beta\]
for some positive constants $\lambda, \Lambda, \beta$.  

\smallskip

Let $f \in L^2((0,T); H^1(\Omega))\cap H^1((0,T); L^2(\Omega)) \cap \mathcal{C}_{par}^{\alpha/2, \alpha}((0,T)\times \Omega)$ be a weak solution of

\begin{equation}
\partial_t f - div(A\nabla f) + B \cdot \nabla f= h + h_1 \text{div}(h_2) \quad \text{in } Q_R(t_0,x_0).
\end{equation}
meaning that for any test function $\varphi \in H_0^1(Q_1 \cap L^\infty(Q_1)$, 
\begin{multline}\label{eq : f campanato}
\iint_{Q_R} \partial_t f(t,x) \varphi(t,x) dxdt + \iint_{Q_R}A(x)\nabla f(t,x)\cdot \nabla \varphi(t,x) dxdt\\ + \iint_{Q_R}B(x)\cdot\nabla f(t,x) \varphi(t,x) dxdt
= \iint_{Q_R} h(t,x) \varphi(t,x) dxdt\\ -\iint_{Q_R} h_1(t,x) h_2(t,x) \cdot \nabla \varphi(t,x) dxdt - \iint_{Q_R} h_2(t,x)\cdot \nabla h_1(t,x) \varphi(t,x) dxdt. 
\end{multline}
There exists \(C_{CM} >0\) such that for any \( (t,x)\in (0,+\infty)\times \Omega\) and any \(0<\rho<R\) such that \( Q_\rho(t,x)\subset Q_R(t_0,x_0) \subset \subset (0,+\infty)\times \Omega\) we have
\begin{equation}\label{eq:decroissance}
\iint_{Q_\rho(t_0,x_0)} |\nabla f(t,x)|^2 dxdt \leqslant C_{CM} \left(\frac{\rho}{R} \right)^{d+2\alpha} R^{d+2}.
\end{equation}
where \(C_{CM}\) depends on \(\alpha,d, R, \lambda, \Lambda, \beta, \|h_1\|_{C^1}, \|h_2\|_\infty, [f]_{\alpha}, \|\mu\|_{\mathcal{L}^{1,d+2\alpha}}\).
\end{lemma}

This means that \(\nabla f\) is locally in a parabolic Morrey space \(\mathcal{L}^{2,d+2\alpha}\). 
%The proof of this proposition follows the same line as in the elliptic case.
% We first use De Giorgi-Nash-Moser theorem to obtain that \(u\) is H\"older continuous and then we use a Cacciopoli inequality along with a Poincar\'e-Wirtinger inequality to obtain the desired growth of the gradient. We refer to \cite[p.177]{Giaquinta_Martinazzi_2012} or \cite[Lemma 4.12]{Han_Lin_2011}. In the parabolic case these tools have to be adapted. 
%We start with a first consequence of the H\"older continuity of \(u\), given by De Giorgi-Nash-Moser theorem and the parabolic Cacciopoli inequality.

\begin{proof}
By translation and dilation, we can assume that \( (t_0,x_0)=(0,0)\) and $R=1$. We set $Q_1 = Q_1(0,0)$.  
We can also assume that \(\rho \in (0,\frac{1}{4}]\) since if \(1/4<\rho<1\) the relation\eqref{eq:decroissance} holds with \(C=4^{-(d+2\alpha)} \iint_{Q_1} |\nabla f|^2\). We further suppose that \(\dashint_{Q_1} f=0\) since \(f-\dashint_{Q_1} f=0\) is a solution to the same equation.
Since $f\in \mathcal{C}_{par}^{\alpha/2, \alpha}((0,T) \times \Omega)$, for any $(t,x) \in Q_1(0,0)$,  
\[|f(t,x)-f(0,0)|^2\leqslant [f]_{\alpha}(|t|^{\alpha /2} +|x|^{\alpha}).\]
For \(0<\rho \leqslant 1/4\), we choose a cut-off function $\zeta\in \C^\infty_c((-1,1)\times B_1(0))$ such that \(\zeta \equiv 1\) in \(Q_\rho\) and 
\begin{equation}\label{eq:prop_zeta}
\text{Supp}(\zeta) \subset (-4\rho^2, 4\rho^2)\times B_{2\rho}(0), \quad 
0\leqslant \zeta\leqslant 1, \quad |\partial_t\zeta|\leqslant \frac{C}{\rho^2}, \quad |\nabla \zeta| \leqslant \frac{2}{\rho}.
\end{equation}
Let $f_0 = f(0,0)$. Since $f$ is bounded on $Q_1$, the function \(\varphi=\zeta^2(f-f_0)\) is in $H_0^1(Q_1) \cap L^\infty(Q_1)$. We test \eqref{eq : f campanato} with \(\varphi=\zeta^2(f-f_0)\) 
%is $0$ at $\partial_{\sqcup} Q_1$
and we obtain
\begin{multline}\label{Caccio star}
\iint_{Q_1}  f \partial_t[\zeta^2(f-f_0)]+\iint_{Q_1} A \nabla f\cdot \nabla (\zeta^2(f-f_0))+\iint_{Q_1} (B\cdot \nabla f) \zeta^2(f-f_0)\\
 = \iint_{Q_1} h \zeta^2 (f -f_0) - \iint_{Q_1} (h_2\cdot \nabla f) \zeta^2 h_1 - 2\iint_{Q_1} (h_2\cdot \nabla \zeta) (f-f_0) \zeta h_1 - \iint_{Q_1} (h_2 \cdot \nabla h_1) \zeta^2 (f-f_0).
\end{multline}
Let us denote by 

\[I = \iint_{Q_1}  f \partial_t[\zeta^2(f-f_0)], \quad II = \iint_{Q_1} A \nabla f\cdot \nabla (\zeta^2(f-f_0)),  \quad III = \iint_{Q_1} (B \cdot \nabla f) \zeta^2(f-f_0),\]

\[IV = \iint_{Q_1} h \zeta^2 (f -f_0), \quad V = \iint_{Q_1} (h_2\cdot \nabla f) \zeta^2 h_1, \]

\[VI =  2\iint_{Q_1} (h_2\cdot \nabla \zeta) (f-f_0) \zeta h_1, \quad VII =\iint_{Q_1} (h_2 \cdot \nabla h_1) \zeta^2 (f-f_0).\] 

We treat each term separately. 

\begin{multline}\label{Caccio I}
I = \iint_{Q_1} \partial_t f \zeta^2 (f-f_0) = \iint_{Q_1} \frac{1}{2} \zeta^2 \partial_t |f-f_0|^2= - \iint_{Q_1} |f-f_0|^2 \zeta \partial_t \zeta  \\
\leqslant \sup_{Q_{2\rho}} |f-f_0|^2 \iint_{Q_{2\rho}} |\partial_t \zeta| \leqslant C[f]_\alpha^2 \rho^{2\alpha} \frac{1}{\rho^2} |Q_{2\rho}| \leqslant C[f]_\alpha^2 \rho^{d+2\alpha}. 
\end{multline}
%\begin{align}\label{Caccio I}
%I &=\iint_{Q_1} f \partial_t[\zeta^2(f-f_0)] \nonumber\\
%&=\iint_{Q_1} 2 \zeta \partial_t \zeta f (f-f_0) + \iint_{Q_1} \zeta^2 f \partial_t (f-f_0) \nonumber\\
%&= \iint_{Q_1} 2 \zeta \partial_t \zeta f (f-f_0) + \iint_{Q_1} \zeta^2 \frac{1}{2} \partial_t |f-f_0|^2\\
%&= \iint_{Q_1} \frac{1}{2} \zeta^2 \partial_t |f-f_0|^2 \nonumber\\
% = - \iint_{Q_1} |f-f_0|^2 \zeta \partial_t \zeta \leqslant \sup_{Q_{2\rho}} |f-f_0|^2 \iint_{Q_{2\rho}} |\partial_t\zeta|.
%\end{align}
For the second term, we use Young's inequality. 

\begin{multline}\label{Caccio II}
II = \iint_{Q_1} A \nabla f\cdot \nabla (\zeta^2(f-f_0)) =  \iint_{Q_R} A \nabla f\cdot \nabla f \zeta^2 + 2 \iint_{Q_1} (A \nabla f \cdot \nabla \zeta) \zeta (f-f_0) \\
\geqslant \iint_{Q_{2\rho}} \lambda |\nabla f|^2 - \delta \iint_{Q_{2\rho}} \Lambda^2 |\nabla f|^2 - C_\delta \sup_{Q_{2\rho}} |f-f_0|^2 \iint_{Q_{2\rho}} |\nabla \zeta|^2\\
\geqslant  \iint_{Q_{2\rho}} \lambda |\nabla f|^2 - \delta \iint_{Q_{2\rho}} \Lambda^2 |\nabla f|^2 - C_\delta [f]_\alpha^2 \rho^{d+2\alpha}. 
\end{multline}
$III$ and $V$ can be dealt with using the same Young's inequality,

\begin{multline}\label{Caccio III}
III = \left|\iint_{Q_1}( B \cdot \nabla f) \zeta^2(f-f_0)\right| \leqslant \delta \iint_{Q_{2\rho}} |\nabla f|^2 \zeta^2 + C_\delta \beta^2 \sup_{Q_{2\rho}} |f-f_0|^2 \iint_{Q_{2\rho}} 1\\
\leqslant \delta \iint_{Q_{2\rho}} |\nabla f|^2 \zeta^2 + C_\delta \beta^2 [f]_\alpha^2 \rho^{d+2+2\alpha}.
\end{multline}

\begin{multline}\label{Caccio V}
V = \left|\iint_{Q_1}( h_2\cdot \nabla f ) \zeta^2 h_1 \right| \leqslant \delta  \iint_{Q_{2\rho}} |\nabla f|^2 \zeta^2 + C_\delta \|h_1\|_\infty^2 \|h_2\|_\infty^2 \iint_{Q_{2\rho}} 1\\
\leqslant \delta  \iint_{Q_{2\rho}} |\nabla f|^2 \zeta^2 +  C_\delta \|h_1\|_\infty^2 \|h_2\|_\infty^2 \rho^{d+2},  
\end{multline}
Moreover,

\begin{equation}\label{Caccio IV}
IV =\iint_{Q_1} h\zeta^2 (f -f_0) \leqslant \sup_{Q_{2\rho}} |f-f_0|\iint_{Q_{2\rho}} |h| \leqslant [f]_\alpha \|h\|_{\infty} \rho^{d+2+\alpha},
\end{equation}

\begin{multline}\label{Caccio VI}
VI =  2\iint_{Q_1} h_2\cdot \nabla \zeta (f-f_0) \zeta h_1 \leqslant \|h_1\|_\infty \|h_2\|_\infty \sup_{Q_{2\rho}} |f-f_0| \iint_{Q_{2\rho}} |\nabla \zeta|\\
\leqslant \|h_1\|_\infty \|h_2\|_\infty [f]_\alpha \rho^{d+1+\alpha},
\end{multline}

\begin{multline}\label{Caccio VII}
VII =\iint_{Q_1} h_2 \cdot \nabla h_1 \zeta^2 (f-f_0) \leqslant \|h_2\|_\infty \|\nabla h_1\|_\infty \sup_{Q_{2\rho}} |f-f_0| \iint_{Q_{2\rho}} 1\\
\leqslant \|h_2\|_\infty \|\nabla h_1\|_\infty [f]_\alpha \rho^{d+2+\alpha}.
\end{multline}
Gathering all the estimates \eqref{Caccio I}, \eqref{Caccio II}, \eqref{Caccio III},  \eqref{Caccio V}, \eqref{Caccio IV},\eqref{Caccio VI}, \eqref{Caccio VII} into \eqref{Caccio star} yields, for $\delta > 0$ small enough (depending only on $\lambda, \Lambda$),

\begin{align*}
\iint_{Q_{\rho}} |\nabla f|^2 &\leqslant C_{\lambda, \delta} \iint_{Q_{2\rho}} (\lambda -\delta (\Lambda^2+2)) |\nabla f|^2 \zeta^2 \\
&\leqslant C[f]_\alpha^2 \rho^{d+2\alpha} + C_{\delta} [f]_\alpha^2 \rho^{d+2\alpha} + C_\delta \beta^2[f]_\alpha^2\rho^{d+2+2\alpha} + C_\delta \|h_1\|_\infty^2 \|h_2\|_\infty^2 \rho^{d+2}\\
&\quad  + [f]_\alpha \|h\|_{\infty} \rho^{d+2+\alpha} +  \|h_1\|_\infty \|h_2\|_\infty [f]_\alpha \rho^{d+1+\alpha} + \|h_2\|_\infty \|\nabla h_1\|_\infty [f]_\alpha \rho^{d+2+\alpha}\\
&\leqslant C_{CM} \rho^{d+2\alpha}.
\end{align*}
%&\leqslant \sup_{Q_{2\rho}} |f-f_0|^2 \iint_{Q_{2\rho}} |\partial_t \zeta| + C_\delta \sup_{Q_{2\rho}} |f-f_0|^2 \iint_{Q_{2\rho}} |\nabla \zeta|^2 \\
%& \quad + C_\delta \beta^2\sup_{Q_{2\rho}} |f-f_0|^2 \iint_{Q_{2\rho}} 1 + C_\delta \|h_1\|_\infty \|h_2\|_\infty \iint_{Q_{2\rho}} |f-f_0|^2 + \sup_{Q_{2\rho}} |f-f_0| \iint_{Q_{2\rho}} |\mu| \\
%&\quad + \|h_1\|_\infty \|h_2\|_\infty \sup_{Q_{2\rho}} |f-f_0| \iint_{Q_{2\rho}} |\nabla \zeta| + \|h_2\|_\infty \|\nabla h_1\|_\infty \sup_{Q_{2\rho}} |f-f_0| \iint_{Q_{2\rho}} 1\\
%& \leqslant 4^\alpha \rho^{2\alpha} \frac{\omega_d \rho^{d+2}}{\rho^2} + C_\delta 4^\alpha \rho^{2\alpha} \frac{\omega_d \rho^{d+2}}{\rho^2} + C_\delta \beta^2 4^\alpha \rho^{2\alpha} \omega_d \rho^{d+2}+ C_\delta \|h_2\|_\infty \|\nabla h_1\|_\infty\rho^{2\alpha} \omega_d \rho^{d+2}\\
%&\quad + 2^\alpha\rho^{\alpha} \|\mu\|_{\mathcal{L}} \rho^{d+2\alpha} + \|h_1\|_\infty \|h_2\|_\infty 2^\alpha \rho^{\alpha} \frac{\omega_d \rho^{d+2}}{\rho} + \|h_2\|_\infty , \|\nabla h_1\|_\infty 2^\alpha \rho^\alpha \omega_d \rho^{d+2}  \\
%&\leqslant C_{\alpha, \delta, h_1, h_2, \mu, \beta,R} \rho^{d+2\alpha}.
This proves the proposition. 
\end{proof}

\textbf{Remark :} Although it is not useful to our present paper, \autoref{prop:grad_spat_Morrey} remains true if we replace the function $h\in L^\infty(Q_R)$ by a function $\mu \in \mathcal{L}^{1,d+\beta}$ for $\beta \in (0,1)$. 

\subsubsection{The Poincaré inequality}

In this subsection, we prove a Poincaré inequality for the spatial gradient. The proof is similar to the proof of \cite[Lemma 2.6]{Yin_1997} and \cite[Lemma 4]{S81}. In our case, this lemma will be applied to the function $v_2 = v-w$, where $w$ is the ``heat''-extension of $v$. 

\begin{lemma}\label{prop : Poincare}
Let $(t_0, x_0)\in (0,+\infty) \times \Omega$ and $R >0$ be such that $Q_R = Q_R(t_0,x_0) \subset \subset (0,T) \times \Omega$. Let $A : \mathbb{R}^d  \rightarrow \mathbb{M}_d(\mathbb{R})$ and $B : \mathbb{R}^d \rightarrow \mathbb{R}^d$ be measurable functions satisfying, 
\[0< \lambda I_d \leqslant A(x) \leqslant \Lambda I_d, \quad |B(x)| \leqslant \beta\]
for some positive constants $\lambda, \Lambda, \beta$. Let $\alpha \in (0,1)$ and $\mu \in \mathcal{L}^{1, d+2\alpha}(Q_R(t_0,x_0))$, $g,h\in \mathcal{C}^0(\overline{B_R(x_0)})\cap \mathcal{C}^1(B_R(x_0))$ and $\mu_1, \mu_2 \in L^2(Q_R(t_0,x_0);\mathbb{R}^d)$ be such that 

\begin{equation}\label{cond mui}
\sup_{Q_r(t_1,x_1) \subset Q_R(t_0,x_0)} \frac{1}{r^d} \iint_{Q_r(t_1, x_1)} |\mu_i|^2 < C_i
\end{equation}
for positive constants $C_1$ and $C_2$. 

Let $(t_1,x_1)\in Q_{R/2}(t_0,x_0)$. 
Let $f\in H^1((0,T)\times \Omega)$ be a weak solution of the equation 

\[\partial_t f - \text{div}(A\nabla f) + B\cdot\nabla f = -\text{div}((g-g(x_1)) \mu_1)+ h\cdot \mu_2 + \mu \quad \text{in } Q_R(t_0,x_0),\]
meaning that for any test function $\varphi\in L^2((t_0-R^2, t_0);H_0^1(B_R(x_0)) \cap L^\infty(Q_R)$,
\begin{multline}\label{eq poincare}
\iint_{Q_R} \partial_t f(t,x) \varphi(t,x) dxdt + \iint_{Q_R}A(x)\nabla f(t,x)\cdot \nabla \varphi(t,x) dxdt\\ + \iint_{Q_R}B(x)\cdot\nabla f(t,x) \varphi(t,x) dxdt
= \iint_{Q_R} (g(t,x)-g(x_1))\mu_1(t,x)\cdot \nabla \varphi(t,x) dxdt\\
 +\iint_{Q_R} h(t,x)\cdot \mu_2(t,x)\varphi(t,x) dxdt + \iint_{Q_R} \mu(t,x\varphi(t,x) dxdt.  
\end{multline}
There exists a constant $C_P > 0$ such that for any  $Q_r(t_1, x_1) \subset Q_{R/2}(t_0, x_0)$, 
\begin{multline*}
\iint_{Q_r(t_1, x_1)} |f-\{f\}_{r,t_1,x_1}|^2\\
\leqslant C_P r^2\left(\iint_{Q_{2r}(t_1, x_1)} |\nabla f|^2 + C_1 [g]_1^2 r^{d+2} +C_2\|h\|_\infty r^{d+2} + \|\mu\|^2_{\mathcal{L}^{1,d+2\alpha}} r^{d+2\alpha}\right).
\end{multline*}
\end{lemma}

\begin{proof}
Let $\overline{Q_r(t_1, x_1)}\subset Q_R(t_0,x_0)$ with $2r < R$. We will write $Q_r = Q_r(t_1, x_1), Q_{2r}= Q_{2r}(t_1, x_1), B_r = B_r(x_1)$ and $B_{2r}= B_{2r}(x_1)$.

\smallskip

Let $\zeta : \mathbb{R}^d \rightarrow \mathbb{R}$ be a smooth cut-off function such that $0\leqslant \zeta \leqslant 1$ with $\zeta = 1$ on $B_r$ and $\zeta = 0$ outside $B_{2r}$, and $|\nabla \zeta | \leqslant 2/r$. We insist that $\zeta$ depends only on the spatial variable. 
Denote by 

\[f_r^\zeta = \frac{\iint_{Q_{2r}} f(t,x)\zeta(x) dx dt}{\iint_{Q_{2r}} \zeta(x) dx dt} = \frac{\iint_{Q_{2r}} f(t,x)\zeta(x) dx dt}{4 r^2 \int_{B_{2r}} \zeta(x) dx},\]

\[f_{r,t}^{\zeta} = \frac{\int_{B_{2r}} f(t,x) \zeta(x) dx}{\int_{B_{2r}} \zeta(x) dx} \; \text{for } t\in (t_1-4r^2, t_1).\]
From \cite[Theorem 1.5]{FKS82}, we deduce the following Poincaré-Wirtinger inequality : there exists a constant $C_\star$ such that for any $t \in (t_1 -r^2 , t_1)$,  

\begin{multline}\label{Poincare ref}
\int_{B_r(x_0)} \left|f(t,x) - f_{r,t}^\zeta\right|^2 dx \leqslant \int_{B_{2r}(x_0)} \left|f(t,x) - f_{r,t}^\zeta\right|^2 \zeta \, dx\\
\leqslant C_\star r^2\int_{B_{2r}(x_0)} |\nabla f(t,x)|^2 \zeta \, dx \leqslant C_\star r^2 \int_{B_{2r}(x_0)} |\nabla f (t,x)|^2 \, dx.
\end{multline}

Let $s,t\in (t_1-4r^2, t_1)$. Denote by $\varphi (\tau, x) = \zeta(x) 1_{(s,t)}(\tau) (f_{r,t}^\zeta - f_{r,s}^\zeta)$. The function $\varphi$ belongs to $L^2((t_1-R^2,t_1);H_0^1(B_R(x_1))) \cap L^\infty(Q_R)$. Using $\varphi$ as a test function in \eqref{eq poincare}, we obtain the following equation : 

\begin{multline}\label{Poinc 0}
\iint_{Q_{2r}} \partial_t f \varphi + \iint_{Q_{2r}} A\nabla f \cdot \nabla \varphi + \iint_{Q_{2r}} B\cdot\nabla f  \varphi\\ = \iint_{Q_{2r}} (g-g(x_1)) \mu_1 \cdot \nabla \varphi + \iint_{Q_{2r}} h \mu_2 \varphi + \iint_{Q_{2r}}\mu \,\varphi. 
\end{multline}
Let us denote 

\[I =\iint_{Q_{2r}} \partial_t f\varphi, \quad II =  \iint_{Q_{2r}} A\nabla f \cdot \nabla \varphi , \quad III=\iint_{Q_{2r}} B\cdot \nabla f \varphi,\]

\[IV = \iint_{Q_{2r}} (g-g(x_1)) \mu_1 \cdot \nabla \varphi , V = \iint_{Q_{2r}} h \mu_2 \varphi, \quad VI = \iint_{Q_{2r}}\mu \,\varphi.\] 
We treat each term separately.

\begin{align}\label{Poinc 1}
I= \iint_{Q_{2r}} \partial_t f\varphi  &=  \int_{B_{2r}}\int_s^t \partial_t f(\tau,x) \zeta(x) (f_{r,t}^\zeta - f_{r,s}^\zeta) d\tau dx \nonumber\\
&= \int_{B_{2r}} (f_{r,t}^\zeta - f_{r,s}^\zeta) (f(t,x)\zeta (x) - f(s,x) \zeta (x))dx \nonumber\\
&= (f_{r,t}^\zeta - f_{r,s}^\zeta) \left(f_{r,t}^\zeta \int_{B_{2r}}\zeta - f_{r,s}^\zeta \int_{B_{2r}}\zeta\right) \nonumber\\
&\geqslant (f_{r,t}^\zeta - f_{r,s}^\zeta)^2 r^d \omega_d. 
\end{align} 
%which resumes as 
%
%\begin{equation}\label{Poinc 1}
%I =\iint_{Q_{2r}} \partial_t f\varphi \geqslant (f_{r,t}^\zeta - f_{r,s}^\zeta)^2 r^d \omega_d.
%\end{equation}
Let $\delta >0$ be small enough. Using the condition $A\leqslant \Lambda I_d$ and Young's inequality,

\begin{multline}\label{Poinc 2}
II =\iint_{Q_{2r}} A(x)\nabla f(\tau,x) \cdot \nabla \varphi(\tau,x) dxd\tau \leqslant \delta (f_{r,t}^\zeta - f_{r,s}^\zeta)^2\iint_{Q_{2r}}|\nabla \zeta|^2 + C_\delta \iint_{Q_{2r}} \Lambda^2 |\nabla f|^2\\
\leqslant \delta (f_{r,t}^\zeta - f_{r,s}^\zeta)^2 \left(\frac{2}{r}\right)^2 (2r)^{d+2} + C_\delta \iint_{Q_{2r}}|\nabla f|^2.
\end{multline}

\begin{multline}\label{Poinc 3}
III =\iint_{Q_{2r}} B\nabla f \cdot \varphi \leqslant \delta (f_{r,t}^\zeta - f_{r,s}^\zeta)^2\iint_{Q_{2r}}|\zeta|^2 +C_\delta \beta^2 \iint_{Q_{2r}} |\nabla f|^2\\
\leqslant  \delta (f_{r,t}^\zeta - f_{r,s}^\zeta)^2 r^{d+2}  + C_\delta \beta^2 \iint_{Q_{2r}}|\nabla f|^2
\end{multline}
By \eqref{cond mui},
\begin{multline}\label{Poinc 5}
V = \iint_{Q_{2r}} h\cdot \mu_2 \varphi \leqslant \delta (f_{r,t}^\zeta - f_{r,s}^\zeta)^2\frac{1}{r^2}\iint_{Q_{2r}}|\zeta|^2 + C_\delta r^2  \|h\|_\infty^2\iint_{Q_{2r}}  |\mu_2|^2 \\
\leqslant \delta (f_{r,t}^\zeta - f_{r,s}^\zeta)^2 r^d + C_\delta C_2 \|h\|_\infty^2 4r^{d+2} \\  
\end{multline}
From the mean value theorem and Young's inequality and \eqref{cond mui},

\begin{multline}\label{Poinc 4}
IV = \iint_{Q_{2r}} (g-g(x_1)) \mu_1 \cdot \nabla \varphi\leqslant\delta (f_{r,t}^\zeta - f_{r,s}^\zeta)^2\iint_{Q_{2r}}|\nabla \zeta|^2 + C_\delta [g]_1^2 4r^2\iint_{Q_{2r}}  |\mu_1|^2\\
\leqslant \delta (f_{r,t}^\zeta - f_{r,s}^\zeta)^2 r^d + C_\delta C_1 [g]_1^2 4r^{d+2}. \\  
\end{multline}
Finally, 

\begin{multline}\label{Poinc 6}
VI =\iint_{Q_{2r}}\mu \,\varphi \leqslant |f_{r,t}^\zeta - f_{r,s}^\zeta|\iint_{Q_{2r}}|\mu| \leqslant |f_{r,t}^\zeta - f_{r,s}^\zeta| \|\mu\|_{\mathcal{L}^{1,d+2\alpha}} r^{d+2\alpha}\\ \leqslant \delta (f_{r,t}^\zeta - f_{r,s}^\zeta)^2r^d +C_\delta  \|\mu\|_{\mathcal{L}^{1,d+2\alpha}}^2r^{d+2+2\alpha}. 
\end{multline}
Combining \eqref{Poinc 1}, \eqref{Poinc 2}, \eqref{Poinc 3} and \eqref{Poinc 4}, we find that there exists a constant $C>0$ depending on $\lambda, \beta, \Lambda$ such that
% depending on $\delta, \alpha, \lambda, \Lambda, \beta, d, \|\mu\|_{\mathcal{L}^\alpha}, [g]_1, \|h\|_\infty, C_1,C_2$ such that

\begin{equation}\label{Poinc star}
(1-5\delta)(f_{r,t}^\zeta - f_{r,s}^\zeta)^2r^d \leqslant C \left[\iint_{Q_{2r}} |\nabla f|^2 + \|\mu\|_{\mathcal{L}^{1, d+2\alpha}}^2 r^{d+2+2\alpha} + C_1[g]_1^2r^{d+2} + C_2\|h\|_\infty r^{d+2}\right] .
\end{equation}
Recall that, for any $c\in \mathbb{R}$, one has 
\begin{equation}\label{eq : variance}
\frac{1}{|Q_r|}\iint_{Q_r} |f - \{f\}_{r,t_1,x_1}|^2 = \frac{1}{|Q_r|}\iint_{Q_r} |f - c|^2 + |\{f\}_{r,t_1,x_1} - c|^2 \leqslant \frac{1}{|Q_r|}\iint_{Q_r} |f - c|^2.
\end{equation}
%$\iint_{Q_r} |f - c|^2 \geqslant \iint_{Q_r} |f - \{f\}_r|^2$ for any constant $c$, 
Applying \eqref{eq : variance} to $c= f_r^{\zeta}$ 
and recalling \eqref{Poincare ref}, one has 
\begin{align}\label{Poinc 7}
\iint_{Q_r} |f - \{f\}_{r,t_1,x_1}|^2 &\leqslant \iint_{Q_r} |f(x,t) - f_r^\zeta|^2 dx\, dt \nonumber\\
&\leqslant 2 \iint_{Q_r} |f(x,t) - f_{r,t}^\zeta|^2 dx\, dt + 2\iint_{Q_r} |f_{r,t}^\zeta - f_r^\zeta|^2 dx \,dt \nonumber\\
&\leqslant 2 C_\star r^2 \int_{t_1 -r^2}^{t_1} \int_{B_{2r}(x_1)} |\nabla f (x,t)|^2 dx dt +2\iint_{Q_{2r}} |f_{r,t}^\zeta - f_r^\zeta|^2 dx\, dt \nonumber\\
&\leqslant  2 Cr^2 \iint_{Q_{2r}} |\nabla f(t,x)|^2dx\, dt + 2|B_{2r}(x_1)| \int_{t_1 -4r^2}^{t_1}  |f_{r,t}^\zeta - f_r^\zeta|^2 dt.
\end{align} 
%\leqslant 2 \int_{t_1 -r^2}^{t_1} \int_{B_r(x_1)} |f(x,t) - f_{r,t}^\zeta|^2 dx\, dt +  2\iint_{Q_r} |f_{r,t}^\zeta - f_r^\zeta|^2 dx \,dt \nonumber\\
%or 
%
%\begin{equation}\label{Poinc 7}
%\iint_{Q_r} |f - \{f\}_{r,t_1,x_1}|^2 \leqslant  2 C r^2 \iint_{Q_{2r}} |\nabla f|^2 + 2\iint_{Q_r} |f_{r,t}^\zeta - f_r^\zeta|^2.
%\end{equation} 
Considering that 
\begin{align*}
\frac{1}{4r^2} \int_{t_1-4r^2}^{t_1} f_{r,s}^\zeta ds = \frac{\int_{t_1-4r^2}^{t_1} \int_{B_{2r}(x_1)} f (s,x) \zeta(x) dx ds}{\int_{t_1-4r^2}^{t_1} \int_{B_{2r}(x_1)} \zeta(x) dx ds} = f_r^\zeta,
\end{align*}
we deduce from Jensen's inequality and \eqref{Poinc star} that 
\begin{align}\label{Poinc 8}
|B_{2r}(x_1)| \int_{t_1 -4r^2}^{t_1}  |f_{r,t}^\zeta - f_r^\zeta|^2 dt &= |B_{2r}(x_1)|\int_{t_1 -4r^2}^{t_1}\left|\frac{1}{4r^2} \int_{t_1-4r^2}^{t_1} (f_{r,s}^\zeta - f_{r,t}^\zeta)ds \right|^2dt \nonumber\\
&\leqslant |B_{2r}(x_1)| \int_{t_1 -4r^2}^{t_1} \frac{1}{4r^2}\int_{t_1 -4r^2}^{t_1} |f_{r,t}^\zeta - f_{r,s}^\zeta|^2 ds dt \nonumber\\
&\leqslant C_P r^2 \left[\iint_{Q_{2r}} |\nabla f|^2 + \|\mu\|_{\mathcal{L}^{1, d+2\alpha}}^2 r^{d+4\alpha}\right. \nonumber\\
& \quad \quad \quad \quad \quad \left. + C_1[g]_1^2r^{d+2} + C_2\|h\|_\infty r^{d+2}\right].
\end{align}

Combining \eqref{Poinc 7} and \eqref{Poinc 8} concludes. 
\end{proof}

\subsection{Proof of inner regularity}

We prove the following : 
 
\begin{proposition}\label{inner regularity}
For all $T >0$, there exists a positive number $\alpha \in (0,1)$ depending only on $T$ such that, for any $Q_R(t_0, x_0)\subset (0,T) \times \Omega$, the pair $(u,v)$ belongs to $C_{par}^{\alpha/2,\alpha}(Q_R(t_0, x_0))$.
\end{proposition}

\begin{corollary}\label{cor : regularité infini}
The pair $(u, v)$ belongs to $\mathcal{C}^\infty ((0,+\infty) \times \Omega)$. 
\end{corollary}

\begin{proof}[Proof of \autoref{cor : regularité infini}]
Let $T >0$. According to \autoref{inner regularity}, there exists a real number $\alpha$ such that $(\overline{u}, \overline{v})\in \mathcal{C}_{par,loc}^{\alpha/2, \alpha} ((0,T) \times \Omega)$, so $(u,v)\in \mathcal{C}_{par,loc}^{\alpha/2, \alpha} ((0,T) \times \Omega)$. 

\smallskip

Since $u$ is a weak solution of 

\[\partial_t u - \text{div}((\eta_\varepsilon + v^2) \nabla u) = 0,\]
according to parabolic Schauder Theorem (see \cite[Theorem 8.11]{K96}), the function $u$ is in $\mathcal{C}_{par, loc}^{1 +\alpha/2, 2 +\alpha}((0,T)\times \Omega)$ (see \eqref{def : Ckl}), which means that $\nabla \overline{u} \in \mathcal{C}_{par, loc}^{1+ \alpha/2, 1 +\alpha/2}((0,T)\times \Omega)$. Since $v$ is a weak solution of 

\[\partial_t v - \varepsilon \Delta v = \frac{v-1}{4\varepsilon} - v|\nabla u|^2,\]
with a source term in $\mathcal{C}_{par, loc}^{1+ \alpha/2, 1 +\alpha/2}((0,T)\times \Omega)$, according to Schauder Theorem again, the function $v$ is in $\mathcal{C}_{par, loc}^{1 +\alpha/2, 2 +\alpha}((0,T)\times \Omega)$. We can conclude from a classical bootstrap argument that the pair $(u,v)$ belongs to $\mathcal{C}^\infty((0,T) \times \Omega)$. 
\end{proof}

\begin{proof}[Proof of \autoref{inner regularity}]
Let $2R > 0$ and $Q_{2R}(t_0, x_0) \subset \subset (0,T) \times \Omega$ be a parabolic cylinder inside $\Omega$.

\smallskip

According to De Giorgi-Nash-Moser Theorem (see \cite[Theorem 3.2.3]{I26}, \cite[Section IV, Theorem 1.2]{DB93})
%(see \cite[Theorem 2.12]{G18}, \cite[Theorem 3]{GIMV19})
there exists a positive real number $\alpha \in (0,1)$ such that $u \in C_{par, loc}^{\alpha/2, \alpha}((0,T) \times \Omega)$. We now prove that $v \in C_{par, loc}^{\alpha/2, \alpha}((0,T) \times \Omega)$. 
Let $(t_0,x_0) \in (0,T) \times \Omega$ and $R >0$ such that $Q_{2R}(t_0,x_0) \subset (0,T)\times \Omega$. According to \autoref{prop:grad_spat_Morrey}, the function $|\nabla u|^2$ belongs to $\mathcal{L}^{1,d+2\alpha}(Q_{2R}(t_0,x_0))$. By setting 

\[\mu = \frac{1-v}{4\varepsilon} + v |\nabla u|^2,\]
the equation of $v$ in \eqref{eq : eq in v} rewrites as 

\[\partial_t v - \varepsilon \Delta v = \mu \quad \quad \text{in } Q_{2R}(t_0,x_0).\]
We claim that there exists a constant $C_1$ such that for any parabolic cylinder $Q_r(t_1,x_1) \subset Q_{R}(t_0,x_0)$ , 

\begin{equation}\label{eq : campanato inner}
\iint_{Q_r(t_1, x_1)} |v - \{v\}_{r,t_1,x_1}|^2 \leqslant C_1 r^{d+2+2\alpha}.
\end{equation}

Suppose that \eqref{eq : campanato inner} is true. We deduce that for any $Q_r(t_1,x_1) \subset Q_{R/2}(t_0,x_0)$, according to Cauchy-Schwarz inequality, 

\begin{multline*}
\iint_{Q_r(t_1,x_1)} |v - \{v\}_{r,t_1,x_1}| \leqslant |Q_r|^{1/2} \left(\iint_{Q_r(t_1,x_1)}  |v - \{v\}_{r,t_1,x_1}|^2 \right)^{1/2}\\
\leqslant \omega_d ^{1/2} r^{(d+2)/2} \sqrt{C_1} r^{(d+2+2\alpha)/2} \leqslant C(\alpha, d) r^{d+2+\alpha}.
\end{multline*}
According to \cite[Lemma 4.3]{L96}, we conclude that $[v]_{\alpha, Q_R(t_0,x_0)} \leqslant C(\alpha, d)$. Combining with the fact that $v\in L^\infty(Q_R(t_0,x_0))$, we deduce that $v\in C_{par}^{\alpha/2, \alpha}(Q_R(t_0,x_0))$, as desired. 

\smallskip

We now turn to the proof of \eqref{eq : campanato inner}. Fix $t_1, x_1 \in Q_R(t_0,x_0)$, so that $Q_R(t_1, x_1) \subset Q_{2R}(t_0, x_0)$. In order to prove \eqref{eq : campanato inner}, we prove the following intermediate inequality : there exists a constant $C_2$ such that for any $\rho < r \leqslant R$, one has 

\begin{equation}\label{inner star}
\iint_{Q_\rho(t_1, x_1)} |v -\{v\}_{\rho,t_1, x_1}|^2 \leqslant C_2\left[\left(\frac{\rho}{r}\right)^{d+4} \iint_{Q_r (t_1, x_1)} |v - \{v\}_{r,t_1, x_1}|^2 + r^{d+2+2\alpha} \right].
\end{equation} 
Using the iteration lemma \cite[Lemma 4.3]{L96} and the fact that $0\leqslant v\leqslant 1$, we deduce that 
\begin{multline*}
\iint_{Q_r(t_1,x_1)} |v - \{v\}_{r,t_1,x_1}|^2\leqslant C_2 r^{d+2+2\alpha} \left(\frac{1}{R^{d+2+2\alpha}}\iint_{Q_{R}(t_1,x_1)} |v - \{v\}_{2r, t_1, x_1}|^2 + 1\right)\\
\leqslant r^{d+2+2\alpha} \left(\frac{1}{R^{d+2+2\alpha}}4R^{d+2} + C\right),
\end{multline*}
which corresponds to \eqref{eq : campanato inner}. 

\medskip

Fix $0 < \rho < r\leqslant R$. Let $w$ be the unique weak solution in $H^1(Q_{2r}(t_1,x_1))$ of the equation 

\[\left\{\begin{array}{lll}
\partial_t w - \varepsilon \Delta w &= 0, &\text{on } Q_{2r}(t_1, x_1),\\
w =v& & \text{on } \partial_{\sqcup} Q_{2r}(t_1,x_1),\\
\end{array}\right.\]
meaning that for any test function $\varphi\in H_0^1(Q_{2r}(t_1,x_1))$, 
\[\iint_{Q_{2r}(t_1,x_1)} \partial_t w \varphi + \varepsilon \iint_{Q_{2r}(t_1,x_1)} \nabla w\cdot \nabla \varphi =0.\]
According to the maximum principle (see \cite[Lemma 3.15]{L96}), $|w| \leqslant v\leqslant 1$. Moreover, $w$ satisfies the following Campanato estimate (see \cite[Lemma 4.5]{L96}) : there exists $C >0$ such that for any $\rho < r$,  

\begin{equation}\label{campan1}
\iint_{Q_\rho(t_1,x_1)} |w- \{w\}_{\rho , t_1, x_1}|^2 \leqslant C \left(\frac{\rho}{r}\right)^{d+4}\iint_{Q_r(t_1,x_1)} |w- \{w\}_{r , t_1, x_1}|^2.
\end{equation}

Let $v_2= v -w$. Notice that for a.e $t\in (t_1 - (2r)^2, t_1)$, the function $v_2(t,\cdot)$ is zero at $\partial B_{2r}(x_1)$. The function $v_2 \in L^\infty(Q_{2r}(t_1,x_1))\cap L^2((t_1-4r^2,t_1);H_0^1(B_{2r}(x_0))) \cap H^1((t_1-4r^2,t_1);L^2(B_{2r}(x_0)))$ satisfies for any test function $\varphi\in L^2((t_1-4r^2,t_1);H_0^1(B_{2r}(x_0))) \cap H^1((t_1-4r^2,t_1);L^2(B_{2r}(x_0)))$, 
%\begin{equation}\label{camp inner v eq v2}
%\left\{\begin{array}{ll}
%\partial_t v_2 - \varepsilon \Delta v_2 = \mu, &\text{on } Q_{2r}(t_1, x_1)\\
%v_2 =0 & \text{on } \partial_{\sqcup} Q_{2r}(t_1,x_1)\\
%\end{array}\right.
%\end{equation}
\begin{equation}\label{camp inner v eq v2}
\iint_{Q_{2r}(t_1,x_1)}\partial_t v_2 \varphi + \varepsilon \iint_{Q_{2r}(t_1,x_1)} \nabla v_2\cdot \nabla \varphi =\iint_{Q_{2r}(t_1,x_1)}\mu \varphi
\end{equation}
Taking $v_2(t,\cdot)$ as a test function in \eqref{camp inner v eq v2}, we obtain 

\begin{equation}\label{inner camp 1}
\iint_{Q_{2r}(t_1, x_1)} \partial_t v_2 v_2 + \iint_{Q_{2r}(t_1,x_1)} \varepsilon |\nabla v_2|^2 = \iint_{Q_{2r}(t_1,x_1)} \mu v_2.
\end{equation}
Observe that $v_2 (t_1 -4r^2, \cdot) = 0$ on $B_{2r}(x_1)$, hence 
\begin{multline}\label{inner camp 2}
\iint_{Q_{2r}(t_1, x_1)} \partial_t v_2 v_2 = \int_{B_{2r}(x_1)} \int_{t_1 -4r^2}^{t_1} \partial_t v_2 v_2 \\
= \int_{B_{2r}(x_1)} \left[\frac{1}{2} |v_2 (t_1, x)|^2 -\frac{1}{2}|v_2 (t_1 -4r^2, x)|^2\right]dx = \int_{B_{2r}(x_1)} \frac{1}{2} |v_2 (t_1, x)|^2  \geqslant 0.
\end{multline}

On the other hand, using that $|v_2| = |v -w|\leqslant 2$,

\begin{equation}\label{inner camp 3}
\iint_{Q_{2r}(t_1,x_1)} \mu v_2 \leqslant 2\iint_{Q_{2r}(t_1,x_1)} |\mu| \leqslant C\|\mu\|_{\mathcal{L}^{1, d+2\alpha}} r^{d+2\alpha}.
\end{equation}  
Injecting \eqref{inner camp 2} and \eqref{inner camp 3} into \eqref{inner camp 1} yields

\begin{equation}\label{inner camp 4}
\iint_{Q_{2r}(t_1,x_1)} |\nabla v_2|^2 \leqslant \frac{1}{\varepsilon} C\|\mu\|_{\mathcal{L}^{1, d+2\alpha}} r^{d+2\alpha}.
\end{equation}

Now for any $\rho < r$, using twice that $(a+b)^2 \leqslant 2(a^2 +b^2)$, twice \autoref{prop : Poincare}, \eqref{campan1} and \eqref{inner camp 4}, there exist constants $C_P, C>0$ such that

\begin{align*}
&\iint_{Q_{\rho}(t_1,x_1)} |v -\{v\}_{\rho, t_1, x_1}|^2\\
&\leqslant 2\iint_{Q_{\rho}(t_1,x_1)} |w -\{w\}_{\rho, t_1, x_1}|^2 + 2 \iint_{Q_{\rho}(t_1,x_1)} |v_2 -\{v_2\}_{\rho, t_1, x_1}|^2\\
&\leqslant  2C\left(\frac{\rho}{r}\right)^{d+4}\iint_{Q_{r}(t_1,x_1)} |w -\{w\}_{r, t_1, x_1}|^2 + 2C_P\rho^2 \iint_{Q_{2\rho}(t_1,x_1)} |\nabla v_2|^2 + 2C_P\rho^{d+2\alpha}\\
&\leqslant 4C\left(\frac{\rho}{r}\right)^{d+4}\iint_{Q_{r}(t_1,x_1)} |v -\{v\}_{r, t_1, x_1}|^2 + 4 C\left(\frac{\rho}{r}\right)^{d+4}\iint_{Q_{r}(t_1,x_1)} |v_2 -\{v_2\}_{r, t_1, x_1}|^2 \\
&\quad + 2C_Pr^2 \iint_{Q_{2r}(t_1,x_1)} |\nabla v_2|^2 + 2C_Pr^{d+2+2\alpha} \\
&\leqslant 4C\left(\frac{\rho}{r}\right)^{d+4}\iint_{Q_{r}(t_1,x_1)} |v -\{v\}_{r, t_1, x_1}|^2  + 4C_P r^2 \iint_{Q_{2r}(t_1,x_1)} |\nabla v_2|^2\\
&\quad +2C_Pr^2 \iint_{Q_{2r}(t_1,x_1)} |\nabla v_2|^2 + 2C_Pr^{d+2+2\alpha}\\
&\leqslant 4C\left(\frac{\rho}{r}\right)^{d+4}\iint_{Q_{r}(t_1,x_1)} |v -\{v\}_{r, t_1, x_1}|^2 + Cr^{d+2+2\alpha}.
\end{align*}
which proves \eqref{inner star} and concludes the proof of \autoref{inner regularity}. 
\end{proof}

\subsection{Regularity at the boundary}

In this section, we prove the global regularity property, which corresponds to the second part of \autoref{Theorem 2}.

\subsubsection{Description of the extension in time}

We start by extending the pair $(u,v)$ on $(-1,+\infty) \times \Omega$ by the time-constant pair $(u_0,v_0)$ on $(-1,0]$. We consider $(\overline{u}, \overline{v})$ the extension of the pair $(u,v)$ on the negative time, defined by

\begin{equation}\label{extension u,v def}
(\overline{u}(t,x), \overline{v}(t,x)) = \left\{\begin{array}{ll}
(u(t-1,x),v(t-1,x)) & \text{if } t >1\\
(u_0(x), v_0(x)) & \text{if } t\leqslant 1.\\
\end{array}
\right.
\end{equation} 
The pair $(\overline{u}, \overline{v})\in \mathcal{F}$ satisfies the following system : 

\begin{equation}\label{extension uv eq}
\left\{\begin{array}{ll}
\partial_t \overline{u} - \text{div}((\eta_\varepsilon + \overline{v}^2) \nabla \overline{u}) = - \text{div}((\eta_\varepsilon + v_0^2) \nabla u_0) 1_{\{t\leqslant 1\}}, &\text{in } \mathcal{D}'((0,\infty)\times\Omega)\\
\partial_t \overline{v} - \varepsilon \Delta \overline{v} + \frac{\overline{v}-1}{4\varepsilon} + \overline{v}|\nabla \overline{u}|^2 = (- \varepsilon \Delta v_0 + \frac{v_0-1}{4\varepsilon} + v_0|\nabla u_0|^2)1_{\{t\leqslant 1\}},  &\text{in } \mathcal{D}'((0,\infty)\times \Omega)\\
\overline{v}(t,x) = 1 & \text{on } (0,\infty) \times \partial\Omega \\
\overline{u}(t,x) = g(t-1,x)1_{\{t > 1\}} + g(0,x) 1_{\{t\leqslant 1\}} & \text{on } (0,\infty) \times \partial\Omega\\
\end{array}\right. 
\end{equation}

We will adopt the following notations : 

\begin{equation}\label{def U0 V0}
U_0 = - \text{div}((\eta_\varepsilon + v_0^2) \nabla u_0) 1_{\{t\leqslant 1\}}, \quad V_0 = \left(- \varepsilon \Delta v_0 + \frac{v_0-1}{4\varepsilon} + v_0|\nabla u_0|^2\right)1_{\{t\leqslant 1\}}.
\end{equation}

Notice that if we suppose that $(u_0, v_0) \in \mathcal{C}^{1,1}(\Omega)$, then $U_0, V_0\in L^\infty((0,+\infty) \times \Omega)$. 

\subsubsection{Description of the reflexion method}

We extend the pair $(\overline{u}, \overline{v})$ on a neighbourhood of $\Omega$. Studying the regularity of $(\overline{u}, \overline{v})$ at the boundary $\partial\Omega$ then amounts to study the inner regularity of the extended pair. The construction of this extension relies on a reflection method across $\partial \Omega$ presented in \cite{S06}. This is the same method as in the proof of \cite[Theorem 3.2]{BMRa23}, to show the regularity up to the boundary in the elliptic case. In the next paragraph, we describe the method and define the variables $(\hat{u}, \hat{v})$ which corresponds to the extension of $(\overline{u},\overline{v})$. 

\smallskip

There exists a $\delta >0$ such that the projection $\pi_\Omega$ on $\partial \Omega$ is well defined and smooth (of class $C^{1,1}$) on the tubular neighbourhood $\{d(x,\partial\Omega) < \delta\}$.
%Using a covering argument, we obtain a positive number $\delta > 0$ such that the projection $\pi_\Omega$ is well defined and of class $C^{1,1}(\Omega)$. 
We define 

\[\left\{\begin{array}{l}
U := \{x\in \mathbb{R}^d, d(x,\partial \Omega) < \delta \},\\
U^{in} := \Omega \cap U,\\
U^{ex} := U \setminus \overline{\Omega}.\\
\end{array}\right.\]  

We define $\sigma$ the geosdesic reflection across $\partial \Omega \cap U$, which satisfies $\sigma_\Omega (x) := 2 \pi_\Omega (x) - x$ for all $x\in U_\delta$. 
The function $\sigma$ is an involutive $C^{1,1}$ diffeomorphism onto its image and satisfies $\sigma(x) = x$ for $x\in \partial \Omega$. 

%Up to reducing $\delta$, we can suppose that 

%\[\sigma(U_{\delta}^{in}) \subset U_\delta^{ex}), \quad  \text{and} \quad \sigma(U_{\delta}^{ex}) \subset U_\delta^{in}.\]

Now consider $\widetilde{\Omega} =U\cup \sigma(U)$,
%\[\widetilde{\Omega} := (U_\delta \cap B_{R/4}(x_0)) \cup \sigma ((U_\delta \cap B_{R/4}(x_0)) \subset B_{R/2}(x_0)\cap U_\delta,\]
so that $\sigma$ satisfies $\sigma (\widetilde{\Omega})= \widetilde{\Omega}$ with $\sigma(\widetilde{\Omega} \cap \Omega) = \widetilde{\Omega} \setminus \overline{\Omega}$ and $\sigma(\widetilde{\Omega} \setminus \overline{\Omega}) = \widetilde{\Omega} \cap \Omega$. 
The equality $\sigma \circ \sigma (x) = x$ yields $D\sigma (x) D\sigma (\sigma (x)) = I_d$, so $[D\sigma (x)]^{-1} = D\sigma (\sigma (x))$ for all $x\in \widetilde{\Omega}$.   
On the other hand, for $x\in \partial\Omega$, the matrix $D\sigma (x)$ is the reflection across the hyperplane $T_x (\partial \Omega)$, so $[D\sigma (x)]^T =  D\sigma (x) = D\sigma (x)^{-1}$ for $x\in \partial\Omega \cap \widetilde{\Omega}$. 

\smallskip

For all $x\in \overline{\Omega}$, let 

\[j(x) := \left\{ \begin{array}{ll}
1 & \text{ if } x\in \widetilde{\Omega} \cap \Omega \\
|\det D \sigma (x) | & \text{ if } x\in \widetilde{\Omega} \setminus \Omega,\\
\end{array} \right.\]
and 
\[M(x) = \left\{ \begin{array}{ll}
1 & \text{ if } x\in \widetilde{\Omega} \cap \Omega \\
j(x) [D\sigma (\sigma (x))]^T D\sigma (\sigma (x)) & \text{ if } x\in \widetilde{\Omega} \setminus \Omega,\\
\end{array}\right.\]
Note that $j$ and $M$ are functions of class $\mathcal{C}^{1,1}$ on $\widetilde{\Omega}$. Moreover, provided $\delta >0$ is small enough, there exists $\lambda >0$ such that 

\begin{equation}\label{hypA}
M(x) \geqslant \lambda I_d \quad \quad \text{for all } x\in \widetilde{\Omega}.
\end{equation}
This extension in space being done, we define the extended pairs $(\hat{u}_\varepsilon , \hat{v}_\varepsilon) $ and $(\widetilde{u}_\varepsilon , \widetilde{v}_\varepsilon )$ on $(0,\infty)\times \widetilde{\Omega}$ : 

\begin{equation}\label{def : u chapeau}
(\hat{u}(t,x),\hat{v}(t,x))  := \left\{ \begin{array}{ll}
(\overline{u}(t,x), \overline{v}(t,x))  & \text{ if } x\in \widetilde{\Omega} \cap \Omega \\
(\overline{u} (t,\sigma (x)), \overline{v} (t,\sigma(x)) & \text{ if } x\in \widetilde{\Omega} \setminus \Omega,\\
\end{array} \right.
\end{equation}
%\quad \text{and} \quad  = \left\{ \begin{array}{ll}
%v_\varepsilon (t,x) & \text{ if } x\in \widetilde{\Omega} \cap \Omega \\
%v_\varepsilon (t,\sigma(x)) & \text{ if } x\in \widetilde{\Omega} \setminus \Omega,\\
%\end{array}\right..
and 
\begin{equation}\label{def : u tilde}
(\widetilde{u}(t,x), \widetilde{v}(t,x)) := \left\{ \begin{array}{ll}
(\overline{u} (t,x), \overline{v}(t,x)) & \text{ if } x\in \widetilde{\Omega} \cap \Omega \\
(2 \overline{g}(t,\pi(x))- \overline{u} (t,\sigma (x)), 2-\overline{v}(t,\sigma(x))) & \text{ if } x\in \widetilde{\Omega} \setminus \Omega,\\
\end{array} \right.
\end{equation}
The pairs $(\hat{u}, \hat{v})$ and $(\widetilde{u}, \widetilde{v})$ belong to $\mathcal{F}$. 
Finally, we set $\widetilde{g}(t,x) := \overline{g}(t, \pi(x)) \in \mathcal{C}^{1,1}((0,T) \times \widetilde{\Omega})$. 

\smallskip

The following lemma establishes the equations satisfied by the pairs $(\hat{u} , \hat{v})$ and $(\widetilde{u} , \widetilde{v} )$. 

\medskip

\begin{lemma}
Recall the definition \eqref{def U0 V0} of $(U_0, V_0)$
The functions $\widetilde{u}, \widetilde{v}, \hat{u}, \hat{v}$ as defined in \eqref{def : u chapeau} and \eqref{def : u tilde} satisfy 

\begin{multline}\label{boundary : equ}
j\, \partial_t \widetilde{u} - div((\eta_\varepsilon + \hat{v}^2_\varepsilon ) M\nabla \widetilde{u} )\\
 = 2j\, \partial_t \widetilde{g} 1_{\widetilde{\Omega}\setminus \overline{\Omega}} - 2 div(1_{\widetilde{\Omega}\setminus \overline{\Omega}} (\eta_\varepsilon + \hat{v}^2) M \nabla \widetilde{g})+ \frac{j}{2}(U_0 1_{\widetilde{\Omega}\cap \Omega} - U_0\circ \sigma  1_{\widetilde{\Omega}\setminus \Omega}) \quad \text{in } \mathcal{D}'((0,\infty) \times \widetilde{\Omega}),
\end{multline}
%\text{in } \mathcal{D}'((0,+\infty)\times \widetilde{\Omega})

\begin{multline}\label{boundary : eqv}
j\, \partial_t \widetilde{v} - \text{div}(M\nabla \widetilde{v})\\
 =(1_{\widetilde{\Omega} \cap \Omega} - 1_{\widetilde{\Omega}\setminus \overline{\Omega}}) \left( \frac{j}{4\varepsilon} (1-\hat{v}) - (M \widetilde{v} \nabla \widetilde{u} \cdot \nabla \widetilde{u})\right) + \frac{j}{2}(V_0 1_{\widetilde{\Omega}\cap \Omega} - V_0\circ \sigma  1_{\widetilde{\Omega}\setminus \Omega}) \quad \text{in } \mathcal{D}'((0,\infty) \times \widetilde{\Omega}).
\end{multline}
%\text{in } \mathcal{D}'((0,+\infty)\times \widetilde{\Omega}).
\end{lemma}

\begin{proof}
Let $\varphi \in \mathcal{C}_c^\infty ((0,\infty)\times \widetilde{\Omega})$. We decompose $\varphi = \varphi ^s +\varphi^a$ as the sum of its symmetric and anti-symmetric part, with 

\[\varphi^s(x) = \frac{1}{2}( \varphi(x) + \varphi (\sigma (x))), \quad \text{ and } \quad \varphi^a(x) = \frac{1}{2}( \varphi(x) - \varphi (\sigma (x))).\]
The functions $\varphi^s$ and $\varphi ^a$ belong to $\mathcal{C}^{1,1}((0,\infty)\times \Omega))$ and satisfy $\varphi ^s(t,\sigma(x)) = \varphi ^s(t,x)$ and $\varphi ^a (t, \sigma(x)) = - \varphi ^a(t,x)$ for all $(t,x) \in (0,\infty) \times \widetilde{\Omega}$. 

\medskip

\underline{Proof of \eqref{boundary : equ} :} In the expression   

\[\int_0^\infty \int_{\widetilde{\Omega}} j \partial_t \widetilde{u} \varphi + \int_{\widetilde{\Omega}} (\eta_\varepsilon + \hat{v}^2 )M \nabla \widetilde{u} \cdot \nabla \varphi,\]
we consider separately the terms with $\varphi^s$ and the terms with $\varphi^a$ : 
\begin{multline}\label{boundary equ 1}
\int_0^\infty\int_{\widetilde{\Omega}\setminus \overline{\Omega}} j \partial_t \widetilde{u}\: \varphi^s  + \int_0^\infty \int_{\widetilde{\Omega}\setminus \overline{\Omega}} (\eta_\varepsilon + \hat{v}^2) M\nabla \widetilde{u} \cdot \nabla \varphi^s \\
= 2 \int_0^\infty \int_{\widetilde{\Omega}\setminus \overline{\Omega}} \left[j\partial_t \widetilde{g} \varphi^s +(\eta_\varepsilon + \hat{v}^2)M\nabla \widetilde{g} \cdot \nabla \varphi ^s\right]\\
- \int_{\widetilde{\Omega}\setminus \overline{\Omega}} \left[j\partial_t (\overline{u}\, \circ \, \sigma) \varphi^s + (\eta_\varepsilon + \hat{v}^2)M\nabla (\overline{u}\, \circ\, \sigma) \cdot \nabla \varphi ^s\right].
\end{multline}
Using $\varphi^s \circ \sigma = \varphi^s$ and changing the variable : 

\begin{multline}\label{boundary equ 2}
\int_0^\infty\int_{\widetilde{\Omega}\setminus \overline{\Omega}}\left[ j(\partial_t\overline{u}\, \circ \,\sigma)(\varphi^s \circ\, \sigma) + (\eta_\varepsilon + \hat{v}^2)M\nabla (\overline{u}\, \circ \,\sigma) \cdot \nabla (\varphi ^s \circ \,\sigma)\right]\\
= \int_{\widetilde{\Omega}\cap \Omega} \left[\partial_t \overline{u} \varphi^s + (\eta_\varepsilon + \overline{v}^2) \nabla \overline{u} \cdot \nabla \varphi^s\right].
\end{multline}
From \eqref{boundary equ 1} and \eqref{boundary equ 2}, we deduce 

\begin{equation}\label{boundary equ 3}
\int_0^\infty \int_{\widetilde{\Omega}}\left[ j \partial_t \widetilde{u} \,\varphi^s + (\eta_\varepsilon + \hat{v}^2) M \nabla \widetilde{u} \cdot \nabla \varphi^s\right] =  2\int_0^\infty \int_{\widetilde{\Omega}\setminus \overline{\Omega}}\left[ j\partial_t \widetilde{g} \,\varphi^s +(\eta_\varepsilon + \hat{v}^2)M\nabla \widetilde{g} \cdot \nabla \varphi ^s\right].
\end{equation}
We now focus on the part with $\varphi^a$. Notice that $\varphi^a =0$ on $(0,\infty) \times \partial\Omega$, so that  

\begin{multline}\label{boundary equ 4}
\int_0^\infty\int_{\widetilde{\Omega}\setminus \overline{\Omega}} j \partial_t \widetilde{u}\, \varphi^a  + \int_0^\infty\int_{\widetilde{\Omega}\setminus \overline{\Omega}} (\eta_\varepsilon + \hat{v}^2 ) M\nabla \widetilde{u} \cdot \nabla \varphi^a \\
= 2\int_0^\infty \int_{\widetilde{\Omega}\setminus \overline{\Omega}} \left[ j\partial_t \widetilde{g} \varphi^a +(\eta_\varepsilon + \hat{v}^2)M\nabla \widetilde{g} \cdot \nabla \varphi ^a \right]\\
 - \int_0^\infty\int_{\widetilde{\Omega}\setminus \overline{\Omega}} \left[ j\partial_t (u\, \circ \,\sigma) \varphi^a + (\eta_\varepsilon + \hat{v}^2)M\nabla (u\, \circ \,\sigma) \cdot \nabla \varphi ^a \right].
\end{multline}
Using that $\varphi^a \circ \sigma = -\varphi^a$ and changing the variable : 

\begin{align*}
&\int_0^\infty\int_{\widetilde{\Omega}\setminus \overline{\Omega}} \left[j\partial_t (u\, \circ \,\sigma) \varphi^a + (\eta_\varepsilon + \hat{v}^2)M\nabla (u\, \circ \,\sigma) \cdot \nabla \varphi ^a\right]\\
&= -\int_0^\infty\int_{\widetilde{\Omega}\setminus \overline{\Omega}}\left[ j\partial_t (u\, \circ \,\sigma) (\varphi^a \circ \sigma) + (\eta_\varepsilon + \hat{v}^2)M\nabla (u\, \circ \,\sigma) \cdot \nabla (\varphi ^a \circ \sigma)\right]\\
&= -\int_0^\infty\int_{\widetilde{\Omega}\cap \Omega}\left[ \partial_t \overline{u} \varphi^a + (\eta_\varepsilon + \overline{v}^2) \nabla \overline{u} \cdot \nabla \varphi^a\right] \\
&=-\int_0^\infty\int_{\widetilde{\Omega}\cap \Omega} U_0 \varphi^a,\\
\end{align*}
where we used \eqref{extension uv eq} in the last equality. 
%&=- \int_{\widetilde{\Omega}\setminus \overline{\Omega}}\left[ j(x)\partial_t u (\sigma) \varphi^a (\sigma) + (\eta_\varepsilon + \hat{v}^2)\nabla u (\sigma) \cdot \nabla \varphi ^a (\sigma) \, j(x)\right]\\
We deduce from \eqref{boundary equ 4} that 

\begin{align}\label{boundary equ 5}
&\int_0^\infty\int_{\widetilde{\Omega}} j\, \partial_t \widetilde{u}\, \varphi^a + (\eta_\varepsilon + \hat{v}^2) M \nabla \tilde{u} \cdot \nabla \varphi^a \nonumber\\
&= \int_0^\infty\int_{\widetilde{\Omega}\setminus \overline{\Omega}} j\, \partial_t \widetilde{u} \varphi^a  + \int_0^\infty\int_{\widetilde{\Omega}\setminus \overline{\Omega}} (\eta_\varepsilon + \hat{v}^2_\varepsilon ) M\nabla \widetilde{u} \cdot \nabla \varphi^a \nonumber\\
&\quad +\int_0^\infty\int_{\widetilde{\Omega}\cap \Omega} \left[\partial_t \overline{u} \varphi^a + (\eta_\varepsilon + \overline{v}^2) \nabla u \cdot \nabla \varphi^a \right] \nonumber\\
&= 2\int_0^\infty\int_{\widetilde{\Omega}\setminus \overline{\Omega}}\left[ j\,\partial_t \widetilde{g} \,\varphi^a +(\eta_\varepsilon + \hat{v}^2)M\nabla \widetilde{g} \cdot \nabla \varphi ^a\right] + \int_0^\infty\int_{\widetilde{\Omega} \cap \Omega} U_0 \varphi^a + 0. \nonumber\\
\end{align}
We observe that
\begin{multline}\label{eq U0nouveau}
\int_0^\infty\int_{\widetilde{\Omega}\cap \Omega} U_0 \varphi^a = \frac{1}{2}\int_0^\infty \int_{\widetilde{\Omega}\cap \Omega} U_0(\varphi - \varphi\circ \sigma)\\
= \frac{1}{2}\int_0^\infty\int_{\widetilde{\Omega}\cap \Omega} U_0\varphi - \frac{1}{2}\int_0^\infty\int_{\widetilde{\Omega}\setminus \Omega} U_0\circ \sigma \varphi j = \int_0^\infty \int_{\widetilde{\Omega}} \frac{j}{2}(U_0 1_{\widetilde{\Omega}\cap \Omega} - U_0\circ \sigma  1_{\widetilde{\Omega}\setminus \Omega})\varphi.
\end{multline}

Combining \eqref{boundary equ 3},  \eqref{boundary equ 5} and \eqref{eq U0nouveau} , we obtain \eqref{boundary : equ}. 
The same computation holds for \eqref{boundary : eqv}. 
\end{proof}

In the sequel, we rewrite \eqref{boundary : equ} as 

\begin{multline}\label{boun equ vrai}
\partial_t \widetilde{u} - \text{div}\left(\frac{(\eta_\varepsilon + \widetilde{v}^2)}{j}M \nabla \widetilde{u}\right) - (\eta_\varepsilon + \widetilde{v}^2) M \frac{\nabla j}{j^2}\cdot \nabla \widetilde{u} \\
= 2 \partial_t \widetilde{g}\,1_{\widetilde{\Omega}\setminus \overline{\Omega}}\, - \frac{2}{j} \text{div}(1_{\widetilde{\Omega}\setminus \overline{\Omega}} (\eta_\varepsilon + \widetilde{v}^2) M \nabla \widetilde{g} ) + \widetilde{U}_0
\end{multline}
and \eqref{boundary : eqv} as 

\begin{equation}\label{boun eqv vrai}
\partial_t \widetilde{v} - \text{div}\left(H \nabla \widetilde{v}\right) - J\cdot \nabla \widetilde{v} = \mu
\end{equation}
where, keeping in mind the definition \eqref{def U0 V0} of $(U_0,V_0)$, we set

\begin{equation}\label{def U0tilde}
\widetilde{U}_0 = \frac{1}{2}U_0 1_{\widetilde{\Omega}\cap \Omega} - \frac{1}{2} U_0\circ \sigma 1_{\widetilde{\Omega}\setminus \overline{\Omega}}, \quad \widetilde{V}_0 = \frac{1}{2}V_0 1_{\widetilde{\Omega}\cap \Omega} - \frac{1}{2} V_0\circ \sigma 1_{\widetilde{\Omega}\setminus \overline{\Omega}}
\end{equation}

\begin{equation}\label{def mu}
H =  \frac{M}{j}, \quad J =  M\frac{\nabla j}{j^2} , \quad \mu = (1_{\widetilde{\Omega} \cap \Omega} - 1_{\widetilde{\Omega} \setminus \overline{\Omega}}) \left(\frac{1-\widetilde{v}}{4\varepsilon} - \frac{1}{j}(M\widetilde{v}\nabla \widetilde{u} \cdot \widetilde{u})\right) + \widetilde{V}_0.
\end{equation}

As $\partial\Omega$ is of class $\mathcal{C}^{2,1}$, for any $x_0\in \widetilde{\Omega}$ and $R >0$ such that $B_{2R}(x_0) \subset \widetilde{\Omega}$, the functions $H$ and $J$ are in $\mathcal{C}^0(B_{2R}(x_0))$. In particular, we can introduce positive constants $c_j, C_j$ and $\lambda$ so that $C_j >j(x) \geqslant c_j > 0$ and $H(x)\geqslant \lambda I_d >0$ for all $x\in B_{2R}(x_0)$. 
 
\subsubsection{Proof of the regularity}

The aim of this subsection is to prove the following proposition. 

\begin{proposition}\label{prop : uv holder}
Let $(t_0,x_0) \in (0,+\infty)\times \widetilde{\Omega}$ and $Q_{2R}(t_0,x_0) \subset\subset (0, +\infty)\times \widetilde{\Omega}$. 
There exists $\alpha \in (0,1)$ and a constant $C$ depending only on $d(x_0, \partial \widetilde{\Omega}), \eta_\varepsilon, \alpha$ such that $(\widetilde{u}, \widetilde{v})\in \mathcal{C}^{\alpha/2,\alpha}_{par}(Q_R(t_0,x_0))$. 
\end{proposition}
From this result and a bootstrap argument, we can prove \autoref{Theorem 2}, as follows :

\begin{proof}[Proof of \autoref{Theorem 2}]
Fix any $T >0$. Let us start with the case $k=1$. 
%We prove by induction on $k$ that $(u,v)\in C^{k,\alpha}([0,T]\times \overline{\Omega})$ for any $k\geqslant 0$, which is sufficient using a recovery argument. 
Gathering \autoref{inner regularity} and \autoref{prop : uv holder}, we deduce that there exists a real number $\alpha \in (0,1)$ such that $u,v \in \mathcal{C}_{par}^{\alpha/2, \alpha}([0,T] \times \overline{\Omega})$. 

\smallskip

Considering the equation for $u$ \eqref{eq : eq in u}, as the conditions (i), (ii) and (iii) for $k=1$ of \autoref{Theorem 2} and the compatibility condition (iv) for $\ell =1$ are fulfilled, we deduce from the fact that $v\in \mathcal{C}_{loc}^{\alpha/2, \alpha}([0,T]\times \overline{\Omega})$ and from \cite[Section IV, Theorem 5.2]{LSU68} that $u\in \mathcal{C}^{1+\alpha/2, 2+\alpha}([0,T]\times \overline{\Omega})$. 
It implies that $\dfrac{1-v}{4\varepsilon} + v|\nabla u|^2\in \mathcal{C}^{\alpha/2, \alpha}([0,T]\times \overline{\Omega})$. 
Considering the equation for $v$ \eqref{eq : eq in v} and again as the conditions (i), (ii), (iii) and (iv) of \autoref{Theorem 2} are satisfied, we can apply \cite[Section IV, Theorem 5.2]{LSU68} and deduce that $v\in \mathcal{C}^{1+\alpha/2, 2+\alpha}([0,T]\times \overline{\Omega})$. 

\smallskip

In particular, $v\in \mathcal{C}^{\beta/2, \beta}([0,T]\times \overline{\Omega})$. Back to the equation \eqref{eq : eq in u} in $u$, we deduce the same way as above that $u\in \mathcal{C}^{1+\beta/2, 2+\beta}([0,T] \times \overline{\Omega})$. Back to the equation \eqref{eq : eq in v} in $v$, we deduce similarly as above that $v\in \mathcal{C}^{1+\beta/2, 2+\beta}([0,T]\times \overline{\Omega})$. 

\medskip

For $k\geqslant 2$, one can iterate the above argument : let $m = \lfloor k/2\rfloor$ and suppose that $(u,v) \in \mathcal{C}^{m -1 + \beta/2, 2m-2 +\beta}([0,T]\times \overline{\Omega})$. Considering the equation for $u$ \eqref{eq : eq in u}, as the conditions (i), (ii) and (iii) for $k=m$ of \autoref{Theorem 2} and the compatibility condition (iv) for $\ell \leqslant m$ are fulfilled, we deduce from the fact that $v\in \mathcal{C}_{loc}^{m -1 + \beta/2, 2m-2 +\beta}([0,T]\times \overline{\Omega})$ and from \cite[Section IV, Theorem 5.2]{LSU68} that $u\in \mathcal{C}^{m+\beta/2, 2m+\beta}([0,T]\times \overline{\Omega})$.
It implies that $\dfrac{1-v}{4\varepsilon} + v|\nabla u|^2\in \mathcal{C}^{m-1+\beta/2, 2m-2+\beta}([0,T]\times \overline{\Omega})$. 
Considering the equation for $v$ \eqref{eq : eq in v} and again as the conditions (i), (ii) and (iii) of \autoref{Theorem 2} are satisfied, we can apply \cite[Section IV, Theorem 5.2]{LSU68} and deduce that $v\in \mathcal{C}^{m+\beta/2, 2m+\beta}([0,T]\times \overline{\Omega})$. 
%Suppose that  $u,v\in \mathcal{C}_{par}^{k/2, \alpha/2, k +\alpha}([0,T]\times \overline{\Omega})$ for an integer $k\geqslant 0$. 
%Considering the equation for $u$ in \eqref{système de départ}, we deduce from the fact that $v\in \mathcal{C}_{loc}^{\alpha/2, \alpha}([0,T]\times \overline{\Omega})$ and from \cite[Section IV, Theorem 5.1]{LSU68} that the function $u\in \mathcal{C}_{par}^{k/2 +1 +\alpha/2,k+ 2+ \alpha}([0,T] \times \overline{\Omega})$. In particular, $|\nabla u|^2 \in \mathcal{C}_{par}^{k/2 +1/2 +\alpha/2,k+ 1+ \alpha}([0,T] \times \overline{\Omega})$.
%
%
%\smallskip
%
%Going back to the equation in $v$ in \eqref{système de départ}, considering that $\dfrac{1-v}{4\varepsilon} - v|\nabla u|^2$ belongs to $\mathcal{C}_{par}^{k+\alpha, k/2 + \alpha/2}([0,T]\times \overline{\Omega})$, we deduce, again invoking \cite[Section IV, Theorem 5.1]{LSU68}, that the function $v$ belongs to $\mathcal{C}_{par}^{k+ 2+ \alpha, k/2 +1 +\alpha/2}([0,T] \times \overline{\Omega})$.
\end{proof}

%The proof that $\widetilde{u}$ is in $C^{0,\alpha}(Q_R(t_0, x_0))$ is done in the following lemma 

As for the inner regularity, the starting point is to be able to establish the existence of $\alpha$ such that $|\nabla \widetilde{u}|^2\in \mathcal{L}^{1,d+2\alpha} (Q_R(t_0,x_0))$, which is performed in the following lemma. 

\begin{lemma}\label{lem : tilu morrey}
There exists a real number $\alpha \in (0,1)$ such that $\widetilde{u}\in \mathcal{C}_{par}^{\alpha/2,\alpha} (Q_R(t_0,x_0))$. 
Moreover, one has $|\nabla \widetilde{u}|^2 \in \mathcal{L}^{1,d+2\alpha}(Q_R(t_0,x_0))$ and $\mu \in \mathcal{L}^{1, d+2\alpha}(Q_R(t_0,x_0))$. 
\end{lemma} 

\begin{proof}
According to \eqref{hypA}, $\exists \lambda, C >0$ such that

\[\frac{(\eta_\varepsilon + \widetilde{v}^2)}{j}M \geqslant \frac{\eta_\varepsilon}{C}\lambda I_d > 0, \quad \left|(\eta_\varepsilon + \widetilde{v}^2) M \frac{\nabla j}{j^2}\right| \leqslant \Lambda,\]

\begin{equation}\label{moche 1}
\left|2 \partial_t \widetilde{g}1_{\widetilde{\Omega}\setminus \overline{\Omega}}\right|, |\widetilde{U}_0| \leqslant C,\quad  \frac{2}{j} \in \mathcal{C}^{1,1}(\widetilde{\Omega}), \quad \left|1_{\widetilde{\Omega}\setminus \overline{\Omega}} (\eta_\varepsilon + \widetilde{v}^2) M \nabla \widetilde{g}\right| \leqslant C.
\end{equation}

Let $T >0$.  
Considering the equation \eqref{eq : eq in u} satisfied by $u$ , according to \cite[Theorem IV.1.2]{DB93}, there exists a real number $\alpha \in (0,1)$ such that $\overline{u} \in \mathcal{C}_{par}^{\alpha/2,\alpha}([0,T]\times \overline{\Omega})$. Using that $\sigma \in \mathcal{C}^{1,1}(\widetilde{\Omega})$ and the expression of $\widetilde{u}$, we deduce that $\widetilde{u}\in \mathcal{C}_{par}^{\alpha/2,\alpha}(Q_{R}(t_0,x_0))$. 
 
Considering the condition \eqref{moche 1}, the right hand-side of \eqref{boun equ vrai} 
satisfies the hypothesis of \autoref{prop:grad_spat_Morrey} with 

\[A= \dfrac{\eta_\varepsilon + \widetilde{v}^2}{j}M, \quad B=- (\eta_\varepsilon + \widetilde{v}^2)A \dfrac{\nabla j}{j^2}, \quad h= 2 \partial_t \widetilde{g} 1_{\widetilde{\Omega}\setminus \overline{\Omega}} + \widetilde{U}_0,\]

\[h_1 = \frac{2}{j}, \quad h_2 = 1_{\widetilde{\Omega} \setminus \overline{\Omega}} (\eta_\varepsilon + \hat{v}^2)M\nabla \widetilde{g}.\]
We deduce that $\nabla \widetilde{u} \in \mathcal{L}^{2,d+2\alpha}(Q_R(t_0,x_0))$. 

\smallskip

We conclude that $\mu \in \mathcal{L}^{1,d+2\alpha}(Q_R(t_0,x_0))$. Observe that $\frac{1-\widetilde{v}}{4\varepsilon}$ is bounded thanks to \autoref{Maximum principle vrai}, that $\widetilde{V}_0$ is bounded because $j\geqslant c_j$ and that  $(A\hat{v}\nabla \hat{u} \cdot \hat{u})/j \in \mathcal{L}^{1,d+2\alpha}(Q_{R}(t_0,x_0))$. We conclude that $\mu \in \mathcal{L}^{1,d+2\alpha}(Q_{R}(t_0,x_0))$.
\end{proof}
The main difficulty of \autoref{prop : uv holder} is to prove that $\widetilde{v}\in \mathcal{C}_{par}^{\alpha/2,\alpha}(Q_R(t_0, x_0))$.
The proof of \autoref{prop : uv holder} follows the same method as the proof of \autoref{inner regularity}, although it is slightly more technical due to the presence of non constant terms in \eqref{boun eqv vrai}.
As in the proof of the inner regularity, it is sufficient to prove a uniform decay estimate of the Campanato norm, namely : 
%The main claim to prove the regularity of $\widetilde{v}$ is the foolowing lemma : 
\begin{lemma}\label{lem : campanato v}
Let $\alpha\in (0,1)$ be given by \autoref{lem : tilu morrey}. There exists a constant $C_1 > 0$ such that for any $(t_1, x_1) \in Q_R(t_0,x_0)$ and any $\rho < r \leqslant R$ with $Q_r(t_1, x_1) \subset Q_R (t_0,x_0)$, 

\[\iint_{Q_\rho(t_1,x_1)} |\widetilde{v} - \{\widetilde{v}\}_{\rho,t_1,x_1}|^2 \leqslant C_1\left(\frac{\rho}{r}\right)^{d+4} \iint_{Q_{r}(t_1,x_1)} |\widetilde{v} - \{\widetilde{v}\}_{r, t_1, x_1}|^2 + C_1r^{d+2 +2\alpha}.\]
\end{lemma}

We briefly recall how this estimate leads to the proof of \autoref{prop : uv holder}. 

\begin{proof}[Proof of \autoref{prop : uv holder}]
Suppose that \autoref{lem : campanato v} is satisfied. Fix $(t_1,x_1) \in Q_R(t_0,x_0)$, so that $Q_R(t_1, x_1) \subset Q_{2R}(t_0,x_0)$. 

Using the iteration lemma \cite[Lemma 4.6]{L96} and the fact that $0\leqslant \widetilde{v} \leqslant 1$ allows to conclude that for any $Q_r(t_1, x_1) \subset Q_R(t_1,x_1)$, 

\begin{multline*}
\iint_{Q_r(t_1,x_1)} |\widetilde{v} - \{\widetilde{v}\}_{r,t_1,x_1}|^2\leqslant C_1 r^{d+2+2\alpha} \left(\frac{1}{R^{d+2+2\alpha}}\iint_{Q_{R}(t_1,x_1)} |\widetilde{v} - \{\widetilde{v}\}_{2r,t_1,x_1}|^2 + 1\right)\\
\leqslant C_1 r^{d+2+2\alpha} \left(\frac{1}{R^{d+2+2\alpha}}4R^{d+2} + 1\right).
\end{multline*}
Using the Cauchy-Schwarz inequality, we deduce that 

\begin{multline*}
\iint_{Q_r (t_1,x_1)} |\widetilde{v} - \{v\}_{r,t_1,x_1}| \leqslant |Q_r|^{1/2} \left(\iint_{Q_r(t_1, x_1)} |\widetilde{v} - \{\widetilde{v}\}_{r,t_1,x_1}|^2\right)^{1/2}= C(\alpha, d) r^{d+2+\alpha}.
\end{multline*}
Hence, according to \autoref{lem : Lieberman}, we deduce that $[\widetilde{v}]_{\alpha, Q_R(t_0, x_0)} \leqslant C(\alpha ,d)$. Recalling that $\widetilde{v}$ is bounded, we deduce that $\widetilde{v} \in C_{par}^{\alpha/2, \alpha} (Q_R(t_0, x_0))$. 
\end{proof}

%There exists a constant $C_1$ such that for any parabolic cylinder $Q_r(t_1,x_1) \subset Q_{R/2}(t_0,x_0)$, 
%
%\[\iint_{Q_r(t_1, x_1)} |\widetilde{v} - \{\widetilde{v}\}_r|^2 \leqslant C_1 r^{d+2+2\alpha}.\] 
%\end{lemma}

The rest of the subsection is devoted to the proof of \autoref{lem : campanato v}. Let us briefly explain the outline of the proof. 
We fix a $(t_1, x_1) \subset Q_R(t_0, x_0)$, so that $Q_R(t_1, x_1) \subset Q_{2R}(t_0,x_0))$ and we fix $r <R$. 
%In order to prove the estimate, we prove the intermediary inequality : there exists a constant $C$ such that for any $\rho < r \leqslant R$, one has 
%\begin{equation}\label{campanato 1}
%\iint_{Q_\rho(t_1,x_1)} |\widetilde{v} - \{\widetilde{v}\}_\rho|^2 \leqslant C\left(\frac{\rho}{r}\right)^{d+4} \iint_{Q_{r}(t_1,x_1)} |\widetilde{v} - \{\widetilde{v}\}_{r, t_1, x_1}|^2 + C|r|^{d+2 +2\alpha}.
%\end{equation}
%In order to prove \eqref{campanato 1}, we fix $r' \in (0,R)$. 
We consider a ``heat''-extension $w$  of $\widetilde{v}$ on $Q_{r}(t_1,x_1)$. The difference $v_2 =\widetilde{v}-w$ satisfies a parabolic PDE  on $Q_{r}(t_1,x_1)$ with homogeneous initial and boundary Dirichlet conditions, so that we can apply $v_2$ as a test function to it and estimate $\iint_{Q_{r}(t_1,x_1)} |\nabla v_2|^2$. 

\smallskip

The source term of the equation in $v_2$ depends on the function $\nabla \widetilde{v}$, so that we need a preliminary estimate of the energy $\iint_{Q_{r}(t_1,x_1)} |\nabla \widetilde{v}|^2$, which we state before getting to the proof of \autoref{lem : campanato v}. 

\begin{lemma}\label{lem : energie tilv}
There exists a constant $C_\star$ depending only on $R$ such that for any $Q_{r}(t_1, x_1) \subset Q_R(t_0, x_0)$, one has 

\[\iint_{Q_{r}(t_1, x_1)} |\nabla \widetilde{v}|^2 \leqslant C_\star r^d.\]
\end{lemma}

\begin{proof}
Let $\zeta\in \mathcal{C}_c^\infty (\mathbb{R}\times \Omega)$ be a test function such that $0\leqslant \zeta \leqslant 1$, $\text{Supp}(\zeta) \subset (t_1-(2r)^2, t_1 +(2r)^2) \times B_{2r}(x_1)$ and $\zeta =1$ on $Q_{r}(t_1, x_1)$, with 

\[|\nabla \zeta| \leqslant \frac{2}{r}, \quad |\partial_t \zeta | \leqslant \frac{C}{r^2}.\]
Applying $\widetilde{v} \zeta^2 \in L_{loc}^\infty((0,+\infty) \times \widetilde{\Omega})\cap H^1((0,+\infty); L^2(\widetilde{\Omega})) \cap L^2((0,+\infty); H_{loc}^1(\widetilde{\Omega}))$ as a test function to \eqref{boun eqv vrai} yields 

\begin{multline}\label{energie tilv 1}
\iint_{Q_{2r}} \zeta^2 \partial_t \widetilde{v} \widetilde{v} + \iint_{Q_{2r}} H\nabla \widetilde{v}\cdot \nabla \widetilde{v} \zeta^2 + 2 \iint_{Q_{2r}} H \nabla \widetilde{v} \cdot \nabla \zeta \, \widetilde{v} \zeta \\
= \iint_{Q_{2r}} \mu \widetilde{v} \zeta^2 - \iint_{Q_{2r}} J\cdot \nabla \widetilde{v} \widetilde{v} \zeta^2.
\end{multline} 
Set 

\[I= \iint_{Q_{2r}} \zeta^2 \partial_t \widetilde{v} \widetilde{v}, \quad II = 2 \iint_{Q_{2r}} H \nabla \widetilde{v} \cdot \nabla \zeta \widetilde{v} \zeta,\]

\[III = \iint_{Q_{2r}} \mu \widetilde{v} \zeta^2, \quad IV = \iint_{Q_{2r}} J\cdot \nabla \widetilde{v} \widetilde{v} \zeta^2.\]
We  estimate each term separately. Using that $0\leqslant \widetilde{v}\leqslant 1$, 

\begin{equation}\label{energie tilv I}
I = \iint_{Q_{2r}} \zeta^2 \partial_t \widetilde{v} \widetilde{v} =\iint_{Q_{2r}} \zeta^2 \partial_t \frac{1}{2}|\widetilde{v}|^2 = \int_{Q_{2r}} |\widetilde{v}|^2 \zeta \partial_t \zeta \leqslant \int_{Q_{2r}} |\partial_t \zeta| \leqslant Cr^{d}.
\end{equation} 
Using Young's inequality,

\begin{multline}\label{energie tilv II}
II = 2 \iint_{Q_{2r}} H \nabla \widetilde{v} \cdot \nabla \zeta \widetilde{v} \zeta \leqslant \delta \iint_{Q_{2r}} |\nabla \widetilde{v}|^2 \zeta^2 + C_\delta \|H\|_{\infty,\widetilde{\Omega}}^2 \iint_{Q_{r}} |\widetilde{v}|^2 |\nabla \zeta|^2\\
\leqslant \delta \iint_{Q_{2r}} |\nabla \widetilde{v}|^2 \zeta^2  + C_\delta \|H\|_{\infty,\widetilde{\Omega}}^2 r^{d}.
\end{multline}
and 

\begin{equation}\label{energie tilv IV}
\iint_{Q_{2r}} J\cdot \nabla \widetilde{v} \widetilde{v} \zeta^2 \leqslant \delta \iint_{Q_{2r}} |\nabla \widetilde{v}|^2 \zeta^2 + C_\delta \|J\|_{\infty,\widetilde{\Omega}}^2 \iint_{Q_{2r}} |\widetilde{v}|^2 \leqslant  \delta \iint_{Q_{2r}} |\nabla \widetilde{v}|^2 \zeta^2  + C_\delta \|J\|_{\infty,\widetilde{\Omega}}^2 r^{d+2}.
\end{equation}
Finally, recalling $\mu\in \mathcal{L}^{1,d+2\alpha}$, 

\begin{equation}\label{energie tilv III}
IV = \iint_{Q_{2r}} \mu \widetilde{v} \zeta^2 \leqslant \iint_{Q_{2r}} |\mu| \leqslant \|\mu\|_{\mathcal{L}^{1,d+2\alpha}} r^{d+2\alpha}.
\end{equation}

Recalling that $H(x) \geqslant \lambda I_d$ for all $x\in \widetilde{\Omega}$, with $\lambda >0$, we can choose $\delta$ small enough such that, combining \eqref{energie tilv I}, \eqref{energie tilv II}, \eqref{energie tilv III} and \eqref{energie tilv IV} into \eqref{energie tilv 1}, 

\[\iint_{Q_r(t_1,x_1)} |\nabla \widetilde{v}|^2 \leqslant \iint_{Q_{2r}} |\nabla \widetilde{v}|^2 \zeta^2 \leqslant Cr^{d} + Cr^{d+2} + Cr^{d+2\alpha} \leqslant C_\star r^d.\]
This concludes the lemma.
\end{proof}

We now turn to the proof of \autoref{lem : campanato v}. 

\begin{proof}[Proof of \autoref{lem : campanato v}]
Fix $r< R$ such that $Q_{r}(t_1, x_1) \subset Q_R (t_0,x_0)$. 
Denote by $H_1 = H(x_1)$ and $J_1 = J(x_1)$. 
Let $w$ be the (unique weak) solution in $H^1((t_1 -(2r)^2) \times B_{2r}(x_1))$ of the following equation 

\[\left\{\begin{array}{rll}
\partial_t w - \text{div}(H_1 \nabla w) + J_1 \nabla w &= 0, &\text{on } Q_{2r}(t_1, x_1)\\
w &= \widetilde{v} & \text{on } \partial_{\sqcup} Q_{2r}(t_1,x_1)\\
\end{array}\right.\]
Notice that $\hat{w} : (t,x) \mapsto w(t, x +J_1t)$ is the solution of $\partial_t \hat{w} - \text{div}(H_1 \nabla \hat{w}) =0$, so $w$ will have the same properties as the classic heat equation (maximum principle and Campanato estimates). In particular, $|w| \leqslant \widetilde{v}\leqslant 1$.  

\smallskip

Let $v_2= \widetilde{v} -w$. The function $v_2$ is a weak solution of the following equation

\begin{equation}\label{camp v eq v2}
\left\{
\begin{array}{l@{\qquad}l}
\begin{split}
\partial_t v_2 - \operatorname{div}(H_1 \nabla v_2) - J_1 \nabla v_2
&= \mu + \operatorname{div}((H-H_1)\nabla \widetilde{v}) \\
&\quad + (J-J_1)\cdot \nabla \widetilde{v},
\end{split}
&
\text{on } Q_{2r}(t_1,x_1), \\[0.5ex]
v_2 =0
&
\text{on } \partial_{\sqcup}Q_{2r}(t_1,x_1).
\end{array}
\right.
\end{equation}

%Indeed, check that 
%
%\begin{align*}
%\partial_t v_2 &= \partial_t \widetilde{v} - \partial_t w \\
%&= \text{div}(H \nabla \widetilde{v}) -J\cdot \nabla\widetilde{v} + \mu - \text{div}(H_1 \nabla w) + J_1 \nabla w\\
%&= \text{div}(H \nabla \widetilde{v}) -J\cdot \nabla\widetilde{v} + \mu - \text{div}(H_1 (\nabla \widetilde{v} - \nabla v_2)) + J_1 (\nabla \widetilde{v} -\nabla v_2)\\
%&= \text{div}(H_1 \nabla v_2) - J_1\cdot \nabla v_2 + \mu + \text{div}((H-H_1) \nabla \widetilde{v}) - (J-J_1) \nabla \widetilde{v}.
%\end{align*}

The rest of the proof is divided into two steps : 

\begin{itemize}
\item[$\bullet$] \textbf{Step 1 :} We prove that 

\[\iint_{Q_{2r}(t_1, x_1)} |\nabla v_2|^2 \leqslant C_2\cdot  2r^{d+2+2\alpha}.\]
This is done by multiplying \eqref{camp v eq v2} by $v_2$ which vanishes at the boundary. 

\item[$\bullet$] \textbf{Step 2 :} We conclude using Poincare's inequality 

\[\iint_{Q_{r}(t_1, x_1)} |v_2 -\{v_2\}_{r, t_1,x_1}|^2 \leqslant C r^2 \iint_{Q_{2r}} |\nabla v_2|^2.\]
\end{itemize}

\medskip

\textbf{Step 1 :} Taking $v_2$ as a test function in \eqref{camp v eq v2}. We obtain that 

\begin{multline}
\iint_{Q_{2r}} \partial_t v_2 v_2 + \iint_{Q_{2r}} H_1\nabla v_2\cdot \nabla v_2 +\iint_{Q_{2r}}  J_1\cdot \nabla v_2 v_2 \\
= \iint_{Q_{2r}} \mu v_2 - \iint_{Q_{2r}} (H-H_1) \nabla \widetilde{v} \cdot \nabla v_2 - \iint_{Q_{2r}} (J-J_1)\cdot \nabla \widetilde{v} v_2.
\end{multline} 
Let us denote this equation by $I +II +III = IV +V +VI$. We estimate each term separately. 

\begin{equation}\label{camp I}
I =\iint_{Q_{2r}} \partial_t v_2 v_2 = \int_{B_{2r}} \frac{1}{2} |v_2 (t_1, x)|^2 \geqslant 0. 
\end{equation}
Recall that $H \geqslant c >0$ on $\widetilde{\Omega}$. In particular, 

\begin{equation}\label{camp II}
II = \iint_{Q_{2r}} H_1 |\nabla v_2|^2 \geqslant c\iint_{Q_{2r}} |\nabla v_2|^2. 
\end{equation}
Using Young's inequality and the fact that $|v_2| = |\widetilde{v} -w| \leqslant 2$, 

\begin{multline}\label{camp III}
III= \iint_{Q_{2r}}  J_1 \nabla v_2 v_2 \leqslant \delta \iint_{Q_{2r}} |\nabla v_2|^2 + C_\delta \|J\|_{\infty,\widetilde{\Omega}} \iint_{Q_{2r}} |v_2|^2\\
 \leqslant \delta \iint_{Q_{2r}} |\nabla v_2|^2 + C_\delta \|J\|_{\infty,\widetilde{\Omega}} (2r)^{d+2}.
\end{multline}

\begin{equation}\label{camp IV}
IV = \iint_{Q_{2r}} \mu v_2 \leqslant 2  \iint_{Q_{2r}} |\mu| \leqslant 2 \|\mu \|_{\mathcal{L}^{1,d+2\alpha}} (2r)^{d+2\alpha}.
\end{equation}
Using Young's inequality, \autoref{lem : energie tilv} and the fact that $H \in \mathcal{C}^{1,1}(\widetilde{\Omega})$, 
\begin{multline}\label{camp V}
V=\iint_{Q_{2r}} (H-H_1) \nabla \widetilde{v} \cdot \nabla v_2 \leqslant \delta \iint_{Q_{2r}} |\nabla v_2|^2 + C_\delta \iint_{Q_{2r}} |H-H_1|^2 |\nabla \widetilde{v}|^2 \\
 \leqslant \delta \iint_{Q_{2r}} |\nabla v_2|^2 + C_\delta \|\nabla H\|_{\infty, \widetilde{\Omega}}^2 r^{d+2}. 
\end{multline}
%\leqslant \delta \iint_{Q_{2r}} |\nabla v_2|^2 + C_\delta \|\nabla H\|_\infty^2 |2r|^2 \iint_{Q_{2r}} |\nabla \widetilde{v}|^2 
Using that $J \in \mathcal{C}^{1,1}(\widetilde{\Omega})$,

\begin{multline}\label{camp VI}
VI =\iint_{Q_{2r}} (J-J_1)\cdot \nabla \widetilde{v} v_2 \leqslant \iint_{Q_{2r}} |v_2|^2 +\iint_{Q_{2r}} |J-J_1|^2 |\nabla \widetilde{v}|^2\\
\leqslant 4(2r)^{d+2} + \|\nabla J\|_{\infty, \widetilde{\Omega}}^2 (2r)^2 \iint_{Q_{2r}} |\nabla \widetilde{v}|^2 \leqslant C r^{d+2}.
\end{multline}
Combining \eqref{camp I}, \eqref{camp II}, \eqref{camp III}, \eqref{camp IV}, \eqref{camp V} and \eqref{camp VI} with $\delta >0$ small enough yields Step 1. 

\medskip

\textbf{Step 2 :} We now conclude. Let $\rho < r$. Observe that equation \eqref{camp v eq v2} satisfies the requirements of \autoref{prop : Poincare}. On the other hand, the function $w$ satisfies the following Campanato estimate : there exists a constant $C_{cam}$ such that 

\[\iint_{Q_\rho} |w- \{w\}_\rho|^2 \leqslant C_{cam}\left(\frac{\rho}{r}\right)^{d+4}\iint_{Q_{r}} |w -\{w\}_{r}|^2.\]

Using twice that $|a+b|^2 \leqslant 2a^2 +2b^2$, twice Poincare's inequality and Step 1, there exists a constant $C_P>0$ such that  

\begin{align*}
\iint_{Q_\rho} |\widetilde{v}- \{\widetilde{v}\}_{\rho,t_1,x_1}|^2&\leqslant 2 \iint_{Q_\rho} |w- \{w\}_{\rho,t_1,x_1}|^2 + 2\iint_{Q_\rho} |v_2- \{v_2\}_{\rho,t_1,x_1}|^2\\
&\leqslant 2 C_{Cam}\left(\frac{\rho}{r}\right)^{d+4}\iint_{Q_{r}} |w -\{w\}_{r,t_1,x_1}|^2 + C_P \rho^2 \iint_{Q_{2\rho}} |\nabla v_2|^2\\
&\leqslant 4C_{Cam}\left(\frac{\rho}{r}\right)^{d+4}\iint_{Q_{r}} |\widetilde{v} -\{\widetilde{v}\}_{r,t_1,x_1}|^2\\
&\quad\quad\quad \quad + 4C_{Cam}\left(\frac{\rho}{r}\right)^{d+4}\iint_{Q_{r}} |v_2 -\{v_2\}_{r}|^2 + C_P r^2 \iint_{Q_{2r}} |\nabla v_2|^2 \\
&\leqslant 4C_{Cam}\left(\frac{\rho}{r}\right)^{d+4}\iint_{Q_{r}} |\widetilde{v} -\{\widetilde{v}\}_{r,t_1,x_1}|^2\\
&\quad \quad \quad \quad  + 4C_{Cam}\cdot 1 \cdot C_P r^2 \iint_{Q_{2r}} |\nabla v_2|^2 +  C_P r^2 \iint_{Q_{2r}} |\nabla v_2|^2\\
&\leqslant C_1 \left(\frac{\rho}{r}\right)^{d+4}\iint_{Q_r} |\widetilde{v} -\{\widetilde{v}\}_{r,t_1,x_1}|^2  + C_2 r^{d+2+2\alpha}. 
\end{align*} 

This concludes the proof of \autoref{lem : campanato v} and the proof of \autoref{prop : uv holder}.
\end{proof}

\textbf{Acknowledgments.} The author would like to thank Jean-François Babadjian and Rémy Rodiac for their guidance during the writing of this paper and their careful feedbacks on the manuscript.

\bibliographystyle{alpha}
\bibliography{bibliography2}

\end{document}